\documentclass[11pt,reqno, a4paper]{amsart}
\usepackage[numbers,sort&compress]{natbib}
\usepackage{amsfonts,bbm}
\usepackage{amssymb,color}
\usepackage{amssymb}
\usepackage{fancyhdr}
\usepackage[titletoc]{appendix}
\usepackage{enumitem}
\usepackage{amsgen}
\usepackage{amscd}
\usepackage{amsmath}
\usepackage{mathrsfs}
\usepackage{cases}

\usepackage[colorlinks=true,backref]{hyperref}
\hypersetup{urlcolor=blue, citecolor=blue, linkcolor=black}

\usepackage[margin=3cm, a4paper]{geometry}

\newtheorem{thm}{Theorem}[section]

\newtheorem{lem}[thm]{Lemma}
\newtheorem{prop}[thm]{Proposition}
\newtheorem{defn}[thm]{ \bf{Definition}}
\theoremstyle{remark}
\newtheorem{remark}[thm]{Remark}

\newcommand{\EQ}[1]{\begin{align*}\begin{split} #1 \end{split}\end{align*}}
\newcommand{\EQn}[1]{\begin{align}\begin{split} #1 \end{split}\end{align}}

\newcommand{\Del}[1]{}

\def\norm#1{\left\|#1\right\|}

\def\normb#1{\big\|#1\big\|}

\newcommand{\fe}{\mathrm{e}}

\def\wt#1{\widetilde{#1}}
\def\wh#1{\widehat{#1}}

\def\loe{\leqslant}
\def\goe{\geqslant}
\def\lsm{\lesssim}

\newcommand{\N}{{\mathbb N}}

\newcommand{\R}{{\mathbb R}}
\newcommand{\C}{{\mathbb C}}
\newcommand{\Z}{{\mathbb Z}}

\newcommand{\F}{{\mathcal{F}}}

\def\ds{\mathrm{\ d} s}

\newcommand{\re}{{\mathrm{Re}}}
\newcommand{\im}{{\mathrm{Im}}}

\def\ph{\varphi}

\def\De{\Delta}

\def\ga{\gamma}

\newcommand{\I}{\infty}

\def\half#1{\frac{#1}{2}}

\numberwithin{equation}{section}

\begin{document}
\title[Non-relativistic limit for KG]{Non-relativistic limit for the cubic nonlinear Klein-Gordon equations}

\author{Zhen Lei}
\address{(Z. Lei) School of Mathematical Sciences; LMNS and Shanghai Key Laboratory for Contemporary Applied Math- ematics\\
Fudan University\\
Shanghai 200433, China}
\email{zlei@fudan.edu.cn}

%
%\author{}
%\address{ Rutgers University\\
%	Department of Mathematics\\
%	110 Frelinghuysen Rd.\\
%	Piscataway, NJ, 08854, USA\\}
%\email{soffer@math.rutgers.edu}
%\thanks{}

\author{Yifei Wu}
\address{(Y. Wu) Center for Applied Mathematics\\
	Tianjin University\\
	Tianjin 300072, China}
\email{yerfmath@gmail.com}
\thanks{}


\date{}

\begin{abstract}\noindent
We investigate the non-relativistic limit of the Cauchy problem for the defocusing cubic nonlinear Klein-Gordon equations whose initial velocity contains a factor of $c^2$, with $c$ being the light speed. While the classical WKB expansion is applied to approximate these solutions, the modulated profiles can be chosen as solutions to either a Schr\"odinger-wave equation or a Schr\"odinger equation. We show that, as the light speed tends to infinity, the error function is bounded by, (1) in the case of 2D and modulated Schr\"odinger-wave profiles, $Cc^{-2}$ with $C$ being a generic constant uniformly for all time, under $H^2$ initial data; (2) in the case of both 2D and 3D and modulated Schr\"odinger profiles, $c^{-2} +(c^{-2}t)^{\alpha/4}$ multiplied by a generic constant uniformly for all time,  under $H^\alpha$ initial data with $2 \leq \alpha \leq 4$.  We also show the sharpness of the upper bounds in (1) and (2), and 
the required minimal regularity on the initial data in (2).
 One of the main tools is an improvement of the well-known result of Machihara, Nakanishi, and Ozawa in \cite{MaNaOz-KG-Limits} which may be of interest by itself. The proof also relies on \textit{a fantastic complex expansion} of the Klein-Gordon equation, \textit{introducing the leftward wave and exploring its enhanced performance} and a \textit{regularity gain mechanism} through a high-low decomposition.

\end{abstract}

\maketitle

\tableofcontents

\section{Introduction}

Consider the defocusing cubic nonlinear Klein-Gordon equation:
\EQn{\label{KGo}
	\frac{\hbar^2}{2mc^2}\partial_{tt} u - \frac{\hbar^2}{2m}\De u +\frac{mc^2}{2}u= - |u|^2 u,
}
where
$u(t,x):\R\times\R^d\rightarrow \R$ ($d = 1, 2, 3$) is a real-valued wave function describing neutral (pseudo-) scalar particles,  $c$ the speed of light,  $\hbar$ the Planck constant, and $m>0$ the rest mass of particles. When $u$ is complex it describes charged particles.  In quantum field theory, the Klein-Gordon equation is one of the most fundamental equations describing spinless scalar or pseudo-scalar particles, for example, $\pi$- and $K$-mesons. It is relativistically-invariant and can be considered as the relativistic version of Schr\"odinger equation.

A natural and major mathematical problem in relativistic quantum mechanism is to study the non-relativistic limit of solution to \eqref{KGo}. In particular, one hopes to have accurate and detailed descriptions on the convergence of solutions as the light speed goes to infinity, figure out the necessary price in terms of how much minimal regularity of the initial data is needed, and explore the underlying mechanism behind all of this. In this paper, we aim to settle these questions systematically, based on several new insights and earlier important theoretical and numerical works which will be explained below.

\subsection{Set-up and A Formal WKB Expansion}

Without loss of generality, we will set $\hbar=2, m=2$ and $c=\varepsilon^{-1}$ in the rest of the paper. Then the defocusing cubic nonlinear Klein-Gordon equation \eqref{KGo} is reduced to
\EQn{
	\label{eq:kg-eps}
	&\varepsilon^2\partial_{tt} u^\varepsilon - \De u^\varepsilon + \frac{1}{\varepsilon^2} u^\varepsilon= - |u^\varepsilon|^2 u^\varepsilon.
}
The equation \eqref{eq:kg-eps} is imposed with the following initial data
\begin{align}
u^\varepsilon(0,x) = u_0(x),\quad  \partial_tu^\varepsilon(0,x) = \frac{u_1(x)}{\varepsilon^2},
\end{align}
Note that this form of initial data is typical in view of the trivial space-independent planar waves of linear equations.

In the context of the study of non-relativistic limits, see for instance, \cite{BaoCai-KG-Limit-12, MasNak-KG-Limits} and the references therein, one employs the WKB ansatz
 \begin{align}\label{WKB}
 u^\varepsilon=\fe^{\frac{it}{\varepsilon^2}}v^\varepsilon+\fe^{-\frac{it}{\varepsilon^2}}\bar v^\varepsilon+r^\varepsilon.
 \end{align}
This ansatz serves to eliminate the pronounced time oscillations in the wave function. The profile function $v^\varepsilon$ is usually called modulated wave function of $u^\varepsilon$. It is worth mentioning that  in the case of a complex-valued wave function, $v^\varepsilon$ and $\bar v^\varepsilon$ can be replaced by the electron component $v^-$ and the positron component $v^+$, respectively, which are not necessarily required to be conjugated with each other, as demonstrated in \cite{MasNak-KG-Limits}. Then \textit{the main mathematical issue} is to study the real error $u^\varepsilon - \big(\fe^{\frac{it}{\varepsilon^2}}v^\varepsilon+\fe^{-\frac{it}{\varepsilon^2}}\bar v^\varepsilon\big)$, or equivalently, the remainder function $r_\varepsilon$, in the non-relativistic limit $\varepsilon \to 0$, including its estimates in terms of a power of $\varepsilon$, the necessary regularity on the initial data, the validation time interval and the underlying mechanism.

Substituting \eqref{WKB} into \eqref{eq:kg-eps}, we have
 \begin{align}\label{eq:r}
 &\varepsilon^2\partial_{tt} r^\varepsilon - \De r^\varepsilon + \frac{1}{\varepsilon^2} r^\varepsilon +A_\varepsilon(v^\varepsilon,r^\varepsilon)+B_\varepsilon(v^\varepsilon)\notag\\
 &\quad + \fe^{\frac{it}{\varepsilon^2}}\Big[\varepsilon^2\partial_{tt}v^\varepsilon+2i\partial_{t} v^\varepsilon - \De v^\varepsilon+ 3|v^\varepsilon|^2 v^\varepsilon\Big]\notag\\
 &\quad+\fe^{-\frac{it}{\varepsilon^2}}\Big[\varepsilon^2\partial_{tt}\bar v^\varepsilon-2i\partial_{t}\bar v^\varepsilon - \De\bar v^\varepsilon+ 3|v^\varepsilon|^2 \bar v^\varepsilon\Big]=0,
 \end{align}
 where
 \begin{align}
 A_\varepsilon(v^\varepsilon,r^\varepsilon)=\,&3\Big[\fe^{\frac{2it}{\varepsilon^2}}\big(v^\varepsilon\big)^2+\fe^{-\frac{2it}{\varepsilon^2}}\big(\bar v^\varepsilon\big)^2\Big]r^\varepsilon
 + 3\Big[\fe^{\frac{it}{\varepsilon^2}}v^\varepsilon+\fe^{-\frac{it}{\varepsilon^2}}\bar v^\varepsilon\Big](r^\varepsilon)^2+(r^\varepsilon)^3;\notag\\
 B_\varepsilon(v^\varepsilon)=\,&\fe^{\frac{3it}{\varepsilon^2}}(v^\varepsilon)^3+\fe^{-\frac{3it}{\varepsilon^2}}\big(\bar v^\varepsilon\big)^3.\label{def:B-var}
 \end{align}
 Note that we have the following relations on the initial data
 \EQn{
	\label{eq:initial-data}
	\left\{ \aligned
	& 2\re{ [v^\varepsilon(0)]}+r^\varepsilon(0)=u_0, \\
	&- \frac{2\im{ [v^\varepsilon(0)]}}{\varepsilon^2}+2\re\big[\partial_tv^\varepsilon(0)\big]+\partial_tr^\varepsilon(0)=\frac{u_1}{\varepsilon^2}.
	\endaligned
	\right.
}

There are two different kinds of quantum equation which governs the dynamics of the modulated profile $v^\varepsilon$. In case I, one formally sets $v^\varepsilon = {\rm Modulate}_w(u^\varepsilon)$ to be the solution of the following Cauchy problem for the defocusing cubic nonlinear {\it  Schr\"odinger-wave equation}:
\EQn{
	\label{eq:nls-wave}
	\left\{ \aligned
	&\varepsilon^2\partial_{tt}v^\varepsilon+2i\partial_{t} v^\varepsilon - \De v^\varepsilon= - 3|v^\varepsilon|^2 v^\varepsilon, \\
	& v^\varepsilon(0,x) = v_0(x)\triangleq \frac12(u_0-iu_1),\quad \partial_tv^\varepsilon(0,x)=v_1(x),
	\endaligned
	\right.
}
with any $v_1$ independence of $\varepsilon$.
In this case, the initial data for  $r^\varepsilon$ is
\begin{align}
r^\varepsilon(0)=0,\quad \partial_t r^\varepsilon(0)=-2\re(v_1),
\end{align}

In case II, one formally sets $v^\varepsilon = v = {\rm Modulate}_s(u^\varepsilon)$ to be the solution of the following Cauchy problem for the defocusing cubic nonlinear Schr${\rm \ddot{o}}$dinger equation \eqref{eq:nls}:
\EQn{
	\label{eq:nls}
	\left\{ \aligned
	&2i\partial_{t} v - \De v= - 3 |v|^2 v, \\
	& v(0,x) = v_0(x),
	\endaligned
	\right.
}
The initial data for $r^\varepsilon$ is given by:
\begin{align}
r^\varepsilon(0)=0,\quad \partial_t r^\varepsilon(0)=-2\mathrm{Re}\big[\partial_tv(0)\big].
\end{align}
Expansion \eqref{WKB} in this case is usually called {\it Modulated Fourier expansion}.

\subsection{Main Results}

In this subsection we will present our main results and put them in context. The main new observations, ingredients and techniques will be explained in the subsequent subsection.

Let us first consider the case when the modulated wave profile satisfies Schr\"odinger-wave equation \eqref{eq:nls-wave}.
There are extensive numerical investigations which have yielded error estimates of $r^\varepsilon$ and the length of the time interval in which the estimates remain valid. A selection of such studies can be found in  \cite{BaoCai-KG-Limit-12,BaoCai-KG-Limit-14,BaoZh-KG-Limit, SchratzZhao-KG-Limits}. We emphasis that the recent numerical result by Bao and Zhao in \cite{BaoZh-KG-Limit} involves a meticulous numerical verification showcasing that  the error function $r_\varepsilon$ exhibits a convergence with error $O(\varepsilon^2)$, uniformly and golbally in time, as $\varepsilon$ goes to 0. This convergence phenomenon is confirmed for smooth initial data $(u_0,u_1)$ with a minimum requirement of $H^2$ regularity. Regrettably, despite these compelling numerical observations,  to the best of our knowledge,  a rigorous analytic study of this phenomenon remains conspicuously absent within the current body of literature.

The first two theorems of this paper aim to establish the non-relativistic limit when the modulated wave profile satisfies Schr\"odinger-wave equation \eqref{eq:nls-wave}.
\begin{thm}\label{thm:main1-global}
Let $d=2$. Suppose that
$$
(u_0,u_1)\in H^2(\R^2) \times H^2(\R^2);\quad
v_1\in L^2(\R^2).
$$
Let $u^\varepsilon, v^{\varepsilon}$ be the corresponding solutions to the focusing cubic nonlinear Klein-Gordon equation \eqref{eq:kg-eps} and the Schr\"odinger-wave equation \eqref{eq:nls-wave}, respectively.
Then the following error estimate
\begin{align}\label{est:main1-global}
\sup\limits_{t\in \R}\big\|u^\varepsilon(t)-\fe^{\frac{it}{\varepsilon^2}}v^{\varepsilon}(t)-\fe^{-\frac{it}{\varepsilon^2}}\bar v^{\varepsilon}(t) \big\|_{L^2_x}
\le C \varepsilon^2
\end{align}
holds true globally in time. Here $C$ is a generic constant depending only on $\|(u_0,u_1)\|_{H^2\times H^2}$ and $\|v_1\|_{L^2_x}$. 
%In particular, it is uniform and global in time.
\end{thm}

\begin{remark}
In \cite{BaoCai-KG-Limit-12,BaoCai-KG-Limit-14,BaoZh-KG-Limit, SchratzZhao-KG-Limits}, the following expression for $v_1$ is imposed
$$
v_1=\frac i2\big[-\Delta v_0+3|v_0|^2v_0\big],
$$
for the intention of aligning it with the potential limit solution presented in equation \eqref{eq:nls}. Such a specific choice of $v_1$ is removed here. It does not change the convergence estimate \eqref{est:main1-global} at all by choosing any $v_1$.
\end{remark}

\begin{remark}
The 3D case turns out to be more challenging. Our current method fails since it relies on the basic energy estimate and its consequence of space-time estimate in Lemma \ref{lem:spacetime-norm-KG} which are substantially weaker in 3D.
\end{remark}
The convergence rate established by this theorem aligns seamlessly with the numerical findings presented in \cite{BaoZh-KG-Limit, SchratzZhao-KG-Limits}, thus rigorously corroborating the second-order convergence behavior of the remainder function. In the next theorem, we shall embark on a demonstration, aiming at showcasing that the quadratic rate of convergence with respect to $\varepsilon$ is optimal, even for smooth enough initial data.

\begin{thm}\label{thm:main1-optimal}
Let $d=2,3$. There exist a class of triplets
$$
(u_0,u_1,v_1)\in \big[\mathcal S(\R^d)\big]^3,
$$
such that the convergence rate $\varepsilon^2$ achieved in Theorem \ref{thm:main1-global} attained its optimality. More precisely,  there exist constants $\varepsilon_0>0$ and $c_0 > 0$, such that
for any $\varepsilon\in (0,\varepsilon_0]$,
\begin{align}
\sup\limits_{t\in \R}\big\|u^\varepsilon(t)-\fe^{\frac{it}{\varepsilon^2}}v^{\varepsilon}(t)-\fe^{-\frac{it}{\varepsilon^2}}\bar v^{\varepsilon} (t)\big\|_{L^2_x}
\ge c_0 \varepsilon^2.
\end{align}
\end{thm}
We emphasis that the counterexample, or the initial data $(u_0,u_1,v_1)$  constructed in this theorem doesn't depend on $\varepsilon$ as $\varepsilon$ goes to 0. In view of the natural scaling of the cubic Klein-Gordon equation, 
it is natural to pursue the sharpness of the regularity requirement on the initial data. The sharpness on the $H^2$ regularity requirement of the initial data is out of the scope of this article and left to be an unsolved interesting problem.
%This will be settled down in the following theorem. 
%\begin{thm}\label{thm:main1-optimal-regularity}
%Let $d=2,3$. Then the $H^2$-regularity required in Theorem \ref{thm:main1-global} attained its optimality. More precisely,  for any $N>0$, there exist a triplet 
%$$
%(u_0,u_1,v_1)\in H^2(\R^d) \times H^2(\R^d)\times L^2(\R^d),
%\quad \mbox{ with }\quad
%\big\|u_0\big\|_{H^2}+\big\|u_1\big\|_{H^2}\ge N,
%$$
%a constant $\varepsilon_0>0$ and a function $\rho: \R^+\to \R^+$,  such that 
%for any $\varepsilon\in (0,\varepsilon_0]$,
%\begin{align}
%\sup\limits_{t\in \R}\big\|u^\varepsilon(t)-\fe^{\frac{it}{\varepsilon^2}}v^{\varepsilon}(t)-\fe^{-\frac{it}{\varepsilon^2}}\bar v^{\varepsilon} (t)\big\|_{L^2_x}
%\ge \rho(N) \varepsilon^2,
%\end{align}
%with $\rho(N)\to+\infty$ as $N\to +\infty$. 
%\end{thm}
%

Now let us consider the case when the modulated wave profile satisfies cubic Schr\"odinger equation \eqref{eq:nls-wave}. There are also a vast literatures focusing on obtaining precise estimates for error term $r^\varepsilon$ in the non-relativistic limit, as $\varepsilon$ approaches 0. For numerical investigations of this problem, one may refer to as \cite{BaoCaiZh-KG-Limit,BaoDong-KG-Limit,BaoZh-KG-Limit,ChCr-KG-Limit,ChMaMuVi-KG-Limit,DoXuZh-KG-Limit,SchratzZhao-KG-Limits,Zhao-KG-Limits} and the references therein. Research in \cite{BaoZh-KG-Limit,ChCr-KG-Limit,SchratzZhao-KG-Limits} has observed quadratic convergence in $\varepsilon$ for smooth data. In \cite{BaoZh-KG-Limit} and \cite{SchratzZhao-KG-Limits}, the numerical insights indicate that a reduction in the regularity of initial data might lead to a proportional decrease in the order of convergence rate. Numerical experiments of Bao and Zhao in \cite{BaoZh-KG-Limit} shows that for relatively non-smooth initial data ($(u_0,u_1)\in H^2$), the convergence rate reduces to be linear in $\varepsilon$.
In \cite{SchratzZhao-KG-Limits}, Schartz and Zhao further observed numerically that choosing the modulated Schrodinger profile of \eqref{eq:nls} necessitates more regularity of the initial data compared to choosing the modulated Schrodinger-wave profile of  \eqref{eq:nls-wave}, in order to achieve optimal quadratic convergence in $\varepsilon$. We also want to point out that there exist intriguing comparisons between the equations \eqref{eq:nls-wave} and \eqref{eq:nls} which govern the dynamics of modulated profiles in \cite{BaoZh-KG-Limit,SchratzZhao-KG-Limits}. Numerical simulations suggest that under the same $\varepsilon$ and at the same time, the error from choosing the modulated Schr\"odinger-wave profiles of \eqref{eq:nls-wave}  is much smaller than the one from  choosing the modulated Schr\"odinger profile of \eqref{eq:nls}. The model \eqref{eq:nls-wave} is a more suitable choice than \eqref{eq:nls} for studying the long-term behavior of the approximation.
It is also important to point out that comparing to using the limit model \eqref{eq:nls}, the drawback of the limit model \eqref{eq:nls-wave} lies in that it tends to exhibit high oscillations in time.

There are also certain analytic results that aim to rigorously verify these numerical findings. For instance, in \cite{MasNak-KG-Limits}, Masmoudi and Nakanishi establish an error estimate of $\varepsilon^\frac12$ within an $L^2$ framework under the assumption of finite energy and $L^2$ convergence of the initial data. It's worth noting that their result in \cite{MasNak-KG-Limits} is valid for times of order $O(1)$, which means it is applicable for any fixed time $T$. While there have been several numerical simulation results estimating error functions in relation to time dependency, the corresponding theoretical analyses in this regard are still quite limited.
Recently, Bao, Lu, and Zhang \cite{BaoLuZh-KG-Limit} extended the study to the three-dimensional problem and proved that for any
$$(u_0,u_1)\in \Big(H^{s+4\alpha}(\mathbb{R}^3)\cap L^1(\mathbb{R}^3)\Big)^2
$$
with $s>\frac32$, the error function
$$\|r(t)\|_{H^s}\lesssim (1+t)\varepsilon^\alpha
$$
holds for $\alpha=1, 2$.
This estimate is valid for times up to $O(\varepsilon^{-\frac12\alpha})$.
These analytic results
%provide essential validation and insights into the numerical observations,
offer a solid foundation for the understanding of the error behavior and convergence rates of the system. However, it's important to highlight that despite these efforts, there remains a shortage of theoretical analyses that comprehensively capture the intricate interplay among error estimates, regularity dependency, and time dependency observed in numerical simulations. In fact, through dimensional analysis, it becomes apparent that an increase in convergence rate can be achieved at the cost of sacrificing solution regularity and shortening the time span. To demonstrate it, let us consider for $f\in L^2(\mathbb{R}^3)$:
$$
\|\fe^{it\Delta}f_\varepsilon\|_{L^2_tL^6_x([0,T]\times \R^3)}\le C\|f\|_{L^2(\R^3)},
$$
where $f_\varepsilon=\varepsilon^\frac12f(\varepsilon^2t,\varepsilon x)$. Furthermore, if we assume that $f\in \dot H^1(\mathbb{R}^3)$, we can derive:
$$
\|\fe^{it\Delta}f_\varepsilon\|_{L^2_tL^6_x([0,T]\times \R^3)}\le CT\varepsilon\|\nabla f\|_{L^2(\R^3)},
$$
thus yielding a higher convergence rate than that obtained with $L^2$ data. Consequently, all the goals including achieving a fixed convergence order, optimizing the minimal regularity, and determining the validation time interval necessitate a comprehensive exploration of the equation's underlying structure.

Our second result serves as a rigorous verifications of the above mentioned numerical ones. It also improves the existing analytic results.
\begin{thm}\label{thm:main3-local-d2}
Let $d=2, 3$ and $2\le \alpha\le 4$. Suppose that
$$
(u_0,u_1)\in H^\alpha(\R^d) \times H^\alpha(\R^d).
$$
Let $u^\varepsilon, v$ be the corresponding solutions to cubic nonlinear Klein-Gordon equation \eqref{eq:kg-eps} and cubic nonlinear Schr\"odinger equation \eqref{eq:nls}, respectively.
Then the following error estimate
\begin{align}\label{est:main3-local-d2}
\big\|u^\varepsilon(t)-\fe^{\frac{it}{\varepsilon^2}}v(t)-\fe^{-\frac{it}{\varepsilon^2}}\bar v(t) \big\|_{L^2_x}
\le C\Big(\varepsilon^2+\big(\varepsilon^2 t\big)^{\frac14\alpha}\Big)
\end{align}
holds for all $t\ge 0$, where $C$ is a generic constant depending only on $\|(u_0,u_1)\|_{H^\alpha\times H^\alpha}$.
\end{thm}
The proof of Theorem \ref{thm:main3-local-d2} requires a substantial improvement of the famous results by Machihara, Nakanishi, and Ozawa \cite{MaNaOz-KG-Limits}, which will be stated as the third result or Theorem \ref{thm:main3} of this paper below.
The conclusion of Theorem \ref{thm:main3-local-d2} in the two-dimensional case is a direct corollary of Theorems \ref{thm:main1-global} and \ref{thm:main3} by selecting $\widetilde v_0^\varepsilon = \widetilde v_0$ and $\widetilde v_1^\varepsilon = \varepsilon^2 v_1$.
The three-dimensional case can be achieved through a slight adaptation of the argument employed in the proofs of Theorems \ref{thm:main1-global} and \ref{thm:main3}. Moreover, it is not hard to extend the restriction $2\le \alpha\le 4$
to $1\le \alpha\le 4$ with a slight modification of the argument, however, we do not pursue this generality. 

The conclusions illustrated in Theorems \ref{thm:main1-global} and \ref{thm:main3-local-d2}, alongside the numerical findings presented in \cite{BaoZh-KG-Limit} and \cite{SchratzZhao-KG-Limits}, show that choosing the modulated Schr\"odinger profiles of \eqref{eq:nls} 
necessitates higher regularity on the initial data, in contrast to choosing the modulated Schr\"odinger-wave profiles of \eqref{eq:nls-wave}. 
It is worth noting that our results exhibit three enhancements over existing ones in  \cite{MasNak-KG-Limits, BaoLuZh-KG-Limit}. Firstly, our prerequisites for regularity are notably more lenient.  Specifically, for achieving quadratic convergence, we merely stipulate that $(u_0, u_1) \in H^4(\mathbb{R}^d) \times H^4(\mathbb{R}^d)$.
Indeed, as we will demonstrate below,  {\it the regularity assumptions are sharp.}
Secondly, our result holds true for all time but not only a finite time interval. Thirdly, our framework adeptly accommodates both two and three dimensional situations, thereby broadening the scope of applicability.

As mentioned above, the third result is a substantial improvement of the famous result of Machihara, Nakanishi, and Ozawa \cite{MaNaOz-KG-Limits}, which will serve as a key step in proving Theorem \ref{thm:main3-local-d2}. To start with, let us recall some studies on the non-relativistic limits of the following Cauchy problem
\EQn{
	\label{eq:nls-wave-N}
	\left\{ \aligned
	&\varepsilon^2\partial_{tt}\widetilde{v}^\varepsilon+2i\partial_{t} \widetilde{v}^\varepsilon - \De \widetilde{v}^\varepsilon= -  |\widetilde{v}^\varepsilon|^2 \widetilde{v}^\varepsilon, \\
	& \widetilde{v}^\varepsilon(0,x) = \widetilde{v}^\varepsilon_0(x),\quad \partial_t\widetilde{v}^\varepsilon(0,x)=\frac{\widetilde{v}^\varepsilon_1(x)}{\varepsilon^2}.
	\endaligned
	\right.
}
which may be derived by applying the following modulated transformation
$$
 u^\varepsilon=\fe^{\frac{it}{\varepsilon^2}}\widetilde{v}^\varepsilon
 $$
 to the cubic nonlinear Klein-Gordon equation. In 1983, Tsutsumi \cite{Tsutsumi-KG-Limits} established that in the two-dimensional case, the limit of \eqref{eq:nls-wave-N} is the solution to the cubic nonlinear Schrödinger equation, commonly expressed as:
\EQn{
	\label{eq:nls-1}
	\left\{ \aligned
	&2i\partial_{t} \widetilde{v} - \De \widetilde{v}= - |\widetilde{v}|^2 \widetilde{v}, \\
	& \widetilde{v}(0,x) = \widetilde{v}_0(x),
	\endaligned
	\right.
}
where the convergence holds assuming that $(\widetilde{v}^\varepsilon_0, \varepsilon^{-1} \widetilde{v}^\varepsilon_1)$ converges to $(\widetilde{v}_0, 0)$ in the $H^2$ norm (note that it is not $(\widetilde{v}^\varepsilon_0, \widetilde{v}^\varepsilon_1)$). In subsequent works such as \cite{Najman-KG-Limits,Machihara-KG-Limits}, the regularity prerequisites were relaxed.  These studies demonstrated convergence in the $L^q$ space with $2\le q<6$, under the assumption of boundedness in $H^1$ and $L^2$ convergence of the initial data $(\widetilde{v}^\varepsilon_0, \varepsilon^{-1} \widetilde{v}^\varepsilon_1)$.  The $H^2$ convergence of the error function was established in \cite{BeCo-KG-Limit}, subject to certain assumptions.
A significant breakthrough in determining the precise convergence order of non-relativistic limits was achieved by Machihara, Nakanishi, and Ozawa in \cite{MaNaOz-KG-Limits}. By assuming finite energy and $L^2$ convergence of the initial data $(\widetilde{v}^\varepsilon_0, \varepsilon^{-1} \widetilde{v}^\varepsilon_1)$, the authors established that the solution to \eqref{eq:nls-wave-N} converges to \eqref{eq:nls-1} up to a fixed time, with an optimal rate of $o(\varepsilon^\frac12)$. Again, it is not $(\widetilde{v}^\varepsilon_0, \widetilde{v}^\varepsilon_1)$. 
Meanwhile, Zhao \cite{Zhao-KG-Limits} carried out numerical simulations, revealing that the solution of \eqref{eq:nls-wave-N} attains quadratic convergence in $\varepsilon$ with linear growth in time, characterized by:
\begin{align*}
\big\|\widetilde{v}^\varepsilon(t)-\widetilde{v}(t)\big\|_{H^1_x}
\le &C(1+t)\varepsilon^2,
\end{align*}
provided that the initial data $(\widetilde{v}^\varepsilon_0,\varepsilon^{-2}\widetilde{v}_1^\varepsilon)$ in \eqref{eq:nls-wave-N} maintains a certain degree of smoothness (but independent of $\varepsilon$). 
However, when the initial data's regularity is weaker, e.g., $(\widetilde{v}^\varepsilon_0,\varepsilon^{-2}\widetilde{v}_1^\varepsilon)\in H^2\times H^2$, the convergence rate diminishes to be linear in $\varepsilon$, i.e.
\begin{align*}
\big\|\widetilde{v}^\varepsilon(t)-\widetilde{v}(t)\big\|_{H^1_x}
\le &C(1+t)\varepsilon.
\end{align*}

Obviously, the results of Machihara, Nakanishi, and Ozawa in \cite{MaNaOz-KG-Limits} is not enough for our purpose since there is an extra $\varepsilon^{-1}$ in the initial velocity. Indeed, to apply the results in \cite{MaNaOz-KG-Limits, Zhao-KG-Limits} directly to the current situation, one has to take the initial velocity $\widetilde{v}^\varepsilon_1$ to be 0. Our third result aims to obtain an improvement of those existing ones, serving a key step in proving our Theorem \ref{thm:main3-local-d2}.
\begin{thm}\label{thm:main3}
Let $d=2,3$,  $s_d=d-2$, $1\le \beta\le 4$ and $\beta\le \alpha\le 4$.  Assume that
  \begin{itemize}
  \item[(1)] $\widetilde{v}_0\in H^\alpha(\R^d),\ \widetilde{v}^\varepsilon_0\in H^\beta_x(\R^d),\ \widetilde{v}_1^\varepsilon\in H^{s_d}(\R^d)$;
  \item[(2)] $(\widetilde{v}^\varepsilon_0 - \widetilde{v}_0,\ \widetilde{v}_1^\varepsilon)\to (0,0) \mbox{ in } L^2(\R^d)\times L^2(\R^d).$
  \end{itemize}
%  When $d=3$,  additionally assume that
%    \begin{itemize}
%    \item[(3)] $\varepsilon^2 T\le \delta_0$.
%  \end{itemize}
Denote the corresponding solutions to \eqref{eq:nls-wave-N} and \eqref{eq:nls-1} by $\widetilde{v}^\varepsilon$ and $\widetilde{v}$, respectively. Then for any $t>0$, one has
\begin{align}\label{est:Thm3-main}
\big\|\widetilde{v}^\varepsilon(t)-\widetilde{v}(t)\big\|_{L^2_x}
\lesssim & \big\|(\widetilde{v}^\varepsilon_0-\widetilde{v}_0, \widetilde{v}_1^\varepsilon)\big\|_{L^2_x\times L^2_x}
+\varepsilon^\beta\big\|\widetilde{v}^\varepsilon_0 - \widetilde{v}_0\big\|_{H^\beta_x} + \varepsilon^{\min\{\alpha,2\}}+\big(\varepsilon^2 t\big)^{\frac14\alpha}.
\end{align}
\end{thm}

We emphasis that the crucial point in Theorem \ref{thm:main3} is that the condition on $\widetilde{v}_1^\varepsilon$ is less stringent so that it is applicable to Theorem \ref{thm:main3-local-d2}. While it is demanded in  \cite{MaNaOz-KG-Limits}  that $\widetilde{v}_1^\varepsilon=o(\varepsilon)$, our Theorem \ref{thm:main3} relaxes this requirement to $\widetilde{v}_1^\varepsilon=o(1)$. Our theoretical findings closely align with the observed convergence patterns presented in \cite{Zhao-KG-Limits}, which numerically indicated a convergence rate of $\varepsilon^2$ for smooth initial data and a rate of $\varepsilon$ for less smooth data. For instance, if $(\alpha, \beta, s_d) = (4, 2, 1)$ and $(\widetilde{v}_0, \widetilde{v}^\varepsilon_0, \widetilde{v}_1^\varepsilon) \in H^4(\R^d) \times H^2(\R^d) \times H^1(\R^d)$, then the bound in Theorem \ref{thm:main3} reads
$$
\big\|\widetilde{v}^\varepsilon(t)-\widetilde{v}(t)\big\|_{L^2_x}
\lesssim  \|(\widetilde{v}^\varepsilon_0-\widetilde{v}_0,\widetilde{v}_1^\varepsilon)\|_{L^2_x\times L^2_x}+(1+t)\varepsilon^2.
$$
Moreover, our result reveals that, to get an extra $\varepsilon^\frac12$ convergence rate, one only needs one additional derivative in the regularity of the initial data (see also \cite{MaNaOz-KG-Limits} wherein the optimal $\varepsilon^\frac12$ rate was established for $H^1\times L^2$ initial data).  Besides, the bound here is explicit in time. It holds for all $t > 0$, in particular, both small and large $t$.

\begin{remark}
If
$(\widetilde{v}_0, \widetilde{v}^\varepsilon_0, \widetilde{v}_1^\varepsilon) \in (H^1(\R^d)^3$, then the above bound reads
$$
\big\|\widetilde{v}^\varepsilon(t)-\widetilde{v}(t)\big\|_{L^2_x}
\lesssim  \|(\widetilde{v}^\varepsilon_0-\widetilde{v}_0,\widetilde{v}_1^\varepsilon)\|_{L^2_x\times L^2_x}+\varepsilon+t^{\frac14}\varepsilon^\frac12.
$$
This is even better than the numerical observations in \cite{Zhao-KG-Limits} where the dependence of this upper bound on time is linear.
In the case of $d=2$, the requirement on $\widetilde{v}_1$ can be relaxed to $\widetilde{v}_1 \in L^2(\R^d)$.
\end{remark}

The following theorem shows that  the regularity requirement  and the convergence bound $(1+t)\varepsilon^2$ achieved in Theorem \ref{thm:main3} attained their optimality. 
%Firstly, we need the following definition. We call the sequence $\{t_n\}_{n=1}^\infty$ is {\it adjacent} if it verifies that 
%$$
%t_n\to +\infty, \quad \mbox{ and } \quad t_{n+1}-t_n \le 1. 
%$$
%Then the conclusion is 
\begin{thm}\label{thm:main3-optimal}
 Let $d=2,3$,  and set 
 $
 \widetilde{v}^\varepsilon_0=\widetilde{v}_0, \widetilde{v}_1^\varepsilon=0.
 $
 Then there exist a class of 
$$
\widetilde{v}_0\in \mathcal S(\R^d),
$$
constants $\varepsilon_0>0$, $c_0>0$ and $\delta_0$, such that
for any $\varepsilon\in (0,\varepsilon_0]$, 
\begin{align}
\big\|\widetilde{v}^\varepsilon(t)-\widetilde{v}(t)\big\|_{L^2_x}
\ge c_0 t \varepsilon^2,\quad \mbox{ when } 1\le t \le \delta_0 \varepsilon^{-2} .
\end{align}
Moreover, for $1\le \alpha\le 4$, there exists an initial data $\widetilde{v}_0 \in H^\alpha(\R^d)$ satisfying 
%$$
%\|\widetilde{v}_0\|_{H^\alpha_x}<+\infty,
%%, \mbox{ for any } s<\alpha;\quad 
%%\|\widetilde{v}_0\|_{H^{\alpha}}=+\infty,
%$$
such that 
\begin{align*}
\big\|\widetilde{v}^\varepsilon(t)-\widetilde{v}(t)\big\|_{L^2_x}
\ge c_0\big(\varepsilon^2 t\big)^{\frac14\alpha}\big|\ln \big(t\varepsilon^2\big) \big|^{-1},\quad \mbox{ when } 1\le t \le \delta_0 \varepsilon^{-2} . 
\end{align*}
\end{thm}

\begin{remark}
The counterexamples constructed in this theorem are equally applicable to show the sharpness of the conclusion in Theorem \ref{thm:main3-local-d2}.  In fact, as we will see, the transition from the wave function $u^\varepsilon$ to the Schr\"dinger profile $v$ in the non-relativistic limits is primarily influenced by the transition from the modulated Schr\"dinger-wave profile $v^\varepsilon$ to the modulated Schr\"odinger profile $v$, as expounded in the proof of Theorem \ref{thm:main3}.
\end{remark}

%
%$$
%\|v_0\|_{H^{\alpha}}=\ln N,  
%$$
%then 
%$$
% C\ln(\ln N)(\varepsilon^2 t)^{\frac14\alpha}\le  \big\|v^\varepsilon(t)-v(t)\big\|_{L^2_x}\le C \big(\varepsilon^2 t\big)^{\frac14\alpha}\ln N
%$$

\subsection{Key ingredients of the proofs}

%
%Denote $v^{\varepsilon}$ to be the solution of the following cubic nonlinear Schr\"odinger-wave equation:
%\EQn{
%	\label{eq:nls-wave}
%	\left\{ \aligned
%	&\varepsilon^2\partial_{tt}v+2i\partial_{t} v - \De v= - 3|v|^2 v, \\
%	& v(0,x) = v_0(x)\triangleq \frac12(u_0-iu_1),\quad \partial_tv(0,x) =0.
%	\endaligned
%	\right.
%}
%
%\begin{conj}\label{conj-2}
%Let    $T^*_\varepsilon\in (0, +\infty]$ be the maximal lifespan of the problem \eqref{eq:nls-wave}.
%There exists a space  $\mathcal Z_0$ and some $\alpha>0$, suppose that
%$$
%(u_0,u_1)\in \big(\mathcal Z_0\big)^2.
%$$
%Then for any $t\in (0, T^*_\varepsilon)$, there exists a constant $C>0$ such that
%\begin{align}
%\big\|u(t)-\fe^{\frac{it}{\varepsilon^2}}v^{\varepsilon}-\fe^{-\frac{it}{\varepsilon^2}}\bar v^{\varepsilon} \big\|_{\dot H^{s_c}_x}
%\le C \varepsilon^\alpha.
%\end{align}
%Either the larger the initial space $\mathcal Z_0$ or the higher the convergence order $\alpha$, the better.
%%where $C$ only depends on $\|(u_0,u_1)\|_{H^s\times H^{s-1}}$.
%\end{conj}
%
%However, to the best of our knowledge, there is no rigorous result on this problem in the literature.

%\noindent{\bf Hypothesis (H):}  Let $d=2,3$, $v_0\in H^1(\R^d)$ and $\varepsilon\in (0,1)$, then the Cauchy problem \eqref{eq:nls-wave}
%is globally well-posed and scatters, and
%$$
%\big\|v\big\|_{X_{s_c}(\R)}\le  C(\|v_0\|_{H^1_x}),
%$$
%with some constant $C$ is independent of $\varepsilon$.
%

First, we elucidate the crucial element of the proofs underlying our main theorems.

\vskip .3cm

$\bullet${\it  Analysis of Non-Resonance in the Nonlinearity: Proofs of Theorem \ref{thm:main1-global} and Theorem \ref{thm:main3-local-d2}}
\vskip .3cm

The nonlinear term $B_\varepsilon(v^\varepsilon)$ defined in \eqref{def:B-var} can be viewed as the inhomogeneous component of the corresponding equation for the residual function $r^\varepsilon$. Nevertheless, this term carries significant magnitude. Specifically, we seek to understand how it differs from $|v^\varepsilon|^2v^\varepsilon$, which originates from the nonlinear term in the limiting equation \eqref{eq:nls-wave}. 
A crucial point for proofs of Theorem \ref{thm:main1-global} and Theorem \ref{thm:main3-local-d2} is  to bound the following  time-integral
$$
\int_0^t K_\varepsilon(s)\Big[\fe^{\frac{3is}{\varepsilon^2}}(v^\varepsilon)^3+\fe^{-\frac{3is}{\varepsilon^2}}\big(\bar v^\varepsilon\big)^3\Big]\,ds, 
%\quad K_\varepsilon(t)=\fe^{-ic(c\pm \langle \nabla\rangle_c)t}
$$
%$$
%\varepsilon^2\partial_{tt}=\fe^{\frac{3is}{\varepsilon^2}}(v^\varepsilon)^3+\fe^{-\frac{3is}{\varepsilon^2}}\big(\bar v^\varepsilon\big)^3
%$$
by $O(\varepsilon^2)$.  
Here $K_\varepsilon$ is a solution operator of linear Klein-Gordon equations \eqref{eq:kg-eps}.  
However, it is not transparent. 
Our initial insight pertains to the non-resonance structure of the time-integral.
To demonstrate, we begin with a rescaling argument, transforming it into the consideration of the following time-integral
$$
\int_0^{\varepsilon^{-2}t}  \fe^{-is\langle\nabla\rangle} \Big[\fe^{3is}(\mathcal{S}_\varepsilon v^\varepsilon)^3+\fe^{-3is}\big(\mathcal{S}_\varepsilon \bar v^\varepsilon\big)^3\Big]\,ds,
$$
where $\mathcal{S}_\varepsilon f(t,x)=\varepsilon f(\varepsilon^2t,\varepsilon x)$.
Note that the estimate for the above integral is easy if its integrand
 is applied by an operator 
$
P_{\ge 1}.
$
Then the non-resonance structure comes in if we look into the following the exponential operators
$$
\fe^{-is(\langle\nabla\rangle-3)}P_{\le 1};\quad \mbox{and }\quad
\fe^{-is(\langle\nabla\rangle+3)}.
$$
Here $P_{\le 1}$ and $P_{\ge 1}$ are the frequency localized operators defined in Section \ref{sec:notations}.
By normal form transformation,
it suffices to estimate the following  time-integral (with the abuse of ignoring useless operators):
$$
\varepsilon^2 \int_0^{\varepsilon^{-2}t}  \fe^{-is\langle\nabla\rangle}\Big[\fe^{3is}(\mathcal{S}_\varepsilon v^\varepsilon)^2 \mathcal{S}_\varepsilon \partial_s v^\varepsilon+\fe^{-3is}\big(\mathcal{S}_\varepsilon \bar v^\varepsilon\big)^2 \mathcal{S}_\varepsilon  \partial_s v^\varepsilon\Big]\,ds.
$$
We emphasis that an extra $\varepsilon^2$ appears in the above expression.  
The substantial difference between $(v^\varepsilon)^3$ and $|v^\varepsilon|^2v^\varepsilon$ is presented. 
Then our next goal is to get the uniform bound  in $\varepsilon$ for $\partial_t v^\varepsilon$ within a suitable space.
%This transformation guides us toward our objective concerning the zero-phase behavior of $\partial_t v^\varepsilon$. More specifically, we require uniform boundedness in $\varepsilon$ for $\partial_t v^\varepsilon$ within a suitable space.

It is natural to turn to the  conservation of the basic energy:
%.   examine physical quantities derived from the equations. In particular, by employing equation \eqref{eq:nls-wave}, we can readily deduce the subsequent conservation law for energy:
\begin{align}\nonumber
\int\big(\varepsilon^2|\partial_tv^\varepsilon|^2 + |\nabla v^\varepsilon|^2  + \frac{3}{2}|v^\varepsilon|^4\big)dx
= \int\big(\varepsilon^2|v_1|^2+\frac{1}{4}|\nabla (u_0 - iu_1)|^2  + \frac{3}{2}|u_0 - iu_1|^4\big)dx.
\end{align}
Unfortunately, this leads to the conclusion that
\begin{align}\label{pt-v-trival}
\|\partial_tv^\varepsilon\|_{L^2_x}=O(\varepsilon^{-1}).
\end{align}
This natural bound is far from our requirement. Our strategy is to gain the $\varepsilon$-loss of  $\partial_tv^\varepsilon$ by crucial observations that a magic complex extension of the system possesses non-resonance structure and a well-behaved initial data.

\vskip .3cm

$\bullet$ {\it Gaining the $\varepsilon$-loss from Time-derivative in Proof of Theorem \ref{thm:main1-global} }
\vskip .3cm

The key ingredient to  gain the $\varepsilon$-loss of  $\partial_tv^\varepsilon$ is the following proposition.
Denote the space $X_{\alpha}(I)$ with $\alpha\in\R, I\subset \R$, with its norm
$$
\|f\|_{X_{\alpha}(I)}\triangleq \big\||\nabla|^\alpha f\big\|_{L^{\frac{2(d+2)}{d}}_{tx}(I)}+\big\||\nabla|^{\alpha} f\big\|_{L^\infty_{t}L^2_x(I)}.
$$
\begin{prop}\label{prop:v-varep}
Let $\gamma\ge 0$, and $v^{\varepsilon}$ to be the solution of \eqref{eq:nls-wave}, then for any $v_0\in H^{2+\gamma}(\R^2)$ and $v_1\in H^{\gamma}(\R^2)$, there holds 
\begin{align}\label{est:v-varep-px}
\left\| \langle\nabla\rangle^\frac32 v^{\varepsilon} \right\|_{L^4_{tx}(\R)}+\big\|v^{\varepsilon} \big\|_{L^\infty_t\dot H^2_x(\R)}\le C,
\end{align}
and
\begin{align}\label{est:v-varep-pt}
\big\| |\nabla|^{\gamma} \partial_tv^{\varepsilon} \big\|_{X_0(\R)}\le C,
\end{align}
where $C$ is a generic constant depending only on $\|v_0\|_{H^{2+\gamma}(\R^2)}$.
\end{prop}
The estimate \eqref{est:v-varep-pt} established in the  proposition above significantly improves the one in \eqref{pt-v-trival}. 
%\begin{remark}
%It might be even more surprising that only a single time-derivative, denoted as $\partial_tv^{\varepsilon}$, exhibits effectively to gain the $\varepsilon$-loss.
%In contrast, there appears to be no mechanism by which higher-order time derivatives, such as $\partial_t^kv^{\varepsilon}$ with $k\ge2$, manifest the possibility of saving more $\varepsilon$-loss. In fact, concrete examples can be constructed to demonstrate that the norm of $\partial_tv^{\varepsilon}$ behaves in the following manner:
% $$
% \big\|\partial_t^k v^{\varepsilon} \big\|_{X_0(\R)}\sim \varepsilon^{2-2k},\quad k\ge1. 
% $$
%For a more in-depth exploration of this observation, one may refer to Section \ref{sec:sharp-global-op1}.
%\end{remark}
The proof of this proposition is presented in Section \ref{sec:prop-key}. The main ideas are as follows.

%
%
%As described above, the energy conservation law yields a straightforward $O(\varepsilon^{-1})$ upper bound for $\|\partial_tv^\varepsilon\|_{L^2_x}$. The estimate \eqref{est:v-varep-pt} established in the  proposition above significantly refines this rudimentary bound, constituting a pivotal element in our analysis. 
%This enhancement is accomplished through the following three aspects.

(1) {\it Rescaled Complex Extension of  Equation \eqref{eq:kg-eps}}: We begin by introducing a thoughtfully chosen complex extension of $u^\varepsilon$, which enables us to producing extra small parameters in both initial data and nonlinearities.
More precisely, we introduce a judiciously designed scaling operator denoted as $\mathcal{S}_\varepsilon$, and employ the transformation $w=e^{it}\mathcal{S}_\varepsilon(v^\varepsilon)$. As a result, the function $w$ adheres to an intricately formulated {\it complex expanded} nonlinear Klein-Gordon equation \eqref{eq:kg-eps} (with the change of coefficient in nonlinear term from 1 to 3, which does not change our analysis):
$$
\partial_{tt}w-\De w+w=-3|w|^2w
$$
with specific initial data
$$
 w(0)=\mathcal{S}_\varepsilon v_0,\quad
\partial_t w(0)= i\mathcal{S}_\varepsilon v_0+\varepsilon^2 \mathcal{S}_\varepsilon v_1.
$$
Firstly, an extra gain of $\varepsilon^2$ appears, exhibiting a magic structure in the initial data. 
Secondly, it becomes evident that:
\begin{align}\label{relationship-w-wt}
w(0)=i\partial_t w(0)+O(\varepsilon^2),\quad \mbox{ in }L^2.
\end{align}
This relationship plays a crucial role in the our analysis.

(2) {\it Enhanced  Performance of the Leftward Wave.}
Another crucial ingredient of the rescaled complex extension is the inherent non-resonance structure in nonlinearity. 
Indeed, this can be seen by introducing  the {\it leftward wave} $\mathcal W_1$ and {\it rightward wave} $\mathcal W_2$,
\EQ{
	\mathcal W_j=\langle \nabla \rangle^{-1}\big(\partial_t\mp i \langle \nabla\rangle\big)w, \mbox{ for } j=1,2,
}
where these waves elegantly fulfill the following systems of equations:
\begin{align}\label{eqs:math-W-12}
\big(\partial_t\pm i \langle \nabla\rangle\big)\mathcal W_j=-\langle \nabla \rangle^{-1}(|w|^2w).
\end{align}
When compared to the rightward wave, it's somewhat surprising that the leftward wave exhibits superior behavior.  The crux of the proof is rooted in the subsequent pivotal observation:

% This observation is rooted in the following two key factors:

\textcircled{1} Well-behaved Initial Data. 
By  \eqref{relationship-w-wt}, it is easy to see that the initial data  $\mathcal W_1(0)$ can be approximated as follows:
$$
\mathcal W_1(0)=\Delta \mathcal{S}_\varepsilon v_0+\varepsilon^2\mathcal{S}_\varepsilon v_1.
$$
Extra $\varepsilon^2$ smallness is present in both terms on the right-hand side of the above expression. 

%This improvement arises from the fact that, owing to \eqref{relationship-w-wt}, $\mathcal W_1(0)$ involves an additional two derivatives, which can be approximated as follows:

%An important estimation reveals that an extra derivative leads to an additional contribution of the order of $\varepsilon$.

\textcircled{2} Exploring Non-Resonant Structure. Ignoring some terms in nonlinearity,   we have the following equation, which characterizes the main difficulties in the nonlinearity:
$$
\big(\partial_t+ i \langle \nabla\rangle\big)\mathcal W_1=-P_{\le1} \Big[|P_{\le 1}\mathcal W_2|^2P_{\le 1}\mathcal W_2\Big]+\cdots.
$$
A non-resonant structure can be explored here. 
Indeed, the corresponding phase function of this nonlinearity is given by:
$$
\langle \xi\rangle+\langle \xi_1\rangle+\langle \xi_2\rangle-\langle \xi_3\rangle\ge 1,\quad \mbox{when } \quad |\xi|\le 1 \mbox{ and }  |\xi_j|\le 1, j=1,2,3.
$$
In conjunction with observation \textcircled{1} above, this insight allows us to obtain an additional factor of $\varepsilon^2$ through the normal-form argument.
%Notably, a fundamental bootstrap argument establishes that the nonlinear term
%$$
%-P_{\le1}\Big[|P_{\le 1}\mathcal W_2|^2P_{\le 1}\mathcal W_2\Big]
%$$
%in  \eqref{eqs:math-W-12} for $j=1$, plays a pivotal role in the solution's analysis.

(3) {\it Passing the properties of the leftward wave to $\partial_t \mathcal{S}_\varepsilon(v^\varepsilon)$. } 
It can be derived that (see Section \ref{subsec:keyprop-end} for the details)
$$
\mathcal W_1\sim \partial_t \mathcal{S}_\varepsilon(v^\varepsilon).
$$
This enables us to gain an additional factor of $\varepsilon^2$ in the estimation of $\partial_t \mathcal{S}_\varepsilon(v^\varepsilon)$.

%
%The behavior of $\mathcal W_1$ is remarkably consistent with that of $\partial_t \mathcal{S}_\varepsilon(v^\varepsilon)$, lending itself to an approximation of the form:
%
%Consequently, analogous to the case of $\mathcal W_1$, we can derive an additional factor of $\varepsilon^2$ in the estimation of $\partial_t \mathcal{S}_\varepsilon(v^\varepsilon)$, thus ensuring the uniform boundedness of $\partial_t v^\varepsilon$.

\vskip .3cm

$\bullet$ {\it  Regularity Gain through the ``High-Low'' Decomposition in Proof of Theorem \ref{thm:main3}}
\vskip .3cm

The primary challenge in establishing   Theorem \ref{thm:main3} stems from the substantial discrepancy in the properties of linear operators between equations \eqref{eq:nls-wave-N} and \eqref{eq:nls-1}. The linear operators for these equations are expressed as follows:
$$
K_\varepsilon(t)=\fe^{-ic(c\pm \langle \nabla\rangle_c)t},\quad S(t)=\fe^{-\frac12it\Delta}.
$$
Here $\langle \nabla\rangle=\sqrt{|\nabla|^2+c^2}$ and $c=\varepsilon^{-1}$. 

In the previous work \cite{MaNaOz-KG-Limits}, the authors derived a convergence rate in the energy space, proportional to the square-root of $\varepsilon$, using uniform Strichartz estimates on $K_\varepsilon$ and spacetime estimates on the remainder operator: 
$$
R_\varepsilon(t)=K_\varepsilon(t)-S(t),
$$
which grows linearly in time and will lead to a quadratic one in the final conclusion. 
Furthermore, this argument relies on the uniform boundedness of the solution in $H^1$, a result stemming from the energy estimates. Specifically, solutions to \eqref{eq:nls-wave-N} exhibit energy conservation, which can be expressed as follows: 
\begin{align}\label{energy-NLS-wave}
\int\big(\varepsilon^2|\partial_t\widetilde{v}^\varepsilon|^2 + |\nabla \widetilde{v}^\varepsilon|^2  + \frac{3}{2}|\widetilde{v}^\varepsilon|^4\big)dx
= \int\big(\varepsilon^{-2}|\widetilde{v}_1^\varepsilon|^2  + |\nabla \widetilde{v}_0^\varepsilon|^2+\frac{3}{2}|\widetilde{v}_0^\varepsilon|^4\big)dx.
\end{align}
In our situation, when $(\widetilde{v}_0^\varepsilon, \widetilde{v}_1^\varepsilon)$ are bounded in $H^1(\mathbb{R}^d)\times L^2(\mathbb{R}^d)$, trivial bounds can be derived from the energy conservation, resulting in:
$$
\|\partial_t \widetilde{v}^\varepsilon\|_{L^2_x}=O(\varepsilon^{-2}),\quad
\mbox{ and }\quad \|\nabla \widetilde{v}^\varepsilon\|_{L^2_x}=O(\varepsilon^{-1}).
$$
These estimates, however, exhibit very bad dependence on $\varepsilon$ and is too weak. In contrast, in \cite{MaNaOz-KG-Limits}, the authors proposed a more stringent assumption by considering that $(\widetilde{v}_0^\varepsilon, \varepsilon^{-1} \widetilde{v}_1^\varepsilon)$ are bounded in $H^1(\mathbb{R}^d)\times L^2(\mathbb{R}^d)$ (the same assumption was also proposed in \cite{MasNak-KG-Limits}). Their stronger assumption eliminates the singularity of the solution $\widetilde{v}$ in $H^1$ and mitigates the singularity of $\partial_t\widetilde{v}$ in $L^2$.
That is, 
$$
\|\partial_t \widetilde{v}^\varepsilon\|_{L^2_x}=O(\varepsilon^{-1}),\quad
\mbox{ and }\quad \|\nabla \widetilde{v}^\varepsilon\|_{L^2_x}=O(1).
$$

Our new approach is capable of  reducing the growth of time to linear and eliminating all of these singularities with only a weaker assumption, namely that $\|(\widetilde{v}_0^\varepsilon, \widetilde{v}_1^\varepsilon)\|_{H^1_x\times L^2_x}$ is bounded.
This represents a significant advancement in understanding the non-relativistic limit of \eqref{eq:kg-eps}, as we are able to achieve a more comprehensive removal of singularities under a less restrictive condition.

Our new approach is based on specific {\it high-low frequency decomposition} of the solution, which helps to gain regularity. Define 
$$
\widetilde{v}_N=NLS (P_{\le N}\widetilde{v}_0) \quad\mbox{ and } \quad \widetilde v^N= \widetilde v-\widetilde{v}_N. 
$$ 
Here, $NLS(u_0)$ signifies the nonlinear solutions of equation \eqref{eq:nls-1} with initial data $u_0$. By implementing distinct strategies for the regular and non-regular cases, we achieve different bounds. Specifically, in the regular case, we obtain a bound for the error function $\widetilde{r}^\varepsilon$ with linear dependence on the highest regularity of the initial data, expressed as
$$
\varepsilon^2 t \|\Delta^2 \widetilde{v}_N\|_{L^2_x}. 
$$
In the non-regular case, the error bound is given by
$$
\big\|\widetilde{v}^N\big\|_{L^2_x}.
$$
Subsequently, through the ``high-low" frequency decomposition, the error bound can be derived as
$$
\Big(\varepsilon^2 t N^{4-\alpha}+ N^{-\alpha}\Big)\big\|\widetilde{v}\big\|_{H^\alpha_x}.
$$
By suitably selecting the frequency $N$, we achieve the desired estimate. 

\vskip .3cm

$\bullet$ {\it Constructions of the Counterexamples in Theorem \ref{thm:main1-optimal} and Theorem \ref{thm:main3-optimal}}
\vskip .3cm

The above high-low frequency decomposition also plays an important role in the constructions of the counterexamples to show the sharpness of the regularity. 
These constructions require a comprehensive understanding of the solution's structure, and distinct strategies  in various scenarios.

(1) Counterexample in Theorem \ref{thm:main1-optimal}. The construction need use a deep result of Cazenave and Weissler \cite{CaWe-1992}.  Let us roughly write the solution as follows,  
$$
\phi(t)=\fe^{it\langle \nabla\rangle}\phi_0
+\fe^{it\langle \nabla\rangle}\big(\mathcal{S}_\varepsilon v^\varepsilon_0\big)^3
-\big(\mathcal{S}_\varepsilon v^\varepsilon(t)\big)^3
+\mathcal N(\phi, \mathcal{S}_\varepsilon v^\varepsilon),
$$
where the Duhamel term $\mathcal N(f,g)$  takes the form of:
$$
\mathcal N(f,g)=\int_0^t \fe^{i(t-s)\langle \nabla\rangle}F(f,g)(s)\,ds.
$$
This rough expression is due to the normal form transformation, which does not change the construction of the counterexample in Theorem \ref{thm:main1-optimal}.
The term $\fe^{it\langle \nabla\rangle}\big[\phi_0+\big(\mathcal{S}_\varepsilon v^\varepsilon_0\big)^3\big]$ will be shown to play the main role for the construction, and after then the construction is more or less standard.  
%we require the following components.

\textcircled{1} Approximation of the Linear Klein-Gordon Flow. To diminish the magnitude of the Duhamel term $\mathcal N$, we employ the following approximation:
 \begin{align*}
\fe^{it\langle\nabla\rangle}=\fe^{it}\fe^{-\frac12 it\Delta}+O\big(t \Delta^2\big).
\end{align*}
The point is that each $\nabla$ introduces an additional factor of $\varepsilon$.
For a chosen time $t$ less than $\varepsilon^{-4}$, it is possible to omit the last part and deduce the linear Klein-Gordon flow to the linear Schrödinger flow.

\textcircled{2} Choosing Data for Linear Schrödinger Flow. Let us choose  a data for the  linear Schrödinger flow so that 
the solution is large in $L^\infty_tL^r_x$  but small in $L^q_tL^r_x$ for some $2\le q<+\infty$. That is, 
%Building on the aforementioned approximation, we can formulate the data to satisfy the conditions: 
$$
\big\|\psi\big\|_{L^\infty_tL^r_x(\R)}\sim 1;\quad 
\big\|\fe^{-\frac12 it\Delta}\psi\big\|_{L^q_tL^r_x(\R)}\ll 1. 
$$
Inspired by the work in \cite{CaWe-1992}, we set 
$$
\psi=\fe^{-ib\frac{|x|^2}{2}}f, \quad \mbox{for }\quad f\in S(\R^d), 
$$
which leads to 
$$
\big\|\fe^{-\frac12 it\Delta}\psi\big\|_{L^q_tL^r_x(\R^+)}=\big\|\fe^{-\frac12 it\Delta}f\big\|_{L^q_tL^r_x([0,\frac1b])}.
$$
By selecting a large value for $b$, we achieve smallness in $L^q_tL^r_x$ for $2\le q<+\infty$. However, the size remains large in $L^\infty_tL^r_x$. 

By employing these two strategies, we can systematically diminish the magnitude of the term $\mathcal N$ without altering the magnitude of $\phi_0+\big(\mathcal{S}_\varepsilon v^\varepsilon_0\big)^3$. Furthermore, the magnitude of the third term can be reduced significantly over an extended time period by applying the dispersive estimate.

(2) Counterexample in Theorem \ref{thm:main3-optimal}. The construction is built upon a solution structure that can be outlined as follows:
$$
\mathcal W_j(t)=\fe^{it\langle \nabla\rangle}\mathcal W_{j,0}
+\int_0^t \fe^{\mp i(t-s)\langle \nabla\rangle} \fe^{is}\mathcal{S}_\varepsilon\big(\partial_{tt}v\big)(s)\,ds
+\mathcal N(\mathcal W_j, \mathcal{S}_\varepsilon v),
$$
where $v$ is the solution of the nonlinear Schr\"odinger equation. By neglecting non-dominant terms, approximating nonlinear solutions with solutions of linear equations, and leveraging the scaling invariance of the linear flow, we arrive at the following approximate expression:
$$
\mathcal W_j(t)=\fe^{\mp it\langle \nabla\rangle}\int_0^t \fe^{is(\pm\langle \nabla\rangle+1-\frac12\Delta)} \mathcal{S}_\varepsilon(\Delta^2 v_0)\big)\,ds+\cdots.
$$

\textcircled{1} Non-Dominance of the Leftward Wave.  
In stark contrast to the remarkable performance of the leftward wave before, we shift our focus in this situation to harness the larger magnitude of the rightward wave. Specifically,   by suitably selecting the initial data, it will be demonstrated that the upper bound of the  {\it Leftward wave} $\mathcal R_1$ is surpassed by the lower bound of the {\it Rightward wave} $\mathcal R_2$.  As previously discussed, our investigation revolves around the exponential operator:
$$
\fe^{is(\langle \nabla\rangle+1-\frac12\Delta)}P_{\le 1},
$$
and ascertain the absence of critical points within its phase. This discovery unveils the non-resonant structure of the leftward wave $\mathcal W_1$, resulting in its smaller bound. However, this property ceases to hold when we examine the exponential operator:
$
\fe^{is(-\langle \nabla\rangle+1-\frac12\Delta)}.
$

\textcircled{2} Approximation of the Linear Flow. Turning our attention to the rightward wave $\mathcal W_2$, we utilize an approximation of the linear operator:
 \begin{align*}
\fe^{is(-\langle \nabla\rangle+1-\frac12\Delta)}=1+\frac18 is \Delta^2+O(s\Delta^3).
\end{align*}
This formula yields two independent dominant terms, originating  from $1$ and $\frac18 is \Delta^2.$ The last term can be effectively  controlled by $t\varepsilon^6$. Furthermore, we also employ the ``High-Low'' Decomposition to handle terms involving high derivatives.

%
%
%
%\begin{conj}[Global and uniform in time non-relativistic limits]\label{conj-2}
%Let    $T^*_\varepsilon\in (0, +\infty]$ be the maximal lifespan of the problem \eqref{eq:nls-wave}.
%There exists a space  $\mathcal Z_0$ and some $\alpha>0$, such that if
%$$
%(u_0,u_1)\in \big(\mathcal Z_0\big)^2.
%$$
%Then for any $t\in (0, T^*_\varepsilon)$, there exists a constant $C>0$ such that
%\begin{align}
%\big\|u(t)-\fe^{\frac{it}{\varepsilon^2}}v^{\varepsilon}-\fe^{-\frac{it}{\varepsilon^2}}\bar v^{\varepsilon} \big\|_{L^2_x}
%\le C \varepsilon^\alpha.
%\end{align}
%Either the larger the initial space $\mathcal Z_0$ or the higher the convergence order $\alpha$, the better.
%%where $C$ only depends on $\|(u_0,u_1)\|_{H^s\times H^{s-1}}$.
%\end{conj}
%Generally, one may set $\mathcal Z_0=H^\gamma(\R^d)$ for some $\gamma\ge 0$.

\subsection{Organization of the paper}
In Section \ref{sec:pre}, we provide notations, definitions, and relevant results. Additionally, we review existing findings on scattering behavior in the context of the nonlinear Schrödinger equation and nonlinear Klein-Gordon equations, which will serve as the basis for our subsequent analysis.
In Section \ref{sec:prop-key}, we introduce the rescaled complex expansion and define the leftward and rightward waves. We commence by establishing Strichartz-type estimates for both the leftward and rightward waves. Subsequently, we demonstrate the heightened preformation of the leftward wave. Utilizing these estimates, we ultimately present the proof of Proposition \ref{prop:v-varep}.
In Section \ref{sec:global}, we initiate with the proof of Theorem \ref{thm:main1-global}, focusing on the exploration of non-resonant structures. We then exhibit the optimal convergence rate by constructing counterexamples.
In Section \ref{sec:singe-modulate}, we provide the proof of Theorem \ref{thm:main3} through high-low frequency decomposition. Additionally, we establish the sharpness of both the convergence rate and the regularity requirement given in Theorem \ref{thm:main3-optimal}.
Finally, in Section \ref{sec:case1}, we furnish the proof for Theorem \ref{thm:main3-local-d2}.

\vskip 1.5cm

\section{Preliminary}\label{sec:pre}

\vskip .5cm
\subsection{Notations and Definition}\label{sec:notations}
%For any $a\in\R$, $a\pm:=a\pm\varepsilon$ for some small $\varepsilon>0$.
For any $z\in\C$, we define $\re z$ and $\im z$ as the real and imaginary part of $z$, respectively. Denote $\langle x\rangle=\sqrt{1+x^2}$.
%Let $\Omega\subset \R^d, a\in \R$, we denote $a\Omega=\{ax: x\in \Omega\}$.
Let $C>0$ denote a generic positive constant. If $f\loe C g$, we write $f\lsm g$ or $f=O(g)$.
%If $f\loe C g$ and $g\loe C f$, we write $f\sim g$.
%Suppose further that $C=C(a)$ depends on $a$, then we write $f\lsm_a g$ and $f\sim_a g$, respectively. If $f\loe 2^{-5}g$, we denote $f\ll g$ or $g\gg f$.

For short, we write $L^p_x=L^p_x(\R^d), H^s_x=H^s_x(\R^d)$, $L^q_tL^r_x(I\times \R^d)=L^q_tL^r_x(I)$ and  $L^q_tH^s_x(I\times \R^d)=L^q_tH^s_x(I)$. Moreover, we denote the norm 
$$
\|f\|_{X\cap Y}\triangleq \|f\|_X+\|f\|_Y. 
$$

 Let $\chi\in C_0^\I(\R^d)$ be a radial, real-valued and smooth cut-off function such that $\chi \goe 0$ satisfies
\EQ{
	\chi(\xi) \triangleq \left\{ \aligned
	&1\text{, when $|\xi|\le \frac12$,}\\
	&0\text{, when $|\xi|\ge 1$.}
	\endaligned
	\right.
}
Denote $\chi_{\le a}(\xi)=\chi(\frac{\xi}a)$ and $\chi_{a\le \cdot b}=\chi_{\le b}-\chi_{\le a}$.
We use $\wh f$ or $\F f$ to denote the Fourier transform of $f$:
\EQ{
\wh f(\xi)=\F f(\xi)\triangleq \int_{\R^d} e^{-ix\cdot\xi}f(x)\rm dx.
}
Moreover, we also need the following frequency cut-off operators: for any $N\in2^\N$,
\EQ{
P_{\loe N} f \triangleq  &\mathcal{F}^{-1}\left(\chi\Big(\frac{\xi}N\Big) \hat{f}(\xi)\right), \\
P_{> N} f \triangleq &f-P_{\loe N} f.
}

The following definition for "admissible exponent pair" will be used frequently.
\begin{defn}
For any $0\loe\gamma\loe1$, we call that the exponent pair $(q,r)\in\R^2$ is $\dot H^\ga$-$admissible$, if $\frac{2}{q}+\frac{d}{r}=\half d-\ga$, $2\loe q\loe\I$, $2\loe r\loe\I$, and $(q,r,d)\ne(2,\I,2)$. If $\ga=0$, we say that $(q,r)$ is $L^2$-$admissible$.
\end{defn}

\subsection{Useful lemmas}

 First, we recall the following Strichartz estimates, see \cite{KT98AJM} and the references therein.
\begin{lem}[Strichartz estimate]\label{lem:strichartz}
	Let $I\subseteq \R$. Suppose that $(q,r)$ and $(\wt{q},\wt{r})$ are $L_x^2$-admissible.
	Then,
	\EQ{
		\norm{ e^{it\langle\nabla\rangle}\ph}_{L_t^qL_x^r(I)} \lsm \norm{\langle\nabla\rangle^{\frac{d+2}2(\frac12-\frac1r)}\ph}_{L_x^2},
	}
	and
	\EQ{
		\normb{\int_0^t e^{i(t-s)\langle\nabla\rangle} F(s)\ds}_{L_t^qL_x^r(\R)} \lsm \norm{\langle\nabla\rangle^{\frac{d+2}2(1-\frac1r-\frac1{\wt{r}})}F}_{L_t^{\wt{q}'} L_x^{\wt{r}'}(\R)}.
	}
\end{lem}

\begin{remark}
It is known that there are some derivatives loss for the Strichartz estimates on the linear Klein-Gordon equations.
The loss of derivatives in these estimates pose a significant obstacle to our analysis.  In particular, the following estimates will be used frequently below: Suppose that
\begin{align*}
\phi(t)= \fe^{it\langle\nabla\rangle}\phi_0+ \int_0^t \fe^{i(t-s)\langle\nabla\rangle} \langle \nabla\rangle^{-1}F(s)\,ds,
\quad (t,x)\in \R\times \R^d.
\end{align*}
Then
\begin{align*}
\big\|\phi\big\|_{L^\frac{2(d+2)}{d}_{tx}(I)}+\big\|\langle\nabla\rangle^\frac12\phi\big\|_{L^\infty_tL^2_x(I)}
\lesssim
\big\|\langle\nabla\rangle^\frac12\phi_0\big\|_{L^2_x}+\big\|F\big\|_{L^\frac{2(d+2)}{d+4}_{tx}(I)}.
\end{align*}
\end{remark}

We also need the following dispersive estimates, see for example \cite{MaStWa-KG-decay}.
\begin{lem}[Dispersive estimate]\label{lem:dispersive}
	Let $2\le p\le +\infty$ and $N>0$.
	Then,
	\EQn{\label{eq:dispersive-1}
		\norm{e^{it\langle\nabla\rangle} P_N\ph}_{L_x^p} \lsm \langle t\rangle^{-d(\frac12-\frac1p)}\norm{\langle\nabla\rangle^{\frac{d+2}2(1-\frac2p)} P_N\ph}_{L_x^{p'}}.
	}
\end{lem}

The following are Coifman-Meyer's  Multiplier Theorem, see \cite{Coifman-Meyer}.
\begin{lem}
Let $m\in  L^\infty(\R^{nd})$ be smooth away from the origin and satisfy that for any multi-indices
$\vec{\alpha}=\{\alpha_1,\alpha_2,\cdots,\alpha_n\}\in \Z^{nd}$ with $\alpha_j\le 2d+1 $ and any $\xi_1,\cdots,\xi_n\in \R^d\setminus\{0\}$,
$$
\big|\partial_{\xi_1}^{\alpha_1}\partial_{\xi_2}^{\alpha_2}\cdots \partial_{\xi_n}^{\alpha_n}m(\xi_1,\xi_2,\cdots,\xi_n)\big|
\le C(\vec \alpha)(|\xi_1|+|\xi_2|+\cdots+|\xi_n|)^{-|\vec \alpha|}.
$$
Then for any $f_j,j=1,2,\cdots,n$,
\begin{align*}
&\left\|\int_{\xi=\xi_1+\cdots+\xi_n}  \fe^{i\xi\cdot x} m(\xi_1,\xi_2,\cdots,\xi_n)\widehat{f_1}(\xi_1)\widehat{f_2}(\xi_2)
\cdots\widehat{f_n}(\xi_n)\,d\xi_1d\xi_2\cdots d\xi_n\right\|_{L^p(\R^d)}\\
&\qquad\qquad \le C(p,p_1,\cdots,p_n)\|f_1\|_{L^{p_1}(\R^d)}\|f_2\|_{L^{p_2}(\R^d)}\cdots \|f_n\|_{L^{p_n}(\R^d)},
\end{align*}
where
$$
0< p< +\infty, 1<p_j\le+\infty \mbox{ for } j=1,2,\cdots,n, \frac1p=\frac1{p_1}+\frac1{p_2}+\cdots +\frac1{p_n}.
$$
\end{lem}
The Coifman-Meyer  Multiplier Theorem is reduced to the Mihlin-H\"ormander Multiplier Theorem when $n=1$ and $1<p<+\infty$.

The Kato-Ponce inequality will be frequently used in this paper. The result was originally proved in \cite{Kato-Ponce} and then extended to the endpoint case in \cite{BoLi-KatoPonce, Li-KatoPonce} recently.
\begin{lem}[Kato-Ponce inequality] \label{lem:kato-Ponce}
For $s>0$, $1<p\le \infty$, $1<p_1,p_3< \infty$ and $1<p_2,p_4 \le \infty$ satisfying $\frac1p=\frac1{p_1}+\frac1{p_2}$ and $\frac1p=\frac1{p_3}+\frac1{p_4}$, the following inequality holds:
% Let $f,g$ be the Schwartz functions. Then f:
\begin{align*}
\big\||\nabla|^s(fg)\big\|_{L^p}\le C\big( \||\nabla|^sf\|_{L^{p_1}}\|g\|_{L^{p_2}}+ \||\nabla|^sg\|_{L^{p_3}}\|f\|_{L^{p_4}}\big),
\end{align*}
where the constant $C>0$ depends on $s,p,p_1,p_2,p_3,p_4$. 
%If $s>\frac dp$ then the following inequality holds:
%\begin{align*}
%\big\||\nabla|^s(fg)\big\|_{L^p}\le C \|J^sf\|_{L^{p}}\|J^sg\|_{L^{p}},
%\end{align*}
%where the constant $C>0$ depends on $s$ and $p$.
\end{lem}
%
%
%\begin{cor}\label{lem:inhomo-Stri}
%	Let $I\subset \R$ and let $u$ be the solution to the following Klein-Gordon equation:
%	\EQ{
%	\left\{ \aligned
%	&\partial_{tt} u - \De u + u= F, \\
%	& u(0,x) = u_0(x),\quad  \partial_tu(0,x) = u_1(x).
%	\endaligned
%	\right.
%}	
%	Then,
%	\EQn{\label{eq:inhomo-Stri-1}
%		\norm{v}_{L^4_{tx}(I)} \lesssim \big\|(u_0,u_1)\big\|_{H^\frac12(\R^2)\times H^{-\frac12}(\R^2)}+\norm{F}_{L^\frac43_{tx}(I)};
%	}
%	and
%	\EQn{\label{eq:inhomo-Stri-2}
%		\norm{v}_{L^\frac{10}{3}_{tx}(I\times\R^3)}  \lesssim \big\|(u_0,u_1)\big\|_{H^\frac12(\R^3)\times H^{-\frac12}(\R^3)}+\norm{\langle \nabla\rangle^{-\frac13}F}_{L^\frac{10}{9}_{t}L^\frac{30}{17}_{t}(I\times\R^3)}.
%	}
%\end{cor}
%

\subsection{Some known results}

We will rely on the following well-known two results for the nonlinear Schr\"odinger equation and nonlinear Klein-Gordon equation. 
Scattering for the defocusing cubic nonlinear Schr\"odinger equation in two dimensional case was first proved by Killip, Tao and Visan \cite{KTV09JEMS} in radial case, and later on extended by Dodson \cite{Dod16Duke} to the non-radial case. In three dimensional case, it was proved by Ginibre and Velo \cite{GV92CMP}. In summary, we state them as: 
\begin{lem}[Scattering for NLS]\label{lem:spacetime-norm-NLS}
	Let $d=2,3$, $s\ge 0$ and $s_d=d-2$, and let $v_0\in H^s\cap H^{s_d}(\R^d)$ and $v$ be the solution of the defocusing cubic nonlinear Schr\"odinger equation \eqref{eq:nls}. Suppose that $(q,r)$ is $L_x^2$-admissible.
	Then there hold, 
	\EQ{
		\norm{|\nabla|^s v}_{L_t^qL_x^r(\R)} \le  C\big(\norm{v_0}_{H^{s_d}_x}\big)\norm{v_0}_{\dot H^s},\quad \mbox{ for any }s \ge 0.
	}
\end{lem}

The following large data scattering results for defocusing cubic nonlinear Klein-Gordon equation was proved by Killip, Stovall, Visan \cite{KillpVisan-KG}.
%
%\begin{lem}[Small data scattering for Klein-Gordon]\label{lem:spacetime-norm-KG-sd}
%	Let $d=2,3$, $s_c=\frac{d-1}{2}$,  $(u_0,u_1)\in H^1(\R^d)\times L^2(\R^d)$ and 	
%	\EQ{
%	\left\{ \aligned
%	&\partial_{tt} u - \De u + u= -|u|^2u, \\
%	& u(0,x) = u_0(x),\quad  \partial_tu(0,x) = u_1(x).
%	\endaligned
%	\right.
%}
%There exists some constant $\delta_0>0$, such that
%$$
%\big\|(u_0,u_1)\big\|_{H^1_x\times L^2_x}\le \delta_0,
%$$
%	then there exists a constant  $C>0$ such that
%	\EQn{\label{eq:strichartz-1}
%		\norm{u}_{X_{s_c}(\R)} \le  C\delta_0.
%	}
%\end{lem}
%

\begin{lem}[Large data scattering for Klein-Gordon]\label{lem:spacetime-norm-KG}
	Let $d=2,3$ and $(u_0,u_1)\in H^1(\R^d)\times L^2(\R^d)$. Suppose that $u$ is a solution of the defocusing  cubic nonlinear Klein-Gordon equation	
	\EQ{
	\left\{ \aligned
	&\partial_{tt} u - \De u + u= -|u|^2u, \\
	& u(0,x) = u_0(x),\quad  \partial_tu(0,x) = u_1(x).
	\endaligned
	\right.
}
	Then there exists a continuous function $C>0$ such that
	\EQ{
		\norm{u}_{X_{s_c}(\R)} \le  C\big(E(u_0,u_1)\big),
	}
	where
	$$
	E(u,u_t)\triangleq \big\|\nabla u\big\|_{L^2_x}^2+\big\|\partial_t u\big\|_{L^2_x}^2+\big\| u\big\|_{L^2_x}^2+\frac12\big\| u\big\|_{L^4_x}^4.
	$$
\end{lem}

We also need the following small data scattering result, which can be found in \cite{KillpVisan-KG}. Note that 
we put a slight strong assumption on the smallness of the initial data in the following lemma for a slight strong spacetime estimates.
\begin{lem}[Small data scattering for Klein-Gordon]\label{lem:spacetime-norm-KG-sd}
	Let $d=2,3, \mu\in\R$, $s_c=\frac{d-1}{2}$,  $(u_0,u_1)\in H^1(\R^d)\times L^2(\R^d)$ and 	
	\EQ{
	\left\{ \aligned
	&\partial_{tt} u - \De u + u= -\mu |u|^2u, \\
	& u(0,x) = u_0(x),\quad  \partial_tu(0,x) = u_1(x).
	\endaligned
	\right.
}
There exists some constant $\delta_0>0$, such that
$$
\big\|\langle\nabla\rangle^\frac12 (u_0,\langle\nabla\rangle^{-1}u_1)\big\|_{\dot H^{s_c}\times \dot H^{s_c}}\le \delta_0,
$$
	then there exists a constant  $C>0$ such that
	\EQ{
		\norm{u}_{X_{s_c}(\R)} \le  C\delta_0.
	}
\end{lem}

  \vskip 1.5cm

\section{Proof of Propositions \ref{prop:v-varep}}\label{sec:prop-key}
 \vskip .5cm

As delineated in the introduction, a trivial bound derived from the energy conservation law takes the form:
$$
\big\|\partial_t v^\varepsilon\big\|_{L^2_x}=O\big(\varepsilon^{-1}\big).
$$
However, this estimate falls short in capturing the true magnitude of the singularity of $\partial_t v^\varepsilon$ inherent in $\varepsilon$. Its weakness becomes evident, significantly limiting our ability to achieve a desired convergence estimate.

To address this issue, our first pivotal insight revolves around the notion of a {\it rescaled complex expansion}. This strategic concept proves invaluable in mitigating the impact of the singularity. By applying appropriate transformations to  \eqref{eq:nls-wave}, we effectively reduce it to a nonlinear Klein-Gordon equation endowed with specific initial data. This refined equation facilitates the derivation of a more favorable spacetime estimate that involves time derivatives, thus contributing to a more accurate understanding of the problem's behavior.

\subsection{Rescaled complex expansion}\label{sec:RCE}

We introduce the Navier-Stokes-like scaling
 $$
\mathcal{S}_\varepsilon f(t,x)\triangleq \varepsilon f(\varepsilon^2t,  \varepsilon x).
 $$
Denote
$$
h^{\varepsilon}(t,x)\triangleq \mathcal{S}_\varepsilon\big(v^{\varepsilon}\big)(t,x),
$$
then $h^{\varepsilon}$ obeys the following equation:
\EQn{\label{eqs:h-varep}
	\left\{ \aligned
	&\partial_{tt}h^{\varepsilon}+2i\partial_{t} h^{\varepsilon} - \De h^{\varepsilon}= - 3|h^{\varepsilon}|^2 h^{\varepsilon}, \\
	& h^{\varepsilon}(0,x) =\mathcal{S}_\varepsilon v_0(x),\quad \partial_th^{\varepsilon}(0,x) =\varepsilon^2 \mathcal{S}_\varepsilon v_1(x).
	\endaligned
	\right.
}
Then we shall prove the following result.
\begin{prop}\label{prop:h-varep}
Let $h^{\varepsilon}$ be the solution of \eqref{eqs:h-varep}. Then for any $\gamma\ge 0, \beta\in [0,\gamma]$ and any $(v_0,v_1)$ satisfying 
$$
(v_0,v_1)\in H^\gamma(\R^2)\times H^{\gamma-2}(\R^2),\quad 
(v_0,v_1)\in H^1(\R^2)\times L^2(\R^2),\quad \mbox{ and }\quad 
v_1\in H^{\beta-1}(\R^2),
$$ 
there hold
\begin{align}\label{est:h-varep-px}
\big\||\nabla|^{\gamma}& P_{\le 1}h^{\varepsilon}\big\|_{L^4_{tx}\cap L^\infty_tL^2_x(\R)}
+\big\||\nabla|^{\beta-\frac12}P_{>  1}h^{\varepsilon}\big\|_{L^4_{tx}\cap L^\infty_t H^\frac12_x(\R)}
\le 
C  \varepsilon^\gamma,
\end{align}
and
\begin{align}\label{est:h-varep-pt}
\big\| |\nabla|^\gamma\partial_th^{\varepsilon} \big\|_{L^4_{tx}(\R)}\le C \varepsilon^2,
\end{align}
where $C>0$ is only dependent of $\big\|(v_0,v_1)\big\|_{H^\gamma\times H^{\gamma-2}}$ and 
$\big\|(v_0,v_1)\big\|_{H^1_x\times L^2_x}$ and $\|v_1\|_{H^{\beta-1}}$.
\end{prop}

%\subsection{Proof of  Propositions \ref{prop:v-varep}}\label{prop:v-varep-proof}
%In this subsection,  the main contribution of this section is the proof of Proposition \ref{prop:v-varep}.

To prove the proposition, we define
\begin{align}\label{def:w-h}
w(t,x)=\fe^{it}h^{\varepsilon}(t,x),
\end{align}
then $w$ is governed by the following {\it complex expanded} nonlinear Klein-Gordon equation with specific initial data:
\EQn{
	\label{eqs:scaling-w}
	\left\{ \aligned
	&\partial_{tt}w-\De w+w=-3|w|^2w, \\
	& w(0)=\mathcal{S}_\varepsilon v_0,\quad
\partial_t w(0)= i\mathcal{S}_\varepsilon v_0+\varepsilon^2 \mathcal{S}_\varepsilon v_1.
	\endaligned
	\right.
}
We emphasis that the forms of $w(0)$  and $\partial_t w(0) $ are crucial for all analysis later.

For the solution $w$ to the equation \eqref{eqs:scaling-w}, it is clear that
\begin{align*}
E\big(w(0),\partial_t w(0)\big)&=\varepsilon^2\big\|\nabla v_0\big\|_{L^2_x}^2+\big\|v_0\big\|_{L^2_x}^2+\big\|iv_0+\varepsilon^2v_1\big\|_{L^2_x}^2+\frac32\varepsilon^2\big\|v_0\big\|_{L^4_x}^4\\
&\le C\big(\|(v_0,v_1)\|_{H^1_x\times L^2_x}\big).
\end{align*}
Then it follows from Lemma \ref{lem:spacetime-norm-KG} that
\begin{align}\label{est:w-spacetime}
\|w\|_{L^4_{tx}(\R)}\le C\big(\|(v_0,v_1)\|_{H^1_x\times L^2_x}\big).
\end{align}
%By scaling invariance, this implies that
%$$
%\|v\|_{L^4_{tx}(\R)}=\|w\|_{L^4_{tx}(\R)}\le C\big(\|v_0\|_{H^1_x}\big).
%$$
%We will improve \eqref{est:w-spacetime} so that

\subsection{Leftward and rightward waves}\label{sec:LR-wave}

Denote
\EQn{\label{relationship-w-W12}
	\left\{ \aligned
	&\mathcal W_1=\langle \nabla \rangle^{-1}\big(\partial_t-i \langle \nabla\rangle\big)w, \\
	&\mathcal W_2=\langle \nabla \rangle^{-1}\big(\partial_t+i \langle \nabla\rangle\big)w.
	\endaligned
	\right.
}
$\mathcal W_1, \mathcal W_2$ are referred to the {\it leftward} and  {\it rigthward waves} respectively.
It is easy to derive  that
\EQn{
	\label{eq:w-W12}
	\left\{ \aligned
	&w=\frac i2\big(\mathcal W_1-\mathcal W_2\big), \\
	&\partial_t w=\frac 12\langle\nabla\rangle\big(\mathcal W_1+\mathcal W_2\big).
	\endaligned
	\right.
}
Moreover, $\mathcal W_1$ and $\mathcal W_2$ satisfy
\EQn{
	\label{eq:W-12-KG}
	\left\{ \aligned
	&\big(\partial_t+i \langle \nabla\rangle\big)\mathcal W_1=-3\langle\nabla\rangle^{-1}\big(|w|^2w\big), \\
	&\big(\partial_t-i \langle \nabla\rangle\big)\mathcal W_2=-3\langle\nabla\rangle^{-1}\big(|w|^2w\big)
	\endaligned
	\right.
}
with the initial datum
\EQn{
	\label{eq:W-12-KG-intialdatum}
	\left\{ \aligned
&\mathcal W_1(0)=\mathcal W_{1,0}\triangleq i\langle \nabla \rangle^{-1}\big(1- \langle \nabla\rangle\big)\mathcal{S}_\varepsilon v_0+\varepsilon^2\langle \nabla \rangle^{-1}\mathcal{S}_\varepsilon v_1;\\
&\mathcal W_2(0)=\mathcal W_{2,0}\triangleq i\langle \nabla \rangle^{-1}\big(1+ \langle \nabla\rangle\big)\mathcal{S}_\varepsilon v_0+\varepsilon^2\langle \nabla \rangle^{-1}\mathcal{S}_\varepsilon v_1.
	\endaligned
	\right.
}
We emphasis that the form of $\mathcal W_{1,0}$ is crucial, which is one of the main reasons that we introduce the complex expansion of $\varphi=\mathcal{S}_\varepsilon(v^\varepsilon)$.

\subsection{Strichartz estimates}\label{subsec:Stri}

Firstly of all, we aim to show the following Strichartz estimates for the leftward and rightward waves:
\begin{lem}
Under the same hypothesis as Proposition \ref{prop:h-varep}, then 
\begin{align}\label{est:mathcal-W12}
\big\||\nabla|^{\gamma}& P_{\le 1}\mathcal W_j\big\|_{L^4_{tx}\cap L^\infty_tL^2_{x} (\R)}
+\big\||\nabla|^{\beta-\frac12} P_{>  1}\mathcal W_j\big\|_{L^4_{tx}\cap L^\infty_t H^\frac12_{x}(\R)}
%+\big\||\nabla|^{\gamma}\mathcal W_j\big\|_{L^\infty_tL^2_{x}(\R)}
\le 
C \varepsilon^\gamma,
\end{align}
where $C>0$ is only dependent of $\big\|(v_0,v_1)\big\|_{H^\gamma\times H^{\gamma-2}}$ and 
$\big\|(v_0,v_1)\big\|_{H^1_x\times L^2_x}$ and $\|v_1\|_{H^{\beta-1}}$.
\end{lem}
\begin{proof}
We first consider the low-frequency part.  By \eqref{eq:W-12-KG} and Duhamel's formula,
$$
P_{\le 1}\mathcal W_j(t)=\fe^{\mp i(t-t_0)\langle \nabla\rangle}P_{\le 1} \mathcal W_j(t_0)
-3\int_{t_0}^t \fe^{\mp i(t-s)\langle \nabla\rangle}\langle\nabla \rangle^{-1} P_{\le 1} \big(|w|^2w\big)\,ds.
$$
By Lemma \ref{lem:strichartz}, we have that for any $I=[t_0,t_1]\subset \R^+$ (the negative time direction can be treated in the same way),
\begin{align}\label{est:W-12-low-1}
\big\||\nabla|^{\gamma} P_{\le 1}\mathcal W_j\big\|_{L^4_{tx}\cap L^\infty_t L^2_x(I)}
\lesssim &
\big\||\nabla|^{\gamma} P_{\le 1}\mathcal W_j(t_0)\big\|_{L^2_x}
+
\big\||\nabla|^{\gamma} P_{\le 1}\big(|w|^2w\big)\big\|_{L^\frac43_{tx}(I)} .
\end{align}
Then we estimate the second term above.

  It is obvious that
\begin{align*}
\big\||\nabla|^{\gamma} P_{\le 1}\big(|w|^2w\big)\big\|_{L^\frac43_{tx}(I)}
\le & \Big\||\nabla|^{\gamma} P_{\le 1}\left(|P_{\le 1} w|^2P_{\le 1} w\right)\Big\|_{L^\frac43_{tx}(I)}\\
&\quad +\Big\||\nabla|^{\gamma} P_{\le 1}\Big(|w|^2w-\left(|P_{\le 1} w|^2P_{\le 1} w\right)\Big)\Big\|_{L^\frac43_{tx}(I)}\\
\lesssim & \Big\||\nabla|^{\gamma} \left(|P_{\le 1} w|^2P_{\le 1} w\right)\Big\|_{L^\frac43_{tx}(I)}\\
&\quad +\Big\||w|^2w-\left(|P_{\le 1} w|^2P_{\le 1} w\right)\Big\|_{L^\frac43_{tx}(I)}.
\end{align*}
Note that
$$
|w|^2w-\left(|P_{\le 1} w|^2P_{\le 1} w\right)=O(P_{>1}w\cdot w^2),
$$
where the terms in $O(fgh)$ is the cubic combination of $f,g,h$ and their complex conjugate.  Therefore, by Bernstein's and Kato-Ponce's inequalities we further have that
\begin{align*}
\big\||\nabla|^{\gamma} P_{\le 1}\big(|w|^2w\big)\big\|_{L^\frac43_{tx}(I)}
\lesssim &
\big\||\nabla|^\gamma P_{\le 1} w \big\|_{L^4_{tx}(I)}\big\|w\big\|_{L^4_{tx}(I)}^2
+ \big\| P_{> 1} w \big\|_{L^4_{tx}(I)}\big\|w\big\|_{L^4_{tx}(I)}^2.
%\\
%\lesssim &
%\big\||\nabla|^\gamma P_{\le 1} w \big\|_{L^4_{tx}(I)}\big\|w\big\|_{L^4_{tx}(I)}^2\\
%&\quad+ \big\| P_{> 1} |\nabla|^{\gamma-\frac12} w \big\|_{L^4_{tx}(I)}\big\|w\big\|_{L^4_{tx}(I)}^2.
\end{align*}
Then by \eqref{eq:w-W12}, it gives that
\begin{align*}
\big\||\nabla|^{\gamma} P_{\le 1}\big(|w|^2w\big)\big\|_{L^\frac43_{tx}(I)}
\lesssim &
\sum\limits_{j=1,2}\Big(\big\||\nabla|^\gamma P_{\le 1} \mathcal W_j \big\|_{L^4_{tx}(I)}+ \big\| P_{> 1} \mathcal W_j \big\|_{L^4_{tx}(I)}\Big)  \big\|w\big\|_{L^4_{tx}(I)}^2.
\end{align*}
Therefore, this estimate combining with \eqref{est:W-12-low-1}, yields that there exist some absolute postive constants $C_0,C_1$, such that
\begin{align}\label{est:W-12-low}
\big\||\nabla|^{\gamma} & P_{\le 1}\mathcal W_j\big\|_{L^4_{tx}\cap L^\infty_t L^2_x(I)}
\le
 C_0\big\||\nabla|^{\gamma} P_{\le 1}\mathcal W_j(t_0)\big\|_{L^2_x} \notag\\
& +
C_1 \sum\limits_{j=1,2}\Big(\big\||\nabla|^\gamma P_{\le 1} \mathcal W_j \big\|_{L^4_{tx}(I)}+ \big\| P_{> 1} \mathcal W_j \big\|_{L^4_{tx}(I)}\Big)
\big\|w\big\|_{L^4_{tx}(I)}^2.
\end{align}

Next, we  consider the high-frequency part.   We first consider the case when $\beta\ge \frac12$.
%, otherwise, it can be reduced to the case of $\gamma=\frac12$ by Bernstein's inequality.
%$$
%\big\||\nabla|^{\gamma-\frac12} P_{> 1}\mathcal W_j\big\|_{L^4_{tx}(I)}\lesssim
%\big\|P_{> 1}\mathcal W_j\big\|_{L^4_{tx}(I)}\lesssim
%$$
By \eqref{eq:W-12-KG} and Duhamel's formula,
$$
P_{> 1}\mathcal W_j(t)=\fe^{\mp i(t-t_0)\langle \nabla\rangle}P_{> 1} \mathcal W_j(t_0)
-3\int_{t_0}^t \fe^{\mp i(t-s)\langle \nabla\rangle}\langle\nabla \rangle^{-1} P_{> 1} \big(|w|^2w\big)\,ds.
$$
Then by Lemma \ref{lem:strichartz}, we have that for any $\beta\in [0,\gamma]$, 
\begin{align}\label{est:W-12-high-1}
&\big\||\nabla|^{\beta-\frac12} P_{> 1}\mathcal W_j\big\|_{L^4_{tx}(I)}
+\big\||\nabla|^{\beta} P_{> 1}\mathcal W_j\big\|_{L^\infty_{t}L^2_x(I)}\notag\\
\lesssim &
\big\||\nabla|^{\beta-\frac12} \langle \nabla\rangle^\frac12 P_{> 1}\mathcal W_j(t_0)\big\|_{L^2_x}
+
\big\||\nabla|^{\beta-\frac12} P_{> 1}\big(|w|^2w\big)\big\|_{L^\frac43_{tx}(I)} \notag\\
\lesssim &
\big\||\nabla|^{\beta}  P_{> 1}\mathcal W_j(t_0)\big\|_{L^2_x}
+
\big\||\nabla|^{\beta-\frac12} P_{> 1}\big(|w|^2w\big)\big\|_{L^\frac43_{tx}(I)} .
\end{align}
Then we estimate the second term above.

   Note that
$$
P_{> 1}\big(|w|^2w\big)=P_{> 1}O\big(P_{\gtrsim 1}w\cdot w^2\big).
$$
Then it infers that
\begin{align*}
\big\||\nabla|^{\beta-\frac12} P_{> 1}\big(|w|^2w\big)\big\|_{L^\frac43_{tx}(I)}
\lesssim  & \Big\||\nabla|^{\beta-\frac12} P_{\gtrsim 1}w\Big\|_{L^4_{tx}(I)}\|w\|_{L^4_{tx}(I)}^2\\
&\quad +\Big\|P_{\gtrsim 1}w\Big\|_{L^4_{tx}(I)}\Big\||\nabla|^{\beta-\frac12} w\Big\|_{L^4_{tx}(I)}\|w\|_{L^4_{tx}(I)}.
\end{align*}
By Bernstein's inequality, we have that
\begin{align*}
\big\|P_{\gtrsim 1}w\big\|_{L^4_{tx}(I)}+&
\Big\||\nabla|^{\beta-\frac12} P_{\gtrsim 1}w\Big\|_{L^4_{tx}(I)}\\
 & \lesssim \Big\||\nabla|^{\beta-\frac12} P_{> 1}w\Big\|_{L^4_{tx}(I)}
+ \Big\||\nabla|^{\gamma} P_{\le 1}w\Big\|_{L^4_{tx}(I)};
\end{align*}
and
\begin{align*}
\Big\||\nabla|^{\beta-\frac12} w\Big\|_{L^4_{tx}(I)}
& \lesssim \Big\||\nabla|^{\beta-\frac12} P_{> 1}w\Big\|_{L^4_{tx}(I)}
+ \|w\|_{L^4_{tx}(I)}.
\end{align*}
Applying these two estimates, we further obtain that
\begin{align*}
\big\||\nabla|^{\beta-\frac12} P_{> 1}\big(|w|^2w\big)\big\|_{L^\frac43_{tx}(I)}
\lesssim  &\Big(\big\||\nabla|^\gamma P_{\le 1} w \big\|_{L^4_{tx}(I)}
+ \big\| P_{> 1} |\nabla|^{\beta-\frac12} w \big\|_{L^4_{tx}(I)}\Big) \big\|w\big\|_{L^4_{tx}(I)}^2.
\end{align*}
Then by \eqref{eq:w-W12}, it drives that
\begin{align*}
\big\||\nabla|^{\beta-\frac12} P_{> 1}\big(|w|^2w\big)\big\|_{L^\frac43_{tx}(I)}
\lesssim &
\sum\limits_{j=1,2}\Big(\big\||\nabla|^\gamma P_{\le 1} \mathcal W_j \big\|_{L^4_{tx}(I)}+ \big\| P_{> 1} |\nabla|^{\beta-\frac12} \mathcal W_j \big\|_{L^4_{tx}(I)}\Big)  \big\|w\big\|_{L^4_{tx}(I)}^2.
\end{align*}
This together with  \eqref{est:W-12-high-1}, yields that
\begin{align}\label{est:W-12-high}
&\big\||\nabla|^{\beta-\frac12} P_{> 1}\mathcal W_j\big\|_{L^4_{tx}(I)}
+\big\||\nabla|^{\beta} P_{> 1}\mathcal W_j\big\|_{L^\infty_{t}L^2_x(I)}\notag\\
\le &
 C_0\big\||\nabla|^{\beta} P_{> 1}\mathcal W_j(t_0)\big\|_{L^2_x} \notag\\
& +
C_1 \sum\limits_{j=1,2}\Big(\big\||\nabla|^\gamma P_{\le 1} \mathcal W_j \big\|_{L^4_{tx}(I)}+ \big\| P_{> 1} |\nabla|^{\beta-\frac12} \mathcal W_j \big\|_{L^4_{tx}(I)}\Big)
\big\|w\big\|_{L^4_{tx}(I)}^2.
\end{align}
Here the constants $C_0,C_1$ may vary with the ones in \eqref{est:W-12-low}, however, we still use the same notations since it is not essential in our analysis.

Combining with \eqref{est:W-12-low}, \eqref{est:W-12-high} and Bernstein's inequality, we have that
\begin{align*}
&\big\||\nabla|^{\gamma} P_{\le 1}\mathcal W_j\big\|_{L^4_{tx}\cap L^\infty_t L^2_x(I)}
+
\big\||\nabla|^{\beta-\frac12} P_{> 1}\mathcal W_j\big\|_{L^4_{tx}(I)}
+\big\||\nabla|^{\beta} P_{> 1}\mathcal W_j\big\|_{L^\infty_{t}L^2_x(I)}
\\
\le &
 2C_0\Big(\big\||\nabla|^{\gamma} P_{\le 1}\mathcal W_j(t_0)\big\|_{L^2_x}+\big\||\nabla|^{\beta}P_{>1} \mathcal W_j(t_0)\big\|_{L^2_x}\Big) \\
& +
2C_1 \sum\limits_{j=1,2}\Big(\big\||\nabla|^\gamma P_{\le 1} \mathcal W_j \big\|_{L^4_{tx}(I)}+ \big\| P_{> 1} |\nabla|^{\beta-\frac12} \mathcal W_j \big\|_{L^4_{tx}(I)}\Big)
\big\|w\big\|_{L^4_{tx}(I)}^2.
\end{align*}
Therefore, denote $X(I)$ to be
$$
X(I)
= \sum\limits_{j=1,2}\Big(\big\||\nabla|^{\gamma} P_{\le 1}\mathcal W_j\big\|_{L^4_{tx}\cap L^\infty_t L^2_x(I)}
+
\big\||\nabla|^{\beta-\frac12} P_{> 1}\mathcal W_j\big\|_{L^4_{tx}(I)}
+\big\||\nabla|^{\beta} P_{> 1}\mathcal W_j\big\|_{L^\infty_{t}L^2_x(I)}\Big),
$$
then
\begin{align}\label{est:W-12}
X(I)
\le
 4C_0\Big(\big\||\nabla|^{\gamma} P_{\le 1}\mathcal W_j(t_0)\big\|_{L^2_x}+\big\||\nabla|^{\beta}P_{>1} \mathcal W_j(t_0)\big\|_{L^2_x}\Big)
 +
4C_1
\big\|w\big\|_{L^4_{tx}(I)}^2\cdot  X(I).
\end{align}
Note that we have the uniform boundedness of $\|w\|_{L^4_{tx}(I)}$ as in \eqref{est:w-spacetime}.
Hence, there exists a constant $K=K(C_1)$ and a sequence of time intervals
$$
\bigcup\limits_{k=0}^K I_k=\R^+,
$$
 with
$$
I_0=[0, t_1], \quad I_k=(t_k, t_{k+1}] \mbox{ for } k=1,\cdots, K-1,  \quad I_{K}=[t_K, +\infty),
$$
such that for any $1\le k\le K$,
\begin{align}\label{subinterval-length-wIk}
4C_1
\big\|w\big\|_{L^4_{tx}(I_k)}^2
\le \frac12.
\end{align}
This together with \eqref{est:W-12} infers that
\begin{align}\label{est:X-Ik}
X(I_k)
\le
 8C_0\Big(\big\||\nabla|^{\gamma} P_{\le 1}\mathcal W_j(t_k)\big\|_{L^2_x}+\big\||\nabla|^{\beta}P_{>1} \mathcal W_j(t_k)\big\|_{L^2_x}\Big).
\end{align}
Furthermore, combining with \eqref{est:W-12-low} and \eqref{est:W-12-high}, we have that
\begin{align*}
\big\||\nabla|^{\gamma} P_{\le 1}&\mathcal W_j\big\|_{L^\infty_tL^2_x(I_k)}+\big\||\nabla|^{\beta}P_{>1} \mathcal W_j\big\|_{L^\infty_tL^2_x(I_k)}\\
\le &
 2C_0\Big(\big\||\nabla|^{\gamma} P_{\le 1}\mathcal W_j(t_k)\big\|_{L^2_x}+\big\||\nabla|^{\beta}P_{>1} \mathcal W_j(t_k)\big\|_{L^2_x}\Big)
  +
2C_1 X(I_k)
\big\|w\big\|_{L^4_{tx}(I_k)}^2.
\end{align*}
Inserting \eqref{subinterval-length-wIk} and \eqref{est:X-Ik} into the estimate above, we get that
\begin{align}
\label{Wj-tk-Linfty}
\big\||\nabla|^{\gamma} P_{\le 1}&\mathcal W_j\big\|_{L^\infty_tL^2_x(I_k)}+\big\||\nabla|^{\beta}P_{>1} \mathcal W_j\big\|_{L^\infty_tL^2_x(I_k)}\notag\\
\le &
 4C_0\Big(\big\||\nabla|^{\gamma} P_{\le 1}\mathcal W_j(t_k)\big\|_{L^2_x}+\big\||\nabla|^{\beta}P_{>1} \mathcal W_j(t_k)\big\|_{L^2_x}\Big).
\end{align}
In particular, the iteration implies that for any $1\le k\le K$,
\begin{align}
\label{Wj-tk-iteration}
\big\||\nabla|^{\gamma} P_{\le 1}&\mathcal W_j\big\|_{L^\infty_tL^2_x(I_k)}+\big\||\nabla|^{\beta}P_{>1} \mathcal W_j\big\|_{L^\infty_tL^2_x(I_k)}\notag\\
\le &
 (4C_0)^k\Big(\big\||\nabla|^{\gamma} P_{\le 1}\mathcal W_{j,0}\big\|_{L^2_x}+\big\||\nabla|^{\beta}P_{>1} \mathcal W_{j,0}\big\|_{L^2_x}\Big).
\end{align}
Therefore, inserting \eqref{Wj-tk-iteration} into \eqref{est:X-Ik} and \eqref{Wj-tk-Linfty}, we obtain that
\begin{align}\label{est:Wj-IKi-fin}
X(I_k)
%+\big\||\nabla|^{\gamma}\mathcal W_j\big\|_{L^\infty_tL^2_x(I_k)}
\le  4(4C_0)^{k+1}\Big(\big\||\nabla|^{\gamma} P_{\le 1}\mathcal W_{j,0}\big\|_{L^2_x}+\big\||\nabla|^{\beta}P_{>1} \mathcal W_{j,0}\big\|_{L^2_x}\Big).
\end{align}
Note that
\begin{align*}
\big\||\nabla|^{\gamma} P_{\le 1}\mathcal W_{j,0}\big\|_{L^2_x}&+\big\||\nabla|^{\beta}P_{>1} \mathcal W_{j,0}\big\|_{L^2_x}\\
%\lesssim &
%\big\||\nabla|^{\gamma}\mathcal{S}_\varepsilon v_0\big\|_{L^2_x}
%+\varepsilon^2\big\||\nabla|^{\gamma}\langle \nabla\rangle^{-1}\mathcal{S}_\varepsilon v_1\big\|_{L^2_x}\\
\lesssim &
\big\||\nabla|^{\gamma}\mathcal{S}_\varepsilon v_0\big\|_{L^2_x}
+\varepsilon^2\Big(\big\||\nabla|^\gamma P_{\le 1} \mathcal{S}_\varepsilon v_1\big\|_{L^2_x}+\big\||\nabla|^{\beta-1}P_{>1}\mathcal{S}_\varepsilon v_1\big\|_{L^2_x}\Big).
\end{align*}
If $\gamma\le 2$, it is controlled by 
$$
 \varepsilon^\gamma \|(v_0,v_1)\|_{H^\gamma\times (H^{\beta-1}\cap L^2)}.
$$
If $\gamma> 2$, it is controlled by 
$$
 \varepsilon^\gamma \|(v_0,v_1)\|_{H^\gamma\times (H^{\beta-1}\cap H^{\gamma-2})}.
$$
Therefore, these combining with \eqref{est:Wj-IKi-fin} implies \eqref{est:mathcal-W12} in the case of $\frac12\le \beta\le \gamma$. 

When $\beta<\frac12$, by Bernstein's inequality, we have that 
\begin{align*}
\big\||\nabla|^{\beta-\frac12} P_{> 1}\mathcal W_j\big\|_{L^4_{tx}(I)}
+\big\||\nabla|^{\beta} P_{> 1}\mathcal W_j\big\|_{L^\infty_{t}L^2_x(I)}
\le  &
\big\|P_{> 1}\mathcal W_j\big\|_{L^4_{tx}(I)}
+\big\||\nabla|^{\frac12} P_{> 1}\mathcal W_j\big\|_{L^\infty_{t}L^2_x(I)}.
\end{align*}
Then by \eqref{est:mathcal-W12} for $\beta=\frac12$, it gives that 
\begin{align*}
\big\||\nabla|^{\beta-\frac12} P_{> 1}\mathcal W_j\big\|_{L^4_{tx}(I)}
+\big\||\nabla|^{\beta} P_{> 1}\mathcal W_j\big\|_{L^\infty_{t}L^2_x(I)}
\le  &
C\varepsilon^\gamma.
\end{align*}
When $\gamma<\frac12$, by \eqref{est:W-12-low}, we have that 
\begin{align*} 
\big\||\nabla|^{\gamma} & P_{\le 1}\mathcal W_j\big\|_{L^4_{tx}\cap L^\infty_t L^2_x(I)}
\le
 C_0\big\||\nabla|^{\gamma} P_{\le 1}\mathcal W_j(t_0)\big\|_{L^2_x} \notag\\
& +
C_1 \sum\limits_{j=1,2}\big\||\nabla|^\gamma P_{\le 1} \mathcal W_j \big\|_{L^4_{tx}(I)}
\big\|w\big\|_{L^4_{tx}(I)}^2
+C_2\varepsilon^\frac12,
\end{align*}
where $C_j,j=1,2,3$ are the constants only dependent of $\|(v_0,v_1)\|_{H^1_x\times L^2_x}$.
Then arguing similarly as above, we obtain that 
\begin{align*} 
\big\||\nabla|^{\gamma} & P_{\le 1}\mathcal W_j\big\|_{L^4_{tx}\cap L^\infty_t L^2_x(I)}
\le
 C\varepsilon^\gamma.
\end{align*}
This proves \eqref{est:mathcal-W12} in the cases  of $\beta<\frac12$  and $\gamma<\frac12$, and thus finishes the proof of the lemma. 
\end{proof}

Now combining \eqref{est:mathcal-W12} with \eqref{eq:w-W12} yields that
\begin{align}\label{est:w-X0}
\big\||\nabla|^{\gamma}& P_{\le 1}w\big\|_{L^4_{tx}\cap L^\infty_tL^2_{x} (\R)}
+\big\||\nabla|^{\beta-\frac12} P_{>  1}w\big\|_{L^4_{tx}\cap L^\infty_t H^\frac12_{x}(\R)}
\le 
C \varepsilon^\gamma;
\end{align}
and
\begin{align}\label{est:wt-X0}
\big\||\nabla|^{\gamma}& P_{\le 1}\partial_t w\big\|_{L^4_{tx}\cap L^\infty_tL^2_{x} (\R)}
+\big\||\nabla|^{\beta-\frac12} P_{>  1}\partial_t w\big\|_{L^4_{tx}\cap L^\infty_t H^\frac12_{x}(\R)}
\le 
C \varepsilon^\gamma.
\end{align}
%where  $C$ only depends on $\| v_0\|_{H^\gamma_x}$, $\| v_1\|_{H^{\beta-1}_x}$ and $\|(v_0,v_1)\|_{H^1_x\times L^2_x}$.

%
%\begin{align}\label{est:w-wt}
%\big\||\nabla|^{\gamma}P_{\le 1} w\big\|_{X_{s_c}(\R)}&+\big\||\nabla|^{\gamma}\partial_t w\big\|_{X_{s_c}(\R)}\notag\\
%&+\big\||\nabla|^{\gamma}w\big\|_{L^\infty_t L^2_x(\R)}+\big\||\nabla|^{\gamma}\partial_t w\big\|_{L^\infty_t L^2_x(\R)}
%\lesssim
%C(\|v_0\|_{H^\alpha_x}) \varepsilon^\gamma.
%\end{align}
Scaling back, we obtain that for any $\gamma>0$,
\begin{align}
&\big\||\nabla|^{\gamma}v^{\varepsilon}\big\|_{X_0(\R)}
\le 
C;\label{est:v-gamma}
\\
&\big\||\nabla|^{\gamma}\partial_t v^{\varepsilon}\big\|_{X_0(\R)}\le
C \varepsilon^{-2}.\label{est:vt-gamma}
\end{align}
The estimate \eqref{est:v-gamma} is exactly required in Proposition \ref{prop:v-varep} for \eqref{est:v-varep-px}. However, \eqref{est:vt-gamma}  is not enough for  \eqref{est:v-varep-pt} and thus for what we need to estimates  the error term \eqref{est:main1-global}.

\subsection{An improved estimates for leftward wave}\label{subsec:Imp-Stri}

To improve the estimate \eqref{est:vt-gamma}, our key observation is that the behavior of $\mathcal W_1$ is better than  that of $\mathcal W_2$. In the following, we shall prove
\begin{lem}
Let $\gamma\ge 0$. Suppose that $(v_0,v_1)\in H^{\gamma+2}\times H^{\gamma}$, then
\begin{align}\label{est:W1-gamma}
\big\||\nabla|^{\gamma}\mathcal W_1\big\|_{L^4_{tx}(\R)}+\big\||\nabla|^{\gamma}\langle\nabla\rangle^\frac12\mathcal W_1\big\|_{L^\infty_t L^2_x(\R)}
\le
C\varepsilon^{2+\gamma},
\end{align}
where  $C$ only depends on $\| v_0\|_{H^{\gamma+2}_x}$ and $\| v_1\|_{H^{\gamma}_x}$.
\end{lem}
\begin{proof}
To illustrate \eqref{est:W1-gamma}, we denote
$$
\varphi_1=\fe^{it\langle\nabla\rangle}\mathcal W_1;\quad
\varphi_2=\fe^{-it\langle\nabla\rangle}\mathcal W_2.
$$
Then we apply the Duhamel formula to rewrite  the equation \eqref{eq:W-12-KG}, in terms of $\varphi_1$, as
\begin{align}\label{duhamel-varphi1-1}
\varphi_1(t)=&\varphi_1(t_0)-3\int_{t_0}^t \fe^{is\langle\nabla\rangle}\langle\nabla\rangle^{-1}\big(|w(s)|^2w(s)\big).
\end{align}
According to \eqref{eq:w-W12}, we write
\begin{align}\label{nonlinear-w}
|w|^2w=&-\frac i8\big|\mathcal W_2\big|^2\mathcal W_2+N_1\big(\mathcal W_1,\mathcal W_2\big)\notag\\
=&-\frac i8P_{\le 1}\Big(\big|P_{\le 1}\mathcal W_2\big|^2P_{\le 1}\mathcal W_2\Big)+N_1\big(\mathcal W_1,\mathcal W_2\big)
+N_2\big(\mathcal W_2\big),
\end{align}
where
\begin{align*}
N_1\big(\mathcal W_1,\mathcal W_2\big)\triangleq & \frac i8\Big(\big|\mathcal W_1\big|^2\mathcal W_1-\mathcal W_1^2\overline{\mathcal W_2}-2\big|\mathcal W_1\big|^2\mathcal W_2+\mathcal W_2^2\overline{\mathcal W_1}+2\big|\mathcal W_2\big|^2\mathcal W_1\Big);\\
N_2\big(\mathcal W_2\big)\triangleq &- \frac i8\big|\mathcal W_2\big|^2\mathcal W_2+\frac i8P_{\le 1}\Big(\big|P_{\le 1}\mathcal W_2\big|^2P_{\le 1}\mathcal W_2\Big).
\end{align*}
Then by \eqref{nonlinear-w}, we further rewrite the formula \eqref{duhamel-varphi1-1} as
\begin{subequations}\label{duhamel-varphi1-2}
\begin{align}
\varphi_1(t)=&\varphi_1(t_0)-\frac 38i\int_{t_0}^t \fe^{is\langle\nabla\rangle}\langle\nabla\rangle^{-1} P_{\le 1}\Big(\big|P_{\le 1}\mathcal W_2\big|^2P_{\le 1}\mathcal W_2\Big)\,ds\label{duhamel-varphi1-2-1}\\
&\quad -3\int_{t_0}^t \fe^{is\langle\nabla\rangle}\langle\nabla\rangle^{-1} \Big(N_1\big(\mathcal W_1,\mathcal W_2\big)
+N_2\big(\mathcal W_2\big)\Big)\,ds.\label{duhamel-varphi1-2-2}
\end{align}
\end{subequations}

We first consider the term \eqref{duhamel-varphi1-2-1}. We will see that  the inherent structure inside is crucial to produce small parameters, which is another reason that we introduce the complex expansion of $h^\varepsilon=\mathcal{S}_\varepsilon(v^\varepsilon)$. Taking Fourier transform, we have
\begin{align*}
\widehat{\eqref{duhamel-varphi1-2-1}}(\xi)
=&
\frac 38i\int_{t_0}^t \int_{\xi_1+\xi_2+\xi_3=\xi}\fe^{is\big(\langle\xi\rangle+\langle\xi_1\rangle+\langle\xi_2\rangle-\langle\xi_3\rangle\big)}\langle\xi\rangle^{-1} \chi_{\le 1}(\xi)\chi_{\le 1}(\xi_1)\chi_{\le 1}(\xi_2)\chi_{\le 1}(\xi_3)\\
&\qquad\qquad \cdot\hat \varphi_2(\xi_1)\hat \varphi_2(\xi_2)\widehat{\overline{\varphi_2}}(\xi_3)\,ds d\xi_1 d\xi_2.
\end{align*}
Note that
\begin{align*}
\fe^{is\big(\langle\xi\rangle+\langle\xi_1\rangle+\langle\xi_2\rangle-\langle\xi_3\rangle\big)}=\frac{d}{ds}\Big(\fe^{is\big(\langle\xi\rangle+\langle\xi_1\rangle+\langle\xi_2\rangle-\langle\xi_3\rangle\big)}\Big) \frac{1}{i(\langle\xi\rangle+\langle\xi_1\rangle+\langle\xi_2\rangle-\langle\xi_3\rangle)}.
\end{align*}
Then using integration by parts, we further get that
\begin{align*}
\widehat{\eqref{duhamel-varphi1-2-1}}(\xi)
=&
\frac 38 \int_{\xi_1+\xi_2+\xi_3=\xi}\fe^{is\big(\langle\xi\rangle+\langle\xi_1\rangle+\langle\xi_2\rangle-\langle\xi_3\rangle\big)}\langle\xi\rangle^{-1}\frac{ \chi_{\le 1}(\xi)\chi_{\le 1}(\xi_1)\chi_{\le 1}(\xi_2)\chi_{\le 1}(\xi_3)}{\langle\xi\rangle+\langle\xi_1\rangle+\langle\xi_2\rangle-\langle\xi_3\rangle}\\
&\qquad\qquad \cdot\hat \varphi_2(\xi_1)\hat \varphi_2(\xi_2)\widehat{\overline{\varphi_2}}(\xi_3)\,d\xi_1d\xi_2\Big|_{t_0}^t\\
&-\frac 38\int_{t_0}^t \int_{\xi_1+\xi_2+\xi_3=\xi}\fe^{is\big(\langle\xi\rangle+\langle\xi_1\rangle+\langle\xi_2\rangle-\langle\xi_3\rangle\big)}\langle\xi\rangle^{-1}\frac{ \chi_{\le 1}(\xi)\chi_{\le 1}(\xi_1)\chi_{\le 1}(\xi_2)\chi_{\le 1}(\xi_3)}{\langle\xi\rangle+\langle\xi_1\rangle+\langle\xi_2\rangle-\langle\xi_3\rangle}\\
&\qquad\qquad \cdot\partial_s\Big[\hat \varphi_2(\xi_1)\hat \varphi_2(\xi_2)\widehat{\overline{\varphi_2}}(\xi_3)\Big]\,d\xi_1d\xi_2\,ds.
\end{align*}
Denote the triple operator $T(\cdot,\cdot,\cdot)$ by
\begin{align*}
\mathcal F\Big(T(f_1,f_2,f_3)\Big)(\xi)=&-\frac 38\langle\xi\rangle^{-1} \int_{\xi_1+\xi_2+\xi_3=\xi}\frac{ \chi_{\le 1}(\xi)\chi_{\le 1}(\xi_1)\chi_{\le 1}(\xi_2)\chi_{\le 1}(\xi_3)}{\langle\xi\rangle+\langle\xi_1\rangle+\langle\xi_2\rangle-\langle\xi_3\rangle}\\
&\qquad\qquad \cdot\widehat{f_1}(\xi_1)\widehat{f_2}(\xi_2)\widehat{f_3}(\xi_3)\,d\xi_1d\xi_2.
\end{align*}
Then it allows us to write
\begin{align*}
\eqref{duhamel-varphi1-2-1}
=&
-\fe^{it\langle\nabla\rangle}T(\mathcal W_2, \mathcal W_2, \mathcal W_2)(t)+\fe^{it_0\langle\nabla\rangle}T(\mathcal W_2, \mathcal W_2, \mathcal W_2)(t_0)\\
&\quad +\int_{t_0}^t \fe^{is\langle\nabla\rangle}T\big(\fe^{-is\langle\nabla\rangle}\partial_s\varphi_2,  \mathcal W_2,  \mathcal W_2\big)(s)\,ds.
\end{align*}
Here we shall slightly abuse notation and write $\hat f(\xi)$ for both itself and its complex conjugate (i.e. $\widehat{\overline{f}}(\xi)$), since it has no influence in our analysis.

Therefore, inserting the above into \eqref{duhamel-varphi1-2}, we obtain that
\begin{align*}
\varphi_1(t)=&\varphi_1(t_0)-\fe^{it\langle\nabla\rangle}T(\mathcal W_2, \mathcal W_2, \mathcal W_2)(t)+\fe^{it_0\langle\nabla\rangle}T(\mathcal W_2, \mathcal W_2, \mathcal W_2)(t_0)\\
&\quad + \int_{t_0}^t \fe^{is\langle\nabla\rangle}T\big(\fe^{-is\langle\nabla\rangle}\partial_s\varphi_2,  \mathcal W_2,  \mathcal W_2\big)(s)\,ds\\
&\quad -3\int_{t_0}^t \fe^{is\langle\nabla\rangle} \langle\nabla\rangle^{-1}\Big(N_1\big(\mathcal W_1,\mathcal W_2\big)
+N_2\big(\mathcal W_2\big)\Big)\,ds.
\end{align*}
Then one can drive that
%\begin{subequations}\label{duhamel-varphi1-2-345}
\begin{align*}
\mathcal W_1(t)=&\fe^{-i(t-t_0)\langle\nabla\rangle}\mathcal W_1(t_0)-T(\mathcal W_2, \mathcal W_2, \mathcal W_2)(t)+\fe^{-i(t-t_0)\langle\nabla\rangle}T(\mathcal W_2, \mathcal W_2, \mathcal W_2)(t_0)\\
%\label{duhamel-varphi1-2-3}\\
&\quad + \int_{t_0}^t \fe^{-i(t-s)\langle\nabla\rangle}T\big(\fe^{-is\langle\nabla\rangle}\partial_s\varphi_2,  \mathcal W_2,  \mathcal W_2\big)(s)\,ds\\
%\label{duhamel-varphi1-2-4}\\
&\quad -3\int_{t_0}^t \fe^{-i(t-s)\langle\nabla\rangle}\langle\nabla\rangle^{-1} \Big(N_1\big(\mathcal W_1,\mathcal W_2\big)
+N_2\big(\mathcal W_2\big)\Big)\,ds.
%\label{duhamel-varphi1-2-5}
\end{align*}
%\end{subequations}
Therefore, by Strichartz estimates Lemma \ref{lem:strichartz} we get that for any $\gamma\ge 0$, any $t_0\in I\subset \R$,
%\begin{subequations}\label{duhamel-varphi1-2-345}
\begin{align*}
&\big\||\nabla|^{\gamma}\mathcal W_1\big\|_{L^4_{tx}(I)}+\big\||\nabla|^{\gamma}\langle\nabla\rangle^\frac12\mathcal W_1\big\|_{L^\infty_t L^2_x(I)}\\
\lesssim &
\big\||\nabla|^{\gamma}\langle\nabla\rangle^\frac12\mathcal W_1(t_0)\big\|_{L^2_x}
+\big\||\nabla|^\gamma T(\mathcal W_2, \mathcal W_2, \mathcal W_2)\big\|_{L^4_{tx}(I)}\\
&+\big\||\nabla|^{\gamma} T(\mathcal W_2, \mathcal W_2, \mathcal W_2)\big\|_{L^\infty_tL^2_x}
+\big\||\nabla|^\gamma T\big(\fe^{-is\langle\nabla\rangle}\partial_s\varphi_2,  \mathcal W_2,  \mathcal W_2\big)\big\|_{L^\frac43_{tx}(I)}\\
&+\big\||\nabla|^\gamma N_1\big(\mathcal W_1,\mathcal W_2\big) \big\|_{L^\frac43_{tx}(I)}
+\big\||\nabla|^\gamma N_2\big(\mathcal W_2\big) \big\|_{L^\frac43_{tx}(I)}\\
\triangleq &
\Pi_1+\cdots +\Pi_6.
\end{align*}
Below we will  estimate $\Pi_1,\Pi_2,\cdots ,\Pi_6$ one by one.

\noindent {\bf Estimate on $\Pi_1$.} Note that  $\langle \nabla\rangle-1=\frac{|\nabla|^2}{1+\langle\nabla\rangle}$. When $t_0=0$,  by \eqref{eq:W-12-KG-intialdatum} we have 
\begin{align}\label{est:mathcal-w0}
\big\||\nabla|^{\gamma}\langle\nabla\rangle^\frac12\mathcal W_1(0)\big\|_{L^2_x}
\lesssim &
\big\||\nabla|^{\gamma}\langle\nabla\rangle^{-\frac12}\big( \langle \nabla\rangle-1\big)\mathcal{S}_\varepsilon v_0\big\|_{L^2_x}
+\varepsilon^2\big\||\nabla|^{\gamma}\langle\nabla\rangle^{-\frac12}\mathcal{S}_\varepsilon v_1\big\|_{L^2_x}\notag\\
\lesssim &
\big\||\nabla|^{\gamma+2}\mathcal{S}_\varepsilon v_0\big\|_{L^2_x}
+\varepsilon^2\big\||\nabla|^{\gamma}\mathcal{S}_\varepsilon \big(\re(v_1)\big)\big\|_{L^2_x}\notag\\
\le  &  C\varepsilon^{\gamma+2}.
\end{align}
Here and below in this proof, $C$ only depends on $\| v_0\|_{H^{\gamma+2}_x}$ and $\| v_1\|_{H^{\gamma}_x}$.

\noindent {\bf Estimate on $\Pi_2$.}
Note that
$$
\frac{1}{\langle\xi\rangle+\langle\xi_1\rangle+\langle\xi_2\rangle-\langle\xi_3\rangle}\in C^\infty\left[\big(\overline{B_1}\big)^4\right],
$$
where $B_1\in \R^2$ is the unit ball.  Then by Coifman-Meyer's Theorem and \eqref{est:mathcal-W12}, we have that
\begin{align*}
\big\||\nabla|^\gamma T(\mathcal W_2, \mathcal W_2, \mathcal W_2)\big\|_{L^4_{tx}(\R)}
\lesssim &
\big\||\nabla|^\gamma \mathcal W_2\big\|_{L^4_{tx}(\R)}\big\|\nabla \mathcal W_2\big\|_{L^\infty_{t}L^2_x(\R)}^2\\
\le & C  \varepsilon^{\gamma+2}.
\end{align*}

\noindent {\bf Estimate on $\Pi_3$.}
Similarly, by Coifman-Meyer's Theorem and \eqref{est:mathcal-W12}, we have that
\begin{align*}
\big\||\nabla|^\gamma T(\mathcal W_{2}, \mathcal W_{2}, \mathcal W_{2})(t
_0)\big\|_{L^2_{x}(\R^2)}
\lesssim &
\big\||\nabla|^\gamma \mathcal W_{2}(t_0)\big\|_{L^2_{x}}\big\|\nabla \mathcal W_{2}(t_0)\big\|_{L^2_x}^2\\
\le & C  \varepsilon^{\gamma+2}.
\end{align*}

\noindent {\bf Estimate on $\Pi_4$.}
By \eqref{eq:W-12-KG}, we have that
$$
\partial_t\varphi_2=-3\fe^{-it\langle\nabla\rangle}\langle\nabla\rangle^{-1}\big(|w(t)|^2w(t)\big).
$$
Therefore, arguing similarly as the estimation on $\Pi_2$, by \eqref{est:mathcal-W12} and \eqref{eq:w-W12}  we obtain that
\begin{align*}
&\big\||\nabla|^\gamma T\big(\fe^{-is\langle\nabla\rangle}\partial_s\varphi_2,  \mathcal W_2,  \mathcal W_2\big)\big\|_{L^\frac43_{tx}(\R)}\\
\lesssim &
\big\||\nabla|^\gamma \fe^{-is\langle\nabla\rangle}\partial_s\varphi_2\big\|_{L^4_tL^2_{x}(\R)}\big\| \mathcal W_2\big\|_{L^4_tL^8_{x}(\R)}^2\\
&\qquad+\big\|\fe^{-is\langle\nabla\rangle}\partial_s\varphi_2\big\|_{L^4_tL^2_{x}(\R)}\big\||\nabla|^\gamma  \mathcal W_2\big\|_{L^4_tL^8_{x}(\R)}\big\|\mathcal W_2\big\|_{L^4_tL^8_{x}(\R)}\\
\lesssim &
\big\||\nabla|^\gamma w\big\|_{L^4_{tx}(\R)}
\big\||\nabla|^\frac34 w\big\|_{L^\infty_{t}L^2(\R)}^2
\big\||\nabla|^\frac14 \mathcal W_2\big\|_{L^\infty_{t}L^2_x(\R)}^2\\
\le & C  \varepsilon^{\gamma+2}.
\end{align*}

\noindent {\bf Estimate on $\Pi_5$.}
By Kato-Ponce's inequality and \eqref{est:mathcal-W12}, we have that
\begin{align*}
&\big\||\nabla|^\gamma N_1\big(\mathcal W_1,\mathcal W_2\big) \big\|_{L^\frac43_{tx}(I)}\\
\lesssim &
\big\||\nabla|^\gamma \mathcal W_1\big\|_{L^4_{tx}(I)}
\Big(\big\|\mathcal W_1\big\|_{L^4_{tx}(I)}^2+\big\|\mathcal W_2\big\|_{L^4_{tx}(I)}^2\Big)\\
&\quad +\big\||\nabla|^\gamma \mathcal W_2\big\|_{L^4_{tx}(I)}\big\||\mathcal W_1\big\|_{L^4_{tx}(I)}
\Big(\big\|\mathcal W_1\big\|_{L^4_{tx}(I)}+\big\||\mathcal W_2\big\|_{L^4_{tx}(I)}\Big)\\
\le &
C \Big(\big\||\nabla|^\gamma \mathcal W_1\big\|_{L^4_{tx}(I)}
+\varepsilon^\gamma \big\|\mathcal W_1\big\|_{L^4_{tx}(I)}\Big)\Big(\big\|\mathcal W_1\big\|_{L^4_{tx}(I)}+\big\||\mathcal W_2\big\|_{L^4_{tx}(I)}\Big).
\end{align*}

\noindent {\bf Estimate on $\Pi_6$.}
By Kato-Ponce's and Bernstein's inequalities, and \eqref{est:mathcal-W12}, we have that
\begin{align*}
&\big\||\nabla|^\gamma N_2\big(\mathcal W_2\big) \big\|_{L^\frac43_{tx}(\R)}\\
\lesssim &
\big\||\nabla|^\gamma P_{>1}\mathcal W_2\big\|_{L^4_{tx}(\R)}
\big\|\mathcal W_2\big\|_{L^4_{tx}(\R)}^2\\
\le &
C
\varepsilon^{2+\gamma}.
\end{align*}

Combining the estimates on $\Pi_2$--$\Pi_6$, we obtain that
\begin{align}\label{est:mathcal-W1-improved}
&\big\||\nabla|^{\gamma}\mathcal W_1\big\|_{L^4_{tx}(I)}+\big\||\nabla|^{\gamma}\langle\nabla\rangle^\frac12\mathcal W_1\big\|_{L^\infty_t L^2_x(I)}\notag\\
\le  &
C_1\big\||\nabla|^{\gamma}\langle\nabla\rangle^\frac12\mathcal W_1(t_0)\big\|_{L^2_x}+C_2\varepsilon^{2+\gamma}\notag\\
&\quad +C_3\Big(\big\||\nabla|^\gamma \mathcal W_1\big\|_{L^4_{tx}(I)}
+\varepsilon^\gamma \big\|\mathcal W_1\big\|_{L^4_{tx}(I)}\Big)\Big(\big\|\mathcal W_1\big\|_{L^4_{tx}(I)}+\big\|\mathcal W_2\big\|_{L^4_{tx}(I)}\Big),
\end{align}
where  $C_j,j=1,2,3$ only depend on $\| v_0\|_{H^{\gamma+2}_x}$  and $\| v_1\|_{H^{\gamma}_x}$.

To proceed, we will do estimates inductively in $\gamma$.

First, for $\gamma=0$, we derive  from \eqref{est:mathcal-W12} that  there exists $C_0>0$ such that
$$
\big\|\mathcal W_1\big\|_{L^4_{tx}(\R)}+\big\||\mathcal W_2\big\|_{L^4_{tx}(\R)}\le C_0.
$$
Therefore, there exists a constant $K=K(C_0,C_3)$ and a sequence of time intervals
$$
\bigcup\limits_{k=0}^K I_k=\R^+,
$$
(the negative time direction can be treated similarly) with
$$
I_0=[0, T_1], \quad I_k=(T_k, T_{k+1}] \mbox{ for } k=1,\cdots, K-1,  \quad I_{K}=[T_K, +\infty),
$$
such that
\begin{align}\label{subinterval-length-W12}
C_3\Big(\big\|\mathcal W_1\big\|_{L^4_{tx}(I)}+\big\|\mathcal W_2\big\|_{L^4_{tx}(I)}\Big)
\le \frac1{4}.
\end{align}

Therefore, by \eqref{est:mathcal-W1-improved}, for any $1\le k\le K$, we have that
\begin{align*}
\big\|\mathcal W_1\big\|_{L^4_{tx}(I_k)}+\big\|\langle\nabla\rangle^\frac12\mathcal W_1\big\|_{L^\infty_t L^2_x(I_k)}
\le
2C_1\big\|\langle\nabla\rangle^\frac12\mathcal W_1(t_k)\big\|_{L^2_x}+2C_2\varepsilon^{2}.
\end{align*}
By iteration, we obtain that
\begin{align*}
\big\|\mathcal W_1\big\|_{L^4_{tx}(\R)}+\big\|\langle\nabla\rangle^\frac12\mathcal W_1\big\|_{L^\infty_t L^2_x(\R)}
\le
(2C_1)^{K+1}\Big(\big\|\langle\nabla\rangle^\frac12\mathcal W_{1,0}\big\|_{L^2_x}+2C_2\varepsilon^{2}\Big).
\end{align*}
Applying \eqref{est:mathcal-w0}, we get
\begin{align}\label{est:W1-0}
\big\|\mathcal W_1\big\|_{L^4_{tx}(\R)}+\big\|\langle\nabla\rangle^\frac12\mathcal W_1\big\|_{L^\infty_t L^2_x(\R)}
\le
C\varepsilon^{2},
\end{align}
where  $C$ only depends on $\| v_0\|_{H^{2}_x}$.

Second, for general $\gamma>0$, we insert  \eqref{est:W1-0} into \eqref{est:mathcal-W1-improved} to get
\begin{align*}
\big\||\nabla|^{\gamma}\mathcal W_1\big\|_{L^4_{tx}(I)}&+\big\||\nabla|^{\gamma}\langle\nabla\rangle^\frac12\mathcal W_1\big\|_{L^\infty_t L^2_x(I)}\\
\le  
&
\tilde C_1\big\||\nabla|^{\gamma}\langle\nabla\rangle^\frac12\mathcal W_1(t_0)\big\|_{L^2_x}+\tilde C_2\varepsilon^{2+\gamma}\\
&\quad +\tilde C_3\big\||\nabla|^\gamma \mathcal W_1\big\|_{L^4_{tx}(I)}
\Big(\big\|\mathcal W_1\big\|_{L^4_{tx}(I)}+\big\|\mathcal W_2\big\|_{L^4_{tx}(I)}\Big),
\end{align*}
where  $\tilde C_j,j=1,2,3$ only depend on $\| v_0\|_{H^{\gamma+2}_x}$ and $\| v_1\|_{H^{\gamma}_x}$.
Then arguing similarly as above, we obtain \eqref{est:W1-gamma}.
\end{proof}
%Applying \eqref{est:mathcal-w0}, we get

%\end{subequations}
Note that Proposition \ref{prop:h-varep} is followed from \eqref{def:w-h}, \eqref{relationship-w-W12}, \eqref{est:mathcal-W12} and \eqref{est:W1-gamma}.

\subsection{End of the proof}\label{subsec:keyprop-end}
Now we are in the position to improve the estimate on $ \partial_tv^{\varepsilon} $. Let $v^{\varepsilon}$ to be the solution of \eqref{eq:nls-wave}, note that
\begin{align}\label{relationship-v-var}
\big\| |\nabla|^{\gamma} \partial_tv^{\varepsilon} \big\|_{L^4_{tx}(\R)}
=&\varepsilon^{-2-\gamma}\big\| |\nabla|^{\gamma} \partial_t\mathcal{S}_\varepsilon \big(v^{\varepsilon}\big) \big\|_{L^4_{tx}(\R)}.
\end{align}
Moreover, by \eqref{relationship-w-W12}, we have that
\begin{align*}
\partial_t\mathcal{S}_\varepsilon \big(v^{\varepsilon}\big)
=&\partial_t\Big(\fe^{-it}w\Big)\\
=&\fe^{-it} \big(-iw+\partial_tw\big)\\
%=&\fe^{-it} \big(-iw+\partial_tw\big)\\
=&\frac12\fe^{-it} \Big[\big(\langle\nabla\rangle+1\big)\mathcal W_1+\big(\langle\nabla\rangle-1\big)\mathcal W_2\Big].
\end{align*}
Therefore, by \eqref{est:mathcal-W12} and \eqref{est:W1-gamma}, 
\begin{align*}
\big\| |\nabla|^{\gamma} \partial_t\mathcal{S}_\varepsilon \big(v^{\varepsilon}\big) \big\|_{L^4_{tx}(\R)}
\lesssim  &
\big\| |\nabla|^{\gamma}\big(\langle\nabla\rangle+1\big)\mathcal W_1\big\|_{L^4_{tx}(\R)}
+\big\| |\nabla|^{\gamma}\big(\langle\nabla\rangle-1\big)\mathcal W_2\big\|_{L^4_{tx}(\R)}\\
\lesssim &
 \big\| |\nabla|^{\gamma}\mathcal W_1\big\|_{L^4_{tx}(\R)}
+\big\| |\nabla|^{\gamma+2}\mathcal W_1\big\|_{L^4_{tx}(\R)}
+\big\| |\nabla|^{\gamma+2}\mathcal W_2\big\|_{L^4_{tx}(\R)}\\
\lesssim &
C\varepsilon^{2+\gamma},
\end{align*}
where  $C$ only depends on $\| v_0\|_{H^{\gamma+2}_x}$ and $\| v_1\|_{H^{\gamma}_x}$. Therefore, this last estimate together with \eqref{relationship-v-var} yields that
\begin{align*}
\big\| |\nabla|^{\gamma} \partial_tv^{\varepsilon} \big\|_{L^4_{tx}(\R)}
\le C.
\end{align*}
This finishes the proof of the key proposition.

 \vskip 1.5cm

\section{Schr\"odinger-wave profile and Uniform-in-time nonrelativistic limit}\label{sec:global}

\vskip .5cm

 Suppose that $v^{\varepsilon}$ is the solution of \eqref{eq:nls-wave}, then by \eqref{eq:r}-\eqref{eq:initial-data},  we obtain that
\EQn{
	\label{eq:kg-r-2}
	\left\{ \aligned
	&\varepsilon^2\partial_{tt} r - \De r + \frac{1}{\varepsilon^2} r +A_\varepsilon\big(v^{\varepsilon},r\big)+B_\varepsilon\big(v^{\varepsilon}\big)
	=0, \\
	& r(0,x) = 0,\quad  \partial_tr(0,x) = -2\re(v_1).
	\endaligned
	\right.
}

Arguing similarly as in Section \ref{sec:prop-key}, we use the scaling argument and denote
$$
R(t,x)=\mathcal{S}_\varepsilon r(t,x),\quad h^{\varepsilon}(t,x)=\mathcal{S}_\varepsilon \big(v^{\varepsilon}\big)(t,x),
$$
then by \eqref{eq:kg-r-2}, $R$ obeys the following equation:
\begin{align}\label{eq:psi-scaling-2}
\partial_{tt} R - \De R + R +G\big(h^{\varepsilon},R\big)+F_1\big(h^{\varepsilon}\big)=0,
\end{align}
where
 \begin{align*}
 G(f,g)=&3\Big[\fe^{2it}f^2+\fe^{-2it}\big(\bar f\big)^2\Big]g
 + 3\Big[\fe^{it}f+\fe^{-it}\bar f\Big]g^2+g^3;\\
F_1(f)=&\fe^{3it}f^3+\fe^{-3it}\big(\bar f\big)^3.
 \end{align*}
Moreover, the initial data of $R$ is that
\begin{align}\label{eq:psi0-scaling}
R(0)=0,\quad \partial_{t} R(0)=-2\varepsilon^2\mathcal{S}_\varepsilon \big(\re(v_1)\big).
\end{align}

 Denote
\begin{align}\label{def:phi}
\phi\triangleq  \langle \nabla\rangle^{-1}\big(\partial_t +i\langle \nabla\rangle \big)R,
\end{align}
then we have that
\begin{align}\label{phi-varphi}
R= \im\> \phi;\quad  \partial_t R= \re\big(\langle \nabla\rangle \phi\big).
\end{align}
Moreover, $\phi$ obeys the following equation
\begin{align} \label{eq:phi}
\big(\partial_t -i\langle \nabla\rangle \big) \phi +\langle \nabla\rangle^{-1}\Big[G\big(h^{\varepsilon},R\big)+F_1\big(h^{\varepsilon}\big)\Big]=0,
\end{align}
with the initial data
\begin{align}\label{phi-initialdata}
\phi(0)=-2\varepsilon^2\langle \nabla\rangle^{-1} \mathcal{S}_\varepsilon \big(\re(v_1)\big).
\end{align}
By Duhamel's formula, we have that
\begin{align}\label{Duh-1}
\phi(t)= \fe^{i(t-t_0)\langle\nabla\rangle}\phi(t_0)- \int_{t_0}^t \fe^{i(t-s)\langle\nabla\rangle} \langle \nabla\rangle^{-1}\Big[G\big(h^{\varepsilon},R)+F_1\big(h^{\varepsilon}\big)\Big]\,ds.
\end{align}
Now we focus on the term
$$
\int_{t_0}^t \fe^{i(t-s)\langle\nabla\rangle} \langle \nabla\rangle^{-1} F_1\big(h^{\varepsilon}\big)\,ds,
$$
which can be split into the following two terms:
$$
\int_{t_0}^t \fe^{i(t-s)\langle\nabla\rangle} \langle \nabla\rangle^{-1}P_{\le 1} F_1\big(h^{\varepsilon}\big)\,ds
+\int_{t_0}^t \fe^{i(t-s)\langle\nabla\rangle} \langle \nabla\rangle^{-1}P_{> 1} F_1\big(h^{\varepsilon}\big)\,ds.
$$
From  the definition of $F_1$, we have  that
\begin{align}
&\int_{t_0}^t \fe^{i(t-s)\langle\nabla\rangle} \langle \nabla\rangle^{-1}P_{\le 1} F_1\big(h^{\varepsilon}\big)\,ds\notag\\
=&\fe^{it\langle\nabla\rangle} \int_{t_0}^t \fe^{-is(\langle\nabla\rangle-3)} \langle \nabla\rangle^{-1}P_{\le 1} \big(h^{\varepsilon}\big)^3\,ds\notag\\
&\quad +\fe^{it\langle\nabla\rangle} \int_{t_0}^t \fe^{-is(\langle\nabla\rangle+3)} \langle \nabla\rangle^{-1}P_{\le 1} \big(\overline{h^{\varepsilon}}\big)^3\,ds.\label{F1-12}
\end{align}
Note that
\begin{align}\label{differ}
\fe^{-is(\langle\nabla\rangle\pm 3)}P_{\le 1} =\frac{d}{ds}\Big(\fe^{-is(\langle\nabla\rangle\pm3)}\Big) \frac{1}{-i(\langle\nabla\rangle\pm3)}P_{\le 1}.
\end{align}
Then using this formula and  integration by parts, we get that
\begin{align}
 & \int_{t_0}^t \fe^{-is(\langle\nabla\rangle-3)} \langle \nabla\rangle^{-1}P_{\le 1} \big(h^{\varepsilon}\big)^3\,ds\notag\\
=& \fe^{-is(\langle\nabla\rangle-3)}  \frac{1}{-i(\langle\nabla\rangle-3)}\langle \nabla\rangle^{-1}P_{\le 1} \big(h^{\varepsilon}\big)^3\Big|_{s=t_0}^{s=t}\notag\\
&\quad-3 \int_{t_0}^t \fe^{-is(\langle\nabla\rangle-3)} \frac{1}{-i(\langle\nabla\rangle-3)} \langle \nabla\rangle^{-1}P_{\le 1} \Big[\big(h^{\varepsilon}\big)^2\> \partial_sh^{\varepsilon}\Big]\,ds.\label{Integ-1}
\end{align}
Similarly,
\begin{align}
 & \int_{t_0}^t \fe^{-is(\langle\nabla\rangle+3)} \langle \nabla\rangle^{-1}P_{\le 1} \big(h^{\varepsilon}\big)^3\,ds\notag\\
=& \fe^{-is(\langle\nabla\rangle+3)}  \frac{1}{-i(\langle\nabla\rangle+3)}\langle \nabla\rangle^{-1}P_{\le 1} \big(\overline{h^{\varepsilon}}\big)^3\Big|_{s=t_0}^{s=t}\notag\\
&\quad-3 \int_{t_0}^t \fe^{-is(\langle\nabla\rangle+3)} \frac{1}{-i(\langle\nabla\rangle+3)} \langle \nabla\rangle^{-1}P_{\le 1} \Big[\big(\overline{h^{\varepsilon}})^2\> \partial_s\overline{h^{\varepsilon}}\Big]\,ds.\label{Integ-2}
\end{align}

Now we insert the terms \eqref{Integ-1}, \eqref{Integ-2}  into \eqref{Duh-1}, and obtain that
\begin{subequations}\label{est:Duh-mod}
\begin{align}
\phi(t)=& \fe^{i(t-t_0)\langle\nabla\rangle}\phi(t_0)\\
&\quad
-\int_{t_0}^t \fe^{i(t-s)\langle\nabla\rangle} \langle \nabla\rangle^{-1}\Big[G\big(h^{\varepsilon},R)+P_{> 1}F_1\big(h^{\varepsilon}\big)\Big]\,ds\label{Duh-mod-1}\\
&\quad +   \frac{\fe^{3it} }{i(\langle\nabla\rangle-3)}\langle \nabla\rangle^{-1}P_{\le 1} \Big(h^{\varepsilon}(t)\Big)^3
-  \frac{\fe^{i(t-t_0)\langle \nabla \rangle}\fe^{3it_0}}{i(\langle\nabla\rangle-3)}\langle \nabla\rangle^{-1}P_{\le 1} \Big(h^{\varepsilon}(t_0)\Big)^3\label{Duh-mod-2}\\
&\quad +  \frac{\fe^{3it} }{i(\langle\nabla\rangle+3)}\langle \nabla\rangle^{-1}P_{\le 1} \Big(\overline{h^{\varepsilon}(t)}\Big)^3
-  \frac{\fe^{i(t-t_0)\langle \nabla \rangle}\fe^{-3it_0}}{i(\langle\nabla\rangle+3)}\langle \nabla\rangle^{-1}P_{\le 1} \Big(\overline{h^{\varepsilon}(t_0)}\Big)^3\label{Duh-mod-3}\\
&\quad-3\fe^{it\langle\nabla\rangle}  \int_{t_0}^t \fe^{-is(\langle\nabla\rangle-3)} \frac{1}{i(\langle\nabla\rangle-3)} \langle \nabla\rangle^{-1}P_{\le 1} \Big[\big(h^{\varepsilon}\big)^2\> \partial_s h^\varepsilon\Big]\,ds\label{Duh-mod-4}\\
&\quad-3\fe^{it\langle\nabla\rangle} \int_{t_0}^t \fe^{-is(\langle\nabla\rangle+3)} \frac{1}{i(\langle\nabla\rangle+3)} \langle \nabla\rangle^{-1}P_{\le 1} \Big[\big(\overline{h^\varepsilon})^2\> \partial_s\big(\overline{h^\varepsilon}\big)\Big]\,ds.\label{Duh-mod-5}
\end{align}
\end{subequations}

We shall give the estimates on the following norm:
$$
\big\|\phi\big\|_{Y_0(\R)}\triangleq \big\|\phi\big\|_{L^4_{tx}(\R)}+\big\|\langle\nabla\rangle^\frac12\phi\big\|_{L^\infty_tL^2_x(\R)}.
$$

For the linear term, we have that
\begin{align} \label{est:linear-phi0}
\big\| \fe^{it\langle\nabla\rangle}\phi(t_0)\big\|_{Y_0(\R)}
\lesssim &
\big\| \langle \nabla\rangle^{\frac12}\phi(t_0)\big\|_{L^2_x}.
\end{align}
In the following subsection, we will consider the nonlinear terms.

 \vskip .5cm

\subsection{Nonlinear estimates}\label{sec:nonlinear}
 \vskip .5cm

Now we estimate \eqref{Duh-mod-1}--\eqref{Duh-mod-5} terms by terms. Let $t_0\in I\subset \R$ be an time interval.

 \vskip .5cm

\noindent {\bf Estimate on \eqref{Duh-mod-1}.}

%Case 1: $d=2$.
By Lemma \ref{lem:strichartz}, we have that
\begin{align*}
\big\|\eqref{Duh-mod-1}\big\|_{Y_0(I)}
\lesssim &
\left\|G\big(h^{\varepsilon},R\big)\right\|_{Z_0(I)}+\left\|P_{> 1}F_1\big(h^{\varepsilon}\big)\right\|_{Z_0(I)},
%\\
%&\quad+\varepsilon^4\big\||\nabla|^{s_c} \langle \nabla\rangle^{-\frac12}F_2\big((v_{tt})^\varepsilon\big)\big\|_{L^1_tL^2_x(I)},
\end{align*}
where
\begin{align*}
\|f\|_{Z_0(I)}&\triangleq \|f\|_{L^\frac43_{tx}(I)}.
\end{align*}
For $G\big(h^{\varepsilon},R\big)$, by H\"older's inequality,
\begin{align*}
\left\|G\big(h^{\varepsilon},R\big)\right\|_{Z_0(I)}
\lesssim &
\sum\limits_{j=0}^2\big\|h^{\varepsilon}\big\|_{L^4_{tx}(I)}^j\big\|R\big\|_{L^4_{tx}(I)}^{3-j}.
\end{align*}
Note that
$$
\big\|\mathcal{S}_\varepsilon f\big\|_{L^4_{tx}(I)}=\big\|f\big\|_{L^4_{tx}(\varepsilon^2I)}.
$$
Therefore,
\begin{align*}
\left\|G\big(h^{\varepsilon},R\big)\right\|_{Z_0(I)}
\lesssim &
\sum\limits_{j=0}^2\big\|v^{\varepsilon}\big\|_{L^4_{tx}(\varepsilon^2I)}^j\big\|R\big\|_{L^4_{tx}(I)}^{3-j}.
\end{align*}

For $F_1\big(h^{\varepsilon}\big)$, since
$$
P_{> 1}F_1(f)=\fe^{3it}\big(P_{\gtrsim 1}f\>f^2\big)+\fe^{-3it}\big(P_{\gtrsim 1}\bar f\>\bar f^2\big)
$$
Then by H\"older's inequality we have that
\begin{align*}
\big\|P_{> 1}F_1\big(h^{\varepsilon}\big)\big\|_{Z_0(\R)}
\lesssim &
\big\|P_{\gtrsim 1}h^{\varepsilon}\big\|_{L^4_{tx}(\R)}\big\|h^{\varepsilon}\big\|_{L^4_{tx}(\R)}^2.
\end{align*}
Note that
$$
\big\|P_{\gtrsim 1}h^{\varepsilon}\big\|_{L^4_{tx}(\R)}\lesssim \big\||\nabla|^2 P_{\le 1}h^{\varepsilon}\big\|_{L^4_{tx}(\R)}
+\big\|P_{>  1}h^{\varepsilon}\big\|_{L^4_{tx}(\R)}.
$$
Then by Proposition \ref{prop:h-varep} we have that
\begin{align*}
\big\|P_{> 1}F_1\big(h^{\varepsilon}\big)\big\|_{Z_0(I)}
\le &
C\varepsilon^2,
\end{align*}
where the constant $C>0$ only depends on $\| v_0\|_{H^{2}_x}$ and $\| v_1\|_{L^2_x}$.

Together with the estimates above, we obtain that
\begin{align*}
\big\|\eqref{Duh-mod-1}\big\|_{Y_0(I)}
\lesssim &
\sum\limits_{j=0}^2\big\|v^{\varepsilon}\big\|_{L^4_{tx}(\varepsilon^2I)}^j\big\|R\big\|_{L^4_{tx}(I)}^{3-j}
+\varepsilon^2.
%&\qquad +\varepsilon^4|I|\big\| v_0\big\|_{H^{4+s_c}_x(\R^d)}.
\end{align*}

 \vskip .3cm

\noindent {\bf Estimates on \eqref{Duh-mod-2} and \eqref{Duh-mod-3}.}  Note that
\begin{align}\label{est:bound-nabal3}
\left\|\frac{1}{\langle\nabla\rangle\pm3}\langle \nabla\rangle^{-\frac12}P_{\le 1}f\right\|_{L^r_x(\R^d)}
\lesssim \|f\|_{L^r_x(\R^d)},\quad \mbox{ for any } 1\le r\le \infty.
\end{align}
Applying this estimate, Lemma \ref{lem:strichartz} and Proposition \ref{prop:v-varep}, we have that
\begin{align*}
\big\|\eqref{Duh-mod-2}\big\|_{Y_0(I)}&+\big\|\eqref{Duh-mod-3}\big\|_{Y_0(I)}\\
\lesssim &
\Big\|\frac{1}{\langle\nabla\rangle\pm3}\langle \nabla\rangle^{-1}P_{\le 1}\Big(h^{\varepsilon}\Big)^3\Big\|_{Y_0(I)}+\Big\|\frac{1}{\langle\nabla\rangle\pm3}\langle \nabla\rangle^{-\frac12}P_{\le 1}\Big(h^{\varepsilon}(t_0)\Big)^3\Big\|_{L^2_x}\\
\lesssim &
\Big\|\Big(h^{\varepsilon}\Big)^3\Big\|_{X_0(I)}+\Big\|\Big(h^{\varepsilon}(t_0)\Big)^3\Big\|_{L^2_x}\\
\lesssim &
\big\|h^{\varepsilon}\big\|_{X_0(I)}\big\|h^{\varepsilon}\big\|_{L^\infty_{tx}(I)}^2+\big\|h^{\varepsilon}(t_0)\big\|_{L^2_x(\R^2)}\big\|h^{\varepsilon}(t_0)\big\|_{L^\infty_x}^2\\
\lesssim &
\varepsilon^2\big\|v^{\varepsilon}\big\|_{X_0(\R)}\big\|v^{\varepsilon}\big\|_{L^\infty_{tx}(\R)}^2+\varepsilon^2\big\|v^{\varepsilon}(t_0)\big\|_{L^2_x}\big\|v^{\varepsilon}(t_0)\big\|_{L^\infty_x}^2\\
\le & C\big(\|(v_0,v_1)\|_{H^{2}_x\times L^2_x}\big) \varepsilon^2.
\end{align*}
%By \eqref{est:v-Xsc} and Sobolev's inequality, we obtain that
%\begin{align*}
%\big\|\eqref{Duh-mod-2}\big\|_{X_{s_c}(I)}+\big\|\eqref{Duh-mod-3}\big\|_{X_{s_c}(I)}+\big\|\eqref{Duh-mod-2}\big\|_{L^\infty_tH^{s_c}_x(I)}+\big\|\eqref{Duh-mod-3}\big\|_{L^\infty_tH^{s_c}_x(I)}
%\le &
%C\varepsilon^2,
%\end{align*}
%where the constant $C>0$ only depend on $\big\| v_0\big\|_{H^2_x(\R^d)}$.

 \vskip .3cm

\noindent {\bf Estimates on \eqref{Duh-mod-4} and \eqref{Duh-mod-5}.}

%Case 1: $d=2$.
By Lemma \ref{lem:strichartz}, \eqref{est:bound-nabal3} and Proposition \ref{prop:h-varep}, we have that
\begin{align*}
\big\|\eqref{Duh-mod-4}\big\|_{Y_0(I)}+\big\|\eqref{Duh-mod-5}\big\|_{Y_0(I)}
\lesssim &
\Big\|\frac{1}{\langle\nabla\rangle\pm3}P_{\le 1}\Big(\big(h^{\varepsilon}\big)^2\> \partial_s h^{\varepsilon}\Big)\Big\|_{Z_0(I)}\\
\lesssim &
\Big\|\big(h^{\varepsilon}\big)^2\> \partial_s h^{\varepsilon}\Big\|_{Z_0(I)}\\
\lesssim &
\big\|h^{\varepsilon}\big\|_{L^4_{tx}(\R)}^2\big\|\partial_s h^{\varepsilon}\big\|_{L^4_{tx}(\R)}\\
\le & C\big(\|(v_0,v_1)\|_{H^{2}_x\times L^2_x}\big) \varepsilon^2.
\end{align*}
%By Lemma \ref{lem:spacetime-norm-NLS} and Sobolev's inequality, we have that
%$$
%\big\|\partial_s\big(v^\varepsilon\big)\big\|_{X_{s_c}(\R)}
%\lesssim
%\varepsilon^2\| v_0\|_{H^{2+s_c}_x(\R^d)}.
%$$
%Therefore, we have that
%\begin{align*}
%\big\|\eqref{Duh-mod-4}\big\|_{X_{s_c}(I)}
%+\big\|\eqref{Duh-mod-5}\big\|_{X_{s_c}(I)}+\big\|\eqref{Duh-mod-4}\big\|_{L^\infty_tH^{s_c}_x(I)}+\big\|\eqref{Duh-mod-5}\big\|_{L^\infty_tH^{s_c}_x(I)}
%\le
%C\varepsilon^2,
%\end{align*}
%where the constant $C>0$ only depend on $\big\| v_0\big\|_{H^{2+s_c}_x(\R^d)}$.
%
%
%Note that from Lemma \ref{lem:spacetime-norm-NLS}, we have that
%\begin{align}
%\label{est:v-Xsc}
%\big\|\langle\nabla\rangle^2 v^{\varepsilon}\big\|_{X_{s_c}(\varepsilon^2I)}
%\lesssim
%\big\|\langle\nabla\rangle^2 v^{\varepsilon}\big\|_{X_{s_c}(\R)}
%\lesssim \big\| v_0\big\|_{H^2(\R^d)}.
%\end{align}
%Therefore,
%\begin{align*}
%\big\|\eqref{Duh-mod-1}\big\|_{X_{s_c}(I)}
%\le  &
%C_1\big\|v\big\|_{X_{s_c}(\varepsilon^2I)}^2\big\|R\big\|_{X_{s_c}(I)}
%+C_2\Big(\big\|R\big\|_{X_{s_c}(I)}^2+\big\|R\big\|_{X_{s_c}(I)}^3\Big)+C_3\varepsilon^2,
%\end{align*}
%where  $C_1>0$ is an absolute constants, and $C_j,j=2,3,4$ only depend on $\| v_0\|_{H^{4+s_c}_x(\R^d)}$.

Now together with the estimates on \eqref{est:Duh-mod} and \eqref{est:linear-phi0}, it infers that
\begin{align*}
\big\|\phi\big\|_{Y_0(I)}
\le&
C_1\big\|\langle \nabla\rangle^{\frac12}\phi(t_0)\big\|_{L^2(\R^d)}+C_2\varepsilon^2
+C_3\big\|v^\varepsilon\big\|_{L^4_{tx}(\varepsilon^2I)}^2\big\|R\big\|_{L^4_{tx}(I)}
\\
&+C_4\Big(\big\|R\big\|_{L^4_{tx}(I)}^2+\big\|R\big\|_{L^4_{tx}(I)}^3\Big),
\end{align*}
where   $C_j,j=1,\cdots,4$ only depend on $\|(v_0,v_1)\|_{H^{2}_x\times L^2_x}$.
Moreover, note that
\begin{align*}
\big\|R\big\|_{L^4_{tx}(I)}\le \big\|\phi\big\|_{L^4_{tx}(I)}.
\end{align*}
%and thus
% \begin{align}\label{R-phi-L4}
%\big\|R\big\|_{L^4_{tx}(I)}\le C\big(\|v_0\|_{H^1(\R^2)}\big).
%\end{align}
Therefore,
\begin{align}\label{est:phi-L4-int}
\big\|\phi\big\|_{Y_0(I)}
\le&
C_1\big\|\langle \nabla\rangle^{\frac12}\phi(t_0)\big\|_{L^2(\R^2)}+C_2\varepsilon^2
+C_3\big\|v^{\varepsilon}\big\|_{L^4_{tx}(\varepsilon^2I)}^2\big\|\phi\big\|_{L^4_{tx}(I)}
\notag\\
&+C_4\Big(\big\|\phi\big\|_{L^4_{tx}(I)}^2+\big\|\phi\big\|_{L^4_{tx}(I)}^3\Big).
\end{align}

\subsection{Proof of Theorem \ref{thm:main1-global}}\label{sec:proof-thm1}

By Proposition \ref{prop:v-varep}, we have that
\begin{align}\label{est:uniform-v}
\big\|v^{\varepsilon}\big\|_{L^4_{tx}(\R)} \le & C\big(\|(v_0,v_1)\|_{H^2\times L^2}\big).
\end{align}
%Moreover, an immediate consequence of Lemma \ref{lem:uniformL4-2d} and \eqref{est:uniform-v}  is the following:
%\begin{align}\label{est:uniformL4-2d}
%\|R\|_{L^4_{tx}(\R)}\le C\Big(\|(u_0,u_1)\|_{H^1\times H^1(\R^2)}\Big).
%\end{align}

Therefore, there exists a constant $K=K(\|(v_0,v_1)\|_{H^2\times L^2},C_2,C_3)$ (but independent of $\varepsilon$) and a sequence of time intervals
$$
\bigcup\limits_{k=0}^K J_k=\R^+,
$$
(the negative time direction can be treated similarly) with
$$
J_0=[0, t_1], \quad J_k=(t_k, t_{k+1}] \mbox{ for } k=1,\cdots, K-1,  \quad J_{K}=[t_K, +\infty),
$$
such that
\begin{align}\label{subinterval-length}
C_3\big\|v^{\varepsilon}\big\|_{L^4_{tx}(J_k)}^2
\le \frac12.
\end{align}

Denote
$$
I_k=\varepsilon^{-2}J_k.
$$
Then by \eqref{est:phi-L4-int}, we obtain that for any $k\le K$,
\begin{align}\label{est:phi-L4-int-boot}
\big\|\phi\big\|_{Y_0(I_k)}
\le&
2C_1\big\|\langle \nabla\rangle^{\frac12}\phi(t_k)\big\|_{L^2(\R^2)}+2C_2\varepsilon^2
+2C_4\Big(\big\|\phi\big\|_{L^4_{tx}(I_k)}^2+\big\|\phi\big\|_{L^4_{tx}(I_k)}^3\Big).
\end{align}

Moreover, denote
$$
a_0(\varepsilon)\triangleq \Big(\varepsilon^2\big\|\langle \nabla\rangle^{\frac12}v_1\big\|_{L^2(\R^2)}+4C_2\varepsilon^2\Big)\cdot \big(4C_1\big)^K.
$$
Choosing $\varepsilon_0$ suitably small  such that for any $\varepsilon\in [0, \varepsilon_0]$, we have that
\begin{align}\label{small-a0}
C_4\big(a_0(\varepsilon)+a_0(\varepsilon)^2\big)\le \frac14.
\end{align}

Suppose that for some $k: 0\le k\le K-1$,
\begin{align}\label{iteration-k}
\big\|\langle \nabla\rangle^{\frac12}\phi(t_k)\big\|_{L^2_x}
\le
\varepsilon^2\big\|\langle \nabla\rangle^{\frac12}v_1\big\|_{L^2(\R^2)}\cdot \big(4C_1\big)^k
+
4C_2\varepsilon^2\sum\limits_{j=1}^k \big(4C_1\big)^{j-1},
\end{align}
then
$$
\big\| \langle \nabla\rangle^{\frac12}\phi(t_k)\big\|_{L^2(\R^d)}\le a_0.
$$
Therefore, applying these to \eqref{est:phi-L4-int-boot} for the time interval $I_{k}$ and bootstrap, we first obtain that
$$
\big\|\phi\big\|_{L^4_{tx}(I_{k})} \le a_0,
$$
and thus inserting into \eqref{est:phi-L4-int-boot} for  $I_{k}$ again, we obtain that
\begin{align*}
\big\|\langle \nabla\rangle^{\frac12}\phi(t_{k+1})\big\|_{L^2_x}
\le
\varepsilon^2\big\|\langle \nabla\rangle^{\frac12}v_1\big\|_{L^2(\R^2)}\cdot \big(4C_1\big)^{k+1}
+
4C_2\varepsilon^2\sum\limits_{j=1}^{k+1} \big(4C_1\big)^{j-1}.
\end{align*}
Hence, \eqref{iteration-k} holds for any $k: 0\le k\le K$ and thus
$$
\big\| \langle \nabla\rangle^{\frac12}\phi\big\|_{L^\infty_tL^2_x(\R)}\le a_0=C\varepsilon^2.
$$
for some $\varepsilon$-independent $C>0$.
Then by \eqref{phi-varphi}, we
 $$
\big\|R\big\|_{L^\infty_tL^2_x(\R)}\le C\varepsilon^2.
$$
Scaling back, we obtain that
  $$
\big\|r\big\|_{L^\infty_tL^2_x(\R)}\le C\varepsilon^2.
$$
This finishes the proof of Theorem \ref{thm:main1-global}.

\begin{remark} \label{rem:3d-smalldata}
A direct usage of the same argument, we can dealt with the three dimensional cases for small data. The conclusion can be read as follows:
Let $s\ge\frac12$, there exists a constant $\delta_0>0$ such that the following properties hold. Suppose that
$$
(u_0,u_1)\in H^{2+s}(\R^3) \times H^{2+s}(\R^3);\quad
\big\|(u_0,u_1)\big\|_{\dot H^\frac12\times \dot H^\frac12}\le \delta_0; \quad
v_1\in H^s(\R^3).
$$
Then
\begin{align*}
\big\|u^\varepsilon(t)-\fe^{\frac{it}{\varepsilon^2}}v^{\varepsilon}(t)-\fe^{-\frac{it}{\varepsilon^2}}\bar v^{\varepsilon}(t) \big\|_{H^s_x}
\le C \varepsilon^2.
\end{align*}
The small data assumption is used to assure the analogous estimate of \eqref{est:uniform-v}:
$$
\big\|v^{\varepsilon}\big\|_{X_\frac12(\R)}<+\infty
$$
 holds in three dimensional case.
A more refined estimate can further lower the restriction $s\ge \frac12$ to $s\ge0$.
\end{remark}
 %\vskip 1.5cm

\subsection{Optimality}\label{sec:sharp-global-op1}

\vskip .5cm

In this subsection, we give the proof of Theorem \ref{thm:main1-optimal}.  Let $\delta_0, a_0,b_0$ are positive constants.   Denote 
$$
f(x)=|x|^{-1}\chi_{a_0\le |\cdot| \le 2a_0}(x),
$$
and 
\begin{align}\label{v0-example}
\big(v_0,v_1\big)\triangleq \delta_0 \fe^{-i\frac{b_0 |x|^2}{2}} (f,g),
\end{align}
where $g\in \mathcal S$ are independent of $\delta_0, a_0,b_0$. 
In the following, we set the parameters satisfying 
\begin{align}
\label{relationships}
0<\varepsilon_0\ll \delta_0\ll a_0\ll b_0^{-1}\ll1.
\end{align}

Moreover, assume that
there exists a constant $c_0>0$ which is independent of $b_0$ such that
\begin{align}\label{lower-bound-v1v0}
\big\|-2\re(v_1)+\frac i2\big(v_0\big)^3-\frac i4\big(\overline{v_0}\big)^3\big\|_{L^2_x}\ge c_0\delta_0^3 \|f\|_{L^6_x}^3.
\end{align}

\begin{remark}
The existence of such $\big(v_0,v_1\big)$ is easy to check. Take an example, we set $v_1=0$, then
 \begin{align*}
\big| \frac i2\big(v_0\big)^3-\frac i4\big(\overline{v_0}\big)^3\big|
=\frac12\delta_0^3 |f|^3 \Big|2\fe^{-i\frac{3b_0 |x|^2}{2}}-\fe^{i\frac{3b_0 |x|^2}{2}}\Big|
\ge \frac12 \delta_0^3 |f|^3.
 \end{align*}
 Then in this case we can set $c_0=\frac18$.
\end{remark}

\subsubsection{Approximation}
Firstly, we consider the behavior of the solution to \eqref{eqs:scaling-w} with $v_0$ given by \eqref{v0-example}.
Treating the same as in Section \ref{sec:prop-key}, we define $\mathcal W_j,j=1,2$ as in \eqref{relationship-w-W12}.
Then by Duhamel's formula, we may write that for $j=1,2$,
\begin{align}
\mathcal W_j(t)=&\fe^{\mp it\langle \nabla\rangle}\mathcal W_{j,0}
-3\int_0^t \fe^{\mp i(t-s)\langle \nabla\rangle}\langle\nabla \rangle^{-1}   \big(|w|^2w\big)\,ds\notag\\
\triangleq &\fe^{\mp it\langle \nabla\rangle}\mathcal W_{j,0}
+\mathcal N_j(w).\label{def:W12}
\end{align}

For the initial data \eqref{v0-example}, we have that there exists a positive constant $C_{a_0, b_0}$ only depending on $a_0, b_0$ such that
\begin{align*}
\Big\|\fe^{\mp it\langle \nabla\rangle}\mathcal W_{j,0} \Big\|_{X_{s_c}(\R)}
\le &
C_{a_0, b_0}\delta_0.
\end{align*}
Then arguing similarly as the proof in Section \ref{sec:prop-key},  we have that
\begin{align}
\big\|w \big\|_{X_{s_c}(\R)}
\le
C_{a_0, b_0}\delta_0.
\label{w-L4-sd}
\end{align}
Moreover, by the proof in Sections \ref{subsec:Stri} and \ref{subsec:Imp-Stri}, we further have that for any $\gamma\ge -s_c$,
\begin{align}
\big\||\nabla|^\gamma w\big\|_{X_{s_c}(\R)}
\le & C_{a_0, b_0}\varepsilon^\gamma;\label{w-L4-ga-sd}\\
\big\||\nabla|^\gamma\mathcal N_j(w)\big\|_{X_{s_c}(\R)}
\le & C_{a_0, b_0}\delta_0^3 \varepsilon^\gamma; \label{Nwj-L4-sd}\\
\big\|\mathcal N_1(w)\big\|_{X_{s_c}(\R)}
\le
& C_{a_0, b_0}\delta_0^3\varepsilon^2. \label{Nw1-L4-sd}
\end{align}
Furthermore, by \eqref{eq:w-W12}, we may write
\begin{align}\label{Duh-w}
w(t)=&\frac i2\big(\mathcal W_1-\mathcal W_2\big)\notag\\
= &\frac i2 \Big[\fe^{-it \langle \nabla\rangle}\mathcal W_{1,0}-\fe^{it \langle \nabla\rangle}\mathcal W_{2,0}\Big]+\frac i2\big(\mathcal N_1(w)-\mathcal N_2(w)\big).
\end{align}
By \eqref{eq:W-12-KG-intialdatum} and \eqref{Duh-w}, we further have that
\begin{align}\label{Duh-w-2}
w(t)
= &-\frac 12 \left[\fe^{-it \langle \nabla\rangle}\langle \nabla\rangle^{-1}\big(1-\langle \nabla\rangle\big)-\fe^{it \langle \nabla\rangle}\langle \nabla\rangle^{-1}\big(1+\langle \nabla\rangle\big)\right]\mathcal{S}_\varepsilon v_0+\mathcal N(w),
\end{align}
where the term $\mathcal N(w)$ reads  as
 \begin{align*}
\mathcal N(w)\triangleq\frac i2\big(\mathcal N_1(w)-\mathcal N_2(w)\big)
+\frac i2  \varepsilon^2\langle \nabla \rangle^{-1}\Big[\fe^{-it \langle \nabla\rangle}-\fe^{it \langle \nabla\rangle}\Big]\mathcal{S}_\varepsilon g.
\end{align*}

$\bullet$ Approximation of $h^\varepsilon$.
Recall that
$$
h^\varepsilon=\fe^{-it}w,
$$
then by \eqref{Duh-w-2}, we rewrite 
\begin{align}\label{apprx-h-1}
h^\varepsilon(t)
= &-\frac 12 \left[\fe^{-it (\langle \nabla\rangle+1)}\langle \nabla\rangle^{-1}\big(1-\langle \nabla\rangle\big)-\fe^{it (\langle \nabla\rangle-1)}\langle \nabla\rangle^{-1}\big(1+\langle \nabla\rangle\big)\right]\mathcal{S}_\varepsilon v_0\notag\\
&\qquad +\fe^{-it}\mathcal N(w),
\end{align}
with
$$
h^\varepsilon(0)=\mathcal{S}_\varepsilon v_0.
$$
Moreover, note that
\begin{align}\label{na-apprx}
\langle \nabla\rangle=1-\big(1+\langle\nabla\rangle\big)^{-1}\Delta.
\end{align}
Then from \eqref{apprx-h-1}, these allow us to rewrite $h^\varepsilon(t)$ as
\begin{align}\label{apprx-h-1-op2}
h^\varepsilon(t)
= &\fe^{it (\langle \nabla\rangle-1)}
\mathcal{S}_\varepsilon v_0 +\fe^{-it}\mathcal N_{\bot}(w),
\end{align}
where 
\begin{align*}
\mathcal N_{\bot}(w)
\triangleq &
\frac i2\fe^{-it}\big(\mathcal N_1(w)-\mathcal N_2(w)\big)
-\frac 12\varepsilon^2 \fe^{-it (\langle \nabla\rangle+1)}\mathcal{S}_\varepsilon(\Delta v_0)\\
&\quad -\frac 12\varepsilon^2\fe^{-it} \left[\fe^{-it \langle \nabla\rangle}-\fe^{it\langle \nabla\rangle}\right]\langle \nabla\rangle^{-1}\big(1+\langle \nabla\rangle\big)^{-1}\mathcal{S}_\varepsilon(\Delta v_0)\\
&\quad+\frac i2  \varepsilon^2\langle \nabla \rangle^{-1}\Big[\fe^{-it \langle \nabla\rangle}-\fe^{it \langle \nabla\rangle}\Big]\mathcal{S}_\varepsilon g.
\end{align*}
By \eqref{Nwj-L4-sd}, and noting that 
\begin{align}\label{est:S-vep-v0}
\big\||\nabla|^\gamma \mathcal{S}_\varepsilon v_0\big\|_{L^2_x}
\le C_{a_0, b_0}\delta_0  \varepsilon^{\gamma-s_c},
\end{align}
we  have that  for any $\gamma\ge -s_c$, 
\begin{align}\label{est:mathcal-N-h}
\big\||\nabla|^{\gamma}\mathcal N_{\bot}(w)\big\|_{X_{s_c}(\R)}
\le 
C_{a_0, b_0}\>\delta_0 \varepsilon^{\gamma} \big(\varepsilon^2+\delta_0^2\big).
\end{align}
A direct consequence of these two estimates and Lemma \ref{lem:strichartz} is that  
\begin{align}
\big\||\nabla|^{\gamma} h^\varepsilon\big\|_{X_{s_c}(\R)}
\le
 C_{a_0, b_0}\delta_0  \varepsilon^{\gamma}.
 \label{h-L4-sd}
\end{align}

$\bullet$ Approximation of $\partial_t  h^\varepsilon$.
Moreover,  by \eqref{eq:w-W12},
\begin{align*}
\partial_t h^\varepsilon
&=\fe^{-it}(\partial_t w-iw)\notag\\
&=\frac12\fe^{-it}\Big[\big(\langle\nabla\rangle+1\big)\mathcal W_1+\big(\langle\nabla\rangle-1\big)\mathcal W_2\Big].
%&=\frac12\fe^{-it}\Big[2\mathcal W_1+\big(\langle\nabla\rangle-1\big)\big(\mathcal W_1+\mathcal W_2\big)\Big].
\end{align*}
Then by \eqref{def:W12} and \eqref{eq:W-12-KG-intialdatum}, we obtain that
\begin{align}\label{apprx-ht-1}
\partial_t h^\varepsilon(t)
&=\frac12\fe^{-it}\Big[\big(\langle\nabla\rangle+1\big)\fe^{-it\langle\nabla\rangle}\mathcal W_{1,0}+\big(\langle\nabla\rangle-1\big)\fe^{it\langle\nabla\rangle}\mathcal W_{2,0}\Big]\notag\\
&\qquad + \frac12\fe^{-it}\Big[\big(\langle\nabla\rangle+1\big)\mathcal N_1(w)+\big(\langle\nabla\rangle-1\big)\mathcal N_2(w)\Big]\notag\\
&=\frac i2 \langle\nabla\rangle^{-1} \Delta \Big[\fe^{-it(\langle\nabla\rangle+1)}-\fe^{it(\langle\nabla\rangle-1)}\Big]\mathcal{S}_\varepsilon v_0 +\widetilde{\mathcal N}(w),
\end{align}
where we denote
\begin{align*}
\widetilde{\mathcal N}(w)\triangleq & \frac12 \fe^{-it}\Big[\big(\langle\nabla\rangle+1\big)\mathcal N_1(w)+\big(\langle\nabla\rangle-1\big)\mathcal N_2(w)\Big]\\
& +\frac 12\varepsilon^2  \langle\nabla\rangle^{-1} \Big[(\langle\nabla\rangle+1)\fe^{-it(\langle\nabla\rangle+1)}+(\langle\nabla\rangle-1)\fe^{it(\langle\nabla\rangle-1)}\Big]\mathcal{S}_\varepsilon v_1.
\end{align*}
%The following approximations play important roles in our analysis.
%First, we note from \eqref{v0-example} that  for any $\gamma\ge 0$,
%\begin{align}\label{est:S-vep-v0}
%\big\||\nabla|^\gamma \mathcal{S}_\varepsilon v_0\big\|_{L^2_x}\lesssim  \varepsilon^{\gamma-s_c}.
%\end{align}

Moreover, by \eqref{apprx-ht-1}, we have that 
\begin{align}\label{apprx-ht-1-op2}
\partial_t h^\varepsilon(t)
&=\frac i2  \Delta \Big[\fe^{-it(\langle\nabla\rangle+1)}-\fe^{it(\langle\nabla\rangle-1)}\Big]\mathcal{S}_\varepsilon v_0 +\widetilde{\mathcal N}_{\bot}(w),
\end{align}
where we denote
\begin{align*}
\widetilde{\mathcal N}_{\bot}(w)\triangleq &\widetilde{\mathcal N}(w) +\frac i2 \langle \nabla\rangle^{-1}\big(1+\langle \nabla\rangle\big)^{-1} \Delta^2 \Big[\fe^{-it(\langle\nabla\rangle+1)}-\fe^{it(\langle\nabla\rangle-1)}\Big]\mathcal{S}_\varepsilon v_0.
\end{align*}
By \eqref{Nwj-L4-sd} and \eqref{est:S-vep-v0}, we also have that   for any $\gamma\ge -s_c$, 
\begin{align}\label{est:mathcal-N-tilde-h}
\big\||\nabla|^{\gamma}\widetilde{\mathcal N}_{\bot}(w)\big\|_{X_{s_c}(\R)}
\le 
C_{a_0, b_0}\>\delta_0\varepsilon^{2+\gamma-s_c} \big(\varepsilon^2+\delta_0^2\big).
\end{align}

\subsubsection{Basic estimates on initial datum}
In this subsubsection, we give the basic estimates on $v_0$ and $f$. 
A direct computation gives that 
\begin{equation}\label{est:f}
\begin{aligned}
\|f\|_{H^\gamma_x}\sim &a_0^{s_c-\gamma}, \quad \mbox{ for any }\gamma\ge0;\\
\||x|^\gamma f\|_{L^2_x}\sim  &a_0^{s_c+\gamma}, \quad \mbox{ for any }\gamma\ge0;\\
\|f\|_{L^r_x}\sim  &a_0^{-1+\frac dr}, \quad \mbox{ for any } r\in [1, +\infty].
\end{aligned}
\end{equation}
Our analysis below also involves the estimates on $\Delta v_0$. Note that 
\begin{align*}
\Delta v_0=\delta_0 \fe^{-i\frac{b_0 |x|^2}{2}}\Big[c_1b_0\big(d+x\cdot \nabla\big)f+c_2b_0^2|x|^2f+\Delta f\Big],
\end{align*}
for some absolute constants $c_1\in \C,c_2\in \C$.
Applying \eqref{est:f}, we have that there exists an absolute constant $C>0$ such that 
\begin{align}\label{est:delta-v0}
\big\|\Delta v_0\big\|_{L^2_x}
\le  C\delta_0 \Big(b_0a_0^{s_c}+b_0^2a_0^{s_c+2}+a_0^{s_c-2}\Big).
\end{align}
%Moreover, we also need the following decomposition:
%$$
%\Delta v_0=v_{\dagger}+v_{\ddagger},
%$$
%where 
%$$
%v_{\dagger}\triangleq \delta_0 \fe^{-i\frac{b_0 |x|^2}{2}}\Big[c_1b_0\big(d+x\cdot \nabla\big)f+c_2b_0^2|x|^2f\Big];\quad 
%v_{\ddagger}\triangleq \delta_0 \fe^{-i\frac{b_0 |x|^2}{2}}\big(\Delta f\big).
%$$
Denote the space $\mathcal X_{s_c}(I)$ (which is a slight modification of $X_{s_c}(I)$) by its norm 
$$
\|f\|_{\mathcal X_{s_c}(I)}\triangleq \|f\|_{L^q_tL^r_x(I)},
$$
where the parameter pair $(q,r)$ satisfying that for arbitrary small constant $\epsilon>0$, 
$$
\frac1q=\frac14+\epsilon, \frac1r=\frac14-\epsilon, \mbox{ when } d=2;
\quad 
(q,r)=(4, 6), \mbox{ when } d=3.
$$
In the following, we need the estimate on 
$$
\big\|\fe^{is \langle\nabla\rangle}
\mathcal{S}_\varepsilon v_0\big\|_{\mathcal X_{s_c}(I)},
$$
which involves the Strichartz estimates of the Schr\"odinger flow on the oscillating data. The latter will be presented in the subsubsection below.  
% Then by \eqref{est:f} and Lemma \ref{lem:conterexmp-ID} below, we have that for any $I=[0,T_0]$, 
%\begin{equation}
%\begin{aligned}
%\big\|v_{\dagger}\big\|_{L^2}\le& C\delta_0 \Big(b_0a_0^{s_c}+b_0^2a_0^{s_c+2}\Big);\\
%\big\|\fe^{it\langle \nabla\rangle}\mathcal{S}_\varepsilon v_{\ddagger}\big\|_{\mathcal X_{s_c}(\varepsilon^{-2}I)}\le& C\delta_0 \Big(b_0a_0^{s_c}+b_0^2a_0^{s_c+2}+a_0^{s_c-2}\Big).
%\end{aligned}
%\end{equation}

\subsubsection{Estimates on the oscillating data}

Now we need the following result. 
%Denote
%\begin{align*}
%\big\|f\big\|_{\mathcal X_{s_c}(I)}\triangleq & \big\|f\big\|_{L^
%\frac{2(d+2)}{d}_{t}L^{\frac{d(d+2)}{2}}_x(I)}.
%\end{align*}
\begin{lem}
\label{lem:conterexmp-ID}
Given $T_0>0$. Let $T_0>0$ and $I=[0,T_0]$, then there exist constant $\varepsilon_0$ such that
for any   $\varepsilon\in [0,\varepsilon_0]$,
\begin{align*}
\Big\|\fe^{it\langle \nabla\rangle}\mathcal{S}_\varepsilon\Big(\fe^{-i\frac{b_0 |x|^2}{2}}f\Big) \Big\|_{\mathcal X_{s_c}(\varepsilon^{-2}I)}
\le &
\rho(b_0)
+
C_{a_0,b_0} \varepsilon^2 T_0^{1+\frac1q},
\end{align*}
where $\rho(b_0)\to 0$ when $b_0\to+\infty$.
\end{lem}
\begin{proof}
%We only consider the estimate on $\fe^{it\langle \nabla\rangle}$, since the estimate on $\fe^{-it\langle \nabla\rangle}$
%can be treated similarly.  
Denote the operator
$$
T_t  =\frac{\fe^{it\langle\nabla\rangle}-\fe^{it}\fe^{-\frac12 it\Delta}}{t\Delta^2}.
$$
Then by the Mihlin-H\"ormander Multiplier Theorem,  we have the $L^p\mapsto L^p$ boundedness which is uniform in time: for any $1<p<+\infty$,
\begin{align*}
\big\|T_t h\big\|_{L^p}\le C\|h\|_{L^p},
\end{align*}
where the constant $C>0$ is independent of $t$ and $h$.  Now we write
\begin{align}\label{apprx-linearKG}
\fe^{it\langle\nabla\rangle}=\fe^{it}\fe^{-\frac12 it\Delta}+t \Delta^2 T_t.
\end{align}
Then we have that
\begin{align}\label{est:Sbf-12}
\Big\|\fe^{ it\langle \nabla\rangle}& \mathcal{S}_\varepsilon\Big(\fe^{-i\frac{b_0 |x|^2}{2}}f\Big) \Big\|_{\mathcal X_{s_c}(\varepsilon^{-2}I)}\notag\\
\le  &
\Big\|\fe^{it}\fe^{-\frac12 it\Delta}\mathcal{S}_\varepsilon\Big(\fe^{-i\frac{b_0 |x|^2}{2}}f\Big) \Big\|_{\mathcal X_{s_c}(\varepsilon^{-2}I)}
+
\Big\|t \Delta^2 T_t \mathcal{S}_\varepsilon\Big(\fe^{-i\frac{b_0 |x|^2}{2}}f\Big)\Big\|_{\mathcal X_{s_c}(\varepsilon^{-2}I)}\notag\\
\lesssim &
\Big\|\fe^{-\frac12 it\Delta}\mathcal{S}_\varepsilon\Big(\fe^{-i\frac{b_0 |x|^2}{2}}f\Big) \Big\|_{\mathcal X_{s_c}(\varepsilon^{-2}I)}
+
\varepsilon^{-2}T_0 \Big\|\Delta^2\mathcal{S}_\varepsilon\Big(\fe^{-i\frac{b_0 |x|^2}{2}}f\Big)\Big\|_{\mathcal X_{s_c}(\varepsilon^{-2}I)}.
\end{align}
For the first term in \eqref{est:Sbf-12}, by changing the variable, we have that
\begin{align}\label{est:Sbf-scaling}
\Big\|\fe^{-\frac12 it\Delta}\mathcal{S}_\varepsilon\Big(\fe^{-i\frac{b_0 |x|^2}{2}}f\Big) \Big\|_{\mathcal X_{s_c}(\varepsilon^{-2}I)}
=&
\Big\|\fe^{-\frac12 it\Delta}\Big(\fe^{-i\frac{b_0 |x|^2}{2}}f\Big) \Big\|_{\mathcal X_{s_c}(I)}.
\end{align}
%Furthermore,
%\begin{align*}
%&\Big\|\fe^{-\frac12 it\Delta}\Big(\fe^{-i\frac{b_0 |x|^2}{2}}f\Big) \Big\|_{\mathcal X_{s_c}(I)}
%\le
%\Big\|\fe^{-\frac12 it\Delta}\Big(\fe^{-i\frac{b_0 |x|^2}{2}}f\Big) \Big\|_{S_c([0,\tilde \delta_0])}
%+
%\Big\|\fe^{-\frac12 it\Delta}\Big(\fe^{-i\frac{b_0 |x|^2}{2}}f\Big) \Big\|_{\mathcal X_{s_c}(\R)}.
%\end{align*}
%Choosing $\tilde \delta_0>0$ suitably small, we have that
%\begin{align}\label{est:Sbf-0}
%\Big\|\fe^{-\frac12 it\Delta}\Big(\fe^{-i\frac{b_0 |x|^2}{2}}f\Big) \Big\|_{S_c([0,\tilde \delta_0])}
%\le \frac13\delta_0,
%\end{align}
%hence, we only need to consider
%\begin{align}\label{est:Sbf-02}
%\Big\|\fe^{-\frac12 it\Delta}\Big(\fe^{-i\frac{b_0 |x|^2}{2}}f\Big) \Big\|_{\mathcal X_{s_c}(\R)}.
%\end{align}
%Moreover, we choose $b_0$ suitably large such that $b_0\tilde \delta_0\ge 2$.
Now we need the following formula (see e.g. \cite{CaWe-1992}),
\begin{align*}
\fe^{-\frac12 it\Delta}\Big(\fe^{-i\frac{b_0 |x|^2}{2}}f\Big)
=& (1+b_0t)^{-\frac d2}\> \fe^{\frac{ib_0}{4(1+b_0t)}|x|^2} \fe^{-it_{b_0} \Delta}f\Big(\frac x{1+b_0t}\Big),
\mbox{ with } t_{b_0}=\frac t{1+b_0t}.
\end{align*}
This gives that for any $r\in [1,+\infty]$,
\begin{align*}
\Big\|\fe^{-\frac12 it\Delta}\Big(\fe^{-i\frac{b_0 |x|^2}{2}}f\Big) \Big\|_{L^r_x}
=&
(1+b_0t)^{-\frac d2+\frac dr}\big\|\fe^{-\frac12 it_{b_0}\Delta}f \big\|_{L^r_x}.
\end{align*}
Inserting this into the right-hand side of \eqref{est:Sbf-scaling},  gives that
\begin{align*}
\Big\|\fe^{-\frac12 it\Delta}\Big(\fe^{-i\frac{b_0 |x|^2}{2}}f\Big) \Big\|_{\mathcal X_{s_c}(\R)}
\le &
\Big\|(1+b_0t)^{-s_c-\frac{2}q}\big\|\fe^{-\frac12 it_{b_0}\Delta}f \big\|_{L^r_x} \Big\|_{L^q_t(\R)}\\
\le &
\Big\|(1+b_0t)^{-\frac 2q}\big\|\fe^{-\frac12 it_{b_0}\Delta}f \big\|_{L^r_x} \Big\|_{L^q_t(\R)}\\
=&
\Big\|\big\|\fe^{-\frac12 it\Delta}f \big\|_{L^r_x} \Big\|_{L^q_t([0,\frac1b])}\\
\le & \big\|\fe^{-\frac12 it\Delta}f \big\|_{\mathcal X_{s_c}([0,\frac1b])}.
\end{align*}
Note that
$$
\big\|\fe^{-\frac12 it\Delta}f \big\|_{\mathcal X_{s_c}(\R)}\le \|f\|_{\dot H^{s_c}_x}\lesssim 1.
$$
Then we obtain that there exists a function $\rho:\R^+\mapsto \R^+$ such that
\begin{align}\label{est:Sbf-1}
\Big\|\fe^{-\frac12 it\Delta}\Big(\fe^{-i\frac{b_0 |x|^2}{2}}f\Big) \Big\|_{\mathcal X_{s_c}(\R)}
= &\rho(b_0),
\end{align}
where $\rho(b_0)\to 0$ when $b_0\to+\infty$.

For the second term in \eqref{est:Sbf-12}, we have that
\begin{align}\label{est:Sbf-2}
\Big\|\Delta^2\mathcal{S}_\varepsilon\Big(\fe^{-i\frac{b_0 |x|^2}{2}}f\Big)\Big\|_{\mathcal X_{s_c}(\varepsilon^{-2}I)}
\lesssim & \varepsilon^4\Big\|\Delta^2\Big(\fe^{-i\frac{b_0 |x|^2}{2}}f\Big)\Big\|_{\mathcal X_{s_c}(I)}\notag\\
\lesssim & \varepsilon^4 T_0^{\frac1q} \Big\|\Delta^2\Big(\fe^{-i\frac{b_0 |x|^2}{2}}f\Big)\Big\|_{L^r_x}\notag\\
\le  &C_{a_0,b_0} \varepsilon^4 T_0^{\frac1q}.
\end{align}
Therefore, inserting the estimates \eqref{est:Sbf-1} and \eqref{est:Sbf-2} into \eqref{est:Sbf-12}, we have that
\begin{align*}
\Big\|\fe^{ it\langle \nabla\rangle}& \mathcal{S}_\varepsilon\Big(\fe^{-i\frac{b_0 |x|^2}{2}}f\Big) \Big\|_{\mathcal X_{s_c}(\varepsilon^{-2}I)}
\le
\rho(b_0)
+
C_{a_0,b_0} \varepsilon^2 T_0^{1+\frac1q}.
\end{align*}
Thus we obtain the desired estimate.
\end{proof}

\subsubsection{Behavior of the remainder function}

Next, we consider the behavior of the remainder function $r^\varepsilon$. Applying the same precess and notations in Section \ref{sec:global}, we define  $\phi$ as in \eqref{def:phi} and conclude  that for any $\gamma\ge-s_c$,
\begin{align}\label{est: Xsc-phi}
\|\phi\|_{\mathcal X_{s_c}(\varepsilon^{-2}I)}\le C_{a_0,b_0}\delta_0 \varepsilon^{2};\quad
\||\nabla|^\gamma\phi\|_{X_{s_c}(\varepsilon^{-2}I)}\le C_{a_0,b_0}\delta_0\varepsilon^{2+\gamma}.
\end{align}
See Remark \ref{rem:3d-smalldata} for the explanation of the three dimensional cases.

%We argue for contradiction and suppose that
%\begin{align}\label{Ass:L2-phi-upperbdd}
%\|\phi\|_{L^\infty_tL^2_x(\R)}\le \delta_0^3\varepsilon^{2-s_c}.
%\end{align}
Denote $T^\varepsilon=\varepsilon^{-2}T$ and $T\in [T_0-1,T_0]$. As in \eqref{est:Duh-mod}, we write that
\begin{subequations}\label{est:Duh-mod'}
\begin{align}
\phi(T^\varepsilon)=&\fe^{iT^\varepsilon\langle\nabla\rangle} \phi(0)\label{Duh-mod'-0}\\
&\quad+ \int_0^{T^\varepsilon} \fe^{i(T^\varepsilon-s)\langle\nabla\rangle} \langle \nabla\rangle^{-1}G\big(h^{\varepsilon},R)\,ds\label{Duh-mod'-1}\\
&\quad +\int_0^{T^\varepsilon} \fe^{i(T^\varepsilon-s)\langle\nabla\rangle} \langle \nabla\rangle^{-1}P_{> 1}F_1\big(h^{\varepsilon}\big)\,ds\label{Duh-mod'-1'}\\
&\quad -   \frac{\fe^{3iT^\varepsilon} }{i(\langle\nabla\rangle-3)}\langle \nabla\rangle^{-1}P_{\le 1} \big(h^{\varepsilon}(T^\varepsilon)\big)^3
-   \frac{\fe^{3iT^\varepsilon} }{i(\langle\nabla\rangle+3)}\langle \nabla\rangle^{-1}P_{\le 1} \Big(\overline{h^{\varepsilon}(T^\varepsilon)}\Big)^3\label{Duh-mod'-2}\\
&\quad
+  \frac{\fe^{iT^\varepsilon\langle \nabla \rangle}}{i(\langle\nabla\rangle-3)}\langle \nabla\rangle^{-1}P_{\le 1} \Big(\mathcal{S}_\varepsilon v_0\Big)^3
+  \frac{\fe^{iT^\varepsilon\langle \nabla \rangle}}{i(\langle\nabla\rangle+3)}\langle \nabla\rangle^{-1}P_{\le 1} \Big(\overline{\mathcal{S}_\varepsilon v_0}\Big)^3\label{Duh-mod'-3}\\
&\quad+3\fe^{iT^\varepsilon\langle\nabla\rangle} \int_0^{T^\varepsilon} \fe^{-is(\langle\nabla\rangle-3)} \frac{1}{i(\langle\nabla\rangle-3)} \langle \nabla\rangle^{-1}P_{\le 1} \Big[\big(h^{\varepsilon}\big)^2\> \partial_s h^\varepsilon\Big]\,ds\label{Duh-mod'-4}\\
&\quad+3\fe^{iT^\varepsilon\langle\nabla\rangle} \int_0^{T^\varepsilon} \fe^{-is(\langle\nabla\rangle+3)} \frac{1}{i(\langle\nabla\rangle+3)} \langle \nabla\rangle^{-1}P_{\le 1} \Big[\big(\overline{h^\varepsilon})^2\> \partial_s\big(\overline{h^\varepsilon}\big)\Big]\,ds.\label{Duh-mod'-5}
\end{align}
\end{subequations}

 {\bf Lower bound of \eqref{Duh-mod'-0}+\eqref{Duh-mod'-3}.}

By \eqref{phi-initialdata} and \eqref{na-apprx}, we write \eqref{Duh-mod'-0} and  \eqref{Duh-mod'-3} as
$$
  \eqref{Duh-mod'-0}
  =-2\varepsilon^2\fe^{iT^\varepsilon\langle \nabla\rangle}\mathcal{S}_\varepsilon \big(\re(v_1)\big)+\mathcal R_1,
 $$
 and
$$
  \eqref{Duh-mod'-3}
  =\frac i4 \fe^{iT^\varepsilon\langle \nabla\rangle}\Big(\mathcal{S}_\varepsilon v_0\Big)^3+\mathcal R_2,
 $$
 where for $j=1,2$,
 \begin{align*}
\big\|\mathcal R_j\big\|_{L^\infty_t L^2_x(\R\times \R^2)}
\lesssim &
\varepsilon^2\big\|\Delta \mathcal{S}_\varepsilon \big(\re(v_1)\big)\big\|_{L^2_x}+\big\|\Delta \mathcal{S}_\varepsilon v_0\big\|_{L^2_x} \big\|\mathcal{S}_\varepsilon v_0\big\|_{L^\infty_x}^2\\
\le &
C_{a_0,b_0}\delta_0 \varepsilon^{4-s_c}.
\end{align*}
Hence, there give that
\begin{align*}
\eqref{Duh-mod'-0}+\eqref{Duh-mod'-3}
=&\fe^{iT^\varepsilon\langle \nabla\rangle}\Big[-2\varepsilon^2\mathcal{S}_\varepsilon \big(\re(v_1)\big)+\frac i2\Big(\mathcal{S}_\varepsilon v_0\Big)^3-\frac i4\Big(\overline{\mathcal{S}_\varepsilon v_0}\Big)^3\Big]+\mathcal R_1+\mathcal R_2\\
=&\varepsilon^2\fe^{iT^\varepsilon\langle \nabla\rangle}\mathcal{S}_\varepsilon \Big[-2\re(v_1)+\frac i2\big(v_0\big)^3-\frac i4\big(\overline{v_0}\big)^3\Big]
+\mathcal R_1+\mathcal R_2.
\end{align*}
From \eqref{lower-bound-v1v0}, we get that 
\begin{align*}
&\Big\|\fe^{iT^\varepsilon\langle \nabla\rangle}\mathcal{S}_\varepsilon \Big[-2\re(v_1)+\frac i2\big(v_0\big)^3-\frac i4\big(\overline{v_0}\big)^3\Big]\Big\|_{L^2_x}\\
= &\Big\|\mathcal{S}_\varepsilon \Big[-2\re(v_1)+\frac i2\big(v_0\big)^3-\frac i4\big(\overline{v_0}\big)^3\Big]\Big\|_{L^2_x}\\
= &\varepsilon^{-s_c}  \Big\|-2\re(v_1)+\frac i2\big(v_0\big)^3-\frac i4\big(\overline{v_0}\big)^3\Big\|_{L^2_x}\\
\ge &c_0\delta_0^3\varepsilon^{-s_c}\|f\|_{L^6_x}^3.
\end{align*}
where $c_0$ is determined in \eqref{lower-bound-v1v0}. Then by \eqref{est:f}, it further gives that there exists some absolute constant $\tilde c_0>0$ such that 
\begin{align*}
&\Big\|\fe^{iT^\varepsilon\langle \nabla\rangle}\mathcal{S}_\varepsilon \Big[-2\re(v_1)+\frac i2\big(v_0\big)^3-\frac i4\big(\overline{v_0}\big)^3\Big]\Big\|_{L^2_x}
\ge 
\tilde c_0 a_0^{-2+s_c}\delta_0^3 \varepsilon^{-s_c} .
\end{align*}
Therefore, it implies that
 \begin{align*}
\big\|\eqref{Duh-mod'-0}+\eqref{Duh-mod'-3}\big\|_{L^2_x}
\ge &\varepsilon^2\Big\|\fe^{iT^\varepsilon\langle \nabla\rangle}\mathcal{S}_\varepsilon \Big[-2\re(v_1)+\frac i2\big(v_0\big)^3-\frac i4\big(\overline{v_0}\big)^3\Big]\Big\|_{L^2_x}\\
&\qquad -\big\|\mathcal R_1\big\|_{L^2_x}-\big\|\mathcal R_2\big\|_{L^2_x}\\
\ge &\tilde c_0a_0^{-2+s_c}\delta_0^3\varepsilon^{2-s_c}-C_{a_0,b_0}\delta_0 \varepsilon^{4-s_c}.
\end{align*}

%
%
% Note that
% $$
%\phi(0)=-\varepsilon^2\langle \nabla\rangle^{-1} \mathcal{S}_\varepsilon v_1.
%$$
%Then by \eqref{v0-example} and \eqref{na-apprx}, we have that
%\begin{align*}
%\big\|\eqref{Duh-mod'-0}\big\|_{L^2_x}
%=&\big\|\phi(0)\big\|_{L^2_x}\\
%=&\varepsilon^2\big\|\langle \nabla\rangle^{-1} \mathcal{S}_\varepsilon v_1\big\|_{L^2_x}\\
%\ge &\varepsilon^2\big\|\mathcal{S}_\varepsilon v_1\big\|_{L^2_x}-C\varepsilon^2\big\|\Delta \mathcal{S}_\varepsilon v_1\big\|_{L^2_x}\\
%=&\varepsilon^{2-s_c}\big\|g\big\|_{L^2_x}+O(\varepsilon^{4-s_c}).
%\end{align*}

 {\bf Upper bound of \eqref{Duh-mod'-1}.}  To do this, we denote the space $\mathcal Z_{s_c}(I)$ by its norm 
$$
\|f\|_{\mathcal Z_{s_c}(I)}\triangleq \|f\|_{L^{q'}_tL^{r'}_x(I)},
$$
where the parameter pair $(q,r)$ satisfying that for arbitrary small constant $\epsilon>0$, 
$$
\frac1q=\frac12-\epsilon, \frac1r=\epsilon, \mbox{ when } d=2;
\quad 
(q,r)=(2, 6), \mbox{ when } d=3.
$$
 Then by Lemma \ref{lem:strichartz}, we have that
\begin{align*}
\big\|\eqref{Duh-mod'-1}\big\|_{L^2_x}
\lesssim & \Big\|G\big(h^{\varepsilon},R)\Big\|_{\mathcal Z_{s_c}(\varepsilon^{-2}I)}.
\end{align*}
%where $(q,r)$ is the pair of
%$$
%q=\frac{4(d+2)}{5d},\quad
%r=\frac{2d(d+2)}{d^2+d+8}.
%$$
Then by H\"older's inequality, we further get that
\begin{align*}
\big\|\eqref{Duh-mod'-1}\big\|_{L^2_x}
\lesssim &\big\|R\big\|_{L^\infty_tL^2_x(\varepsilon^{-2}I)} \sum\limits_{j=0}^2\big\|h^{\varepsilon}\big\|_{\mathcal X_{s_c}(\varepsilon^{-2}I)}^j\big\|R\big\|_{\mathcal X_{s_c}(\varepsilon^{-2}I)}^{2-j}\\
\lesssim &\big\|\phi\big\|_{L^\infty_tL^2_x(\varepsilon^{-2}I)} \sum\limits_{j=0}^2\big\|h^{\varepsilon}\big\|_{\mathcal X_{s_c}(\varepsilon^{-2}I)}^j\big\|\phi\big\|_{\mathcal X_{s_c}(\varepsilon^{-2}I)}^{2-j}.
\end{align*}
By \eqref{apprx-h-1-op2} and Lemma \ref{lem:conterexmp-ID}, we have that 
\begin{align*} 
\big\|h^{\varepsilon}\big\|_{\mathcal X_{s_c}(\varepsilon^{-2}I)}
\lesssim &
\big\|\fe^{it (\langle \nabla\rangle-1)}
\mathcal{S}_\varepsilon v_0\big\|_{\mathcal X_{s_c}(\varepsilon^{-2}I)} +\big\|\fe^{-it}\mathcal N_{\bot}(w)\big\|_{\mathcal X_{s_c}(\varepsilon^{-2}I)}\\
\lesssim &
\big\|\fe^{it \langle \nabla\rangle}
\mathcal{S}_\varepsilon v_0\big\|_{\mathcal X_{s_c}(\varepsilon^{-2}I)} +\big\|\mathcal N_{\bot}(w)\big\|_{\mathcal X_{s_c}(\varepsilon^{-2}I)}\\
\le &
\delta_0\rho(b_0)
+C_{a_0, b_0,T_0}\>\delta_0 \big(\varepsilon^2+\delta_0^2\big).
\end{align*}
Applying this estimate and  \eqref{est: Xsc-phi}, we obtain that
\begin{align*}
\big\|\eqref{Duh-mod'-1}\big\|_{L^2_x}
\le  &\delta_0^3\varepsilon^{2-s_c}\Big(\rho(b_0)
+C_{a_0, b_0,T_0}\> \big(\varepsilon^2+\delta_0^2\big)\Big),
\end{align*}
where as above, the constant $C_{a_0, b_0,T_0}>0$ only depends on $a_0, b_0$ and $T_0$.

 {\bf Upper bound of \eqref{Duh-mod'-1'}.} Similar as above, by Lemma \ref{lem:strichartz} and \eqref{est: Xsc-phi},  we have that
\begin{align*}
\big\|\eqref{Duh-mod'-1'}\big\|_{L^2_x}
\lesssim &\Big\|P_{> 1}F_1\big(h^{\varepsilon}\big)\Big\|_{L^q_{t}L^r_x(\varepsilon^{-2}I)}
\end{align*}
where $(q,r)$ is the pair of
$$
q=\frac{4(d+2)}{5d},\quad
r=\frac{2d(d+2)}{d^2+d+8}.
$$
Then by H\"older's and Bernstein's inequalities, we have that 
\begin{align*}
\big\|\eqref{Duh-mod'-1'}\big\|_{L^2_x}
\lesssim &\big\|P_{\gtrsim 1}h^{\varepsilon}\big\|_{L^\infty_tL^2_x(\varepsilon^{-2}I)}^\frac12\big\|P_{\gtrsim 1}|\nabla|^{s_c} h^{\varepsilon}\big\|_{L^\frac{2(d+2)}{d}_{tx}(\varepsilon^{-2}I)}^\frac12\big\||\nabla|^{s_c}h^{\varepsilon}\big\|_{L^
\frac{2(d+2)}{d}_{tx}(\varepsilon^{-2}I)}^2\\
\lesssim &\big\||\nabla|^{4} h^{\varepsilon}\big\|_{L^\infty_tL^2_x(\varepsilon^{-2}I)}^\frac12\big\||\nabla|^{4} h^{\varepsilon}\big\|_{L^\frac{2(d+2)}{d}_{tx}(\varepsilon^{-2}I)}^\frac12\big\||\nabla|^{s_c}h^{\varepsilon}\big\|_{L^
\frac{2(d+2)}{d}_{tx}(\varepsilon^{-2}I)}^2\\
\lesssim &\big\||\nabla|^{4-s_c}h^{\varepsilon}\big\|_{X_{s_c}(\varepsilon^{-2}I)}\big\|h^{\varepsilon}\big\|_{X_{s_c}(\varepsilon^{-2}I)}^2.
\end{align*}
By \eqref{h-L4-sd}, we obtain that 
\begin{align*}
\big\|\eqref{Duh-mod'-1'}\big\|_{L^2_x}
\le  C_{a_0, b_0}\delta_0^3  \varepsilon^{4-s_c}.
\end{align*}

 {\bf Upper bound of \eqref{Duh-mod'-2}.}
 By \eqref{est:bound-nabal3}, \eqref{apprx-h-1-op2} and \eqref{est:mathcal-N-h}, we have that
\begin{align}\label{est:Duh-mod'-2-1}
\big\| \eqref{Duh-mod'-2}\big\|_{L^2_x}
\lesssim &
\big\|h^{\varepsilon}(T^\varepsilon)\big\|_{L^6_x}^3\notag\\
\lesssim &
\big\|\fe^{iT^\varepsilon\langle \nabla\rangle}\mathcal{S}_\varepsilon v_0\big\|_{L^6_x}^3
+\big\|\fe^{-it}\mathcal N_{\bot}(w)\big\|_{L^6_x}^3.
\end{align}
 By Lemma \ref{lem:dispersive},  we have that
 \begin{align*}
 \big\|\fe^{iT^\varepsilon\langle \nabla\rangle}\mathcal{S}_\varepsilon v_0\big\|_{L^6_x}
 \lesssim & \langle T^\varepsilon\rangle^{-\frac d3}  \big\|\langle \nabla\rangle^{\frac{d+2}{3}}\mathcal{S}_\varepsilon v_0\big\|_{L^\frac65_x}.
 \end{align*}
 Note that
 \begin{align*}
\big\|\langle \nabla\rangle^{\frac{d+2}{3}}\mathcal{S}_\varepsilon v_0\big\|_{L^\frac65_x}
\le &
\big\|\mathcal{S}_\varepsilon v_0\big\|_{L^\frac65_x}
+\big\| |\nabla|^{\frac{d+2}{3}}\mathcal{S}_\varepsilon v_0\big\|_{L^\frac65_x}\\
=  &
 \varepsilon^{1-\frac56 d} \big\|v_0\big\|_{L^\frac65_x}
 +\varepsilon^{\frac53-\frac12 d} \big\| |\nabla|^{\frac{d+2}{3}}v_0\big\|_{L^\frac65_x}.
% \lesssim  &
% \varepsilon^{1-\frac56 d} \big\|f\big\|_{L^\frac65_x}
% +\varepsilon^{1-\frac56 d}\big(\varepsilon b_0\big)^{\frac{d+2}{3}} \big\| \langle\nabla\rangle^{\frac{d+2}{3}}(\langle x\rangle^2f)\big\|_{L^\frac65_x}.
\end{align*}
Note that by \eqref{est:f}, 
\begin{align*}
\big\|v_0\big\|_{L^\frac65_x}
\le \delta_0  \big\|f\big\|_{L^\frac65_x}
\le C\delta_0;\quad 
 \big\| |\nabla|^{\frac{d+2}{3}}v_0\big\|_{L^\frac65_x}
 \le C_{a_0,b_0}.
\end{align*}
We further have that 
 \begin{align*}
\big\|\langle \nabla\rangle^{\frac{d+2}{3}}\mathcal{S}_\varepsilon v_0\big\|_{L^\frac65_x}
 \le  &
C \varepsilon^{1-\frac56 d}
 +C_{a_0,b_0}\varepsilon^{\frac53-\frac12 d}.
\end{align*}
 Set $T_0\ge 2$ and thus $T\ge1$,
 then it reduces that
  \begin{align*}
 \big\|\fe^{iT^\varepsilon\langle \nabla\rangle}\mathcal{S}_\varepsilon v_0\big\|_{L^6_x}
 \le &  \big(T^\varepsilon\big)^{-\frac d3} C \varepsilon^{1-\frac56 d}
 +C_{a_0,b_0}\varepsilon^{\frac53-\frac12 d}\\
 \le &  T_0^{-\frac d3} \varepsilon^{\frac13(2-s_c)}
 \Big(C+C_{a_0,b_0}\varepsilon^{\frac23}\Big).
 \end{align*}
 Moreover, 
 \begin{align*}
 \big\|\fe^{-it}\mathcal N_{\bot}(w)\big\|_{L^6_x}
 \lesssim 
  \big\||\nabla|^\frac d3\mathcal N_{\bot}(w)\big\|_{L^6_x}
\le C_{a_0,b_0}\delta_0^3 \varepsilon^{1-\frac d6}.
 \end{align*}
 Inserting these two estimates above into \eqref{est:Duh-mod'-2-1}, we obtain that
 \begin{align*}
\big\| \eqref{Duh-mod'-2}\big\|_{L^2_x}
\le &
T_0^{-d}\delta_0^3 \varepsilon^{2-s_c}
 \Big(C+C_{a_0,b_0}(\delta_0^2+\varepsilon^2\big)\Big).
\end{align*}

 {\bf Upper bound of \eqref{Duh-mod'-4} and  \eqref{Duh-mod'-5}.}
 We only focus on the term \eqref{Duh-mod'-4}, since the other term  \eqref{Duh-mod'-5} can be treated similarly. 
 
 By \eqref{apprx-h-1-op2} and  \eqref{apprx-ht-1-op2}, we rewrite it as 
 \begin{align*}
\eqref{Duh-mod'-4}
 = &
 \frac 32i \fe^{iT^\varepsilon\langle\nabla\rangle} \int_0^{T^\varepsilon} \fe^{-is(\langle\nabla\rangle-3)} \frac{1}{i(\langle\nabla\rangle-3)} \langle \nabla\rangle^{-1}P_{\le 1}\notag\\
 &\quad  \Bigg\{\Big[\fe^{is (\langle \nabla\rangle-1)}
\mathcal{S}_\varepsilon v_0 +\fe^{-is}\mathcal N_{\bot}(w)\Big]^2\> \Big[ \Delta \Big(\fe^{-is(\langle\nabla\rangle+1)}-\fe^{is(\langle\nabla\rangle-1)}\Big)\mathcal{S}_\varepsilon v_0 +\widetilde{\mathcal N}_{\bot}(w)\Big]\Bigg\}\,ds.
\end{align*}
Therefore, it gives that 
\begin{align*}
\eqref{Duh-mod'-4}
= &
\mathcal P_{\bot}+\mathcal R_{\bot},
\end{align*}
where the  terms $\mathcal P_{\bot}$ and $\mathcal R_{\bot}$ are defined as 
\begin{align*}
\mathcal P_{\bot}\triangleq &
- \frac 34 \fe^{iT^\varepsilon\langle\nabla\rangle} \int_0^{T^\varepsilon} \fe^{-is(\langle\nabla\rangle-3)} 
\langle \nabla\rangle^{-1}P_{\le 1}\Bigg\{\Big[\fe^{is (\langle\nabla\rangle-1)}
\mathcal{S}_\varepsilon v_0\Big]^2\\
&\qquad \qquad\cdot  \Big[ \Delta \Big(\fe^{-is(\langle\nabla\rangle+1)}-\fe^{is(\langle\nabla\rangle-1)}\Big)\mathcal{S}_\varepsilon v_0 \Big]\Bigg\}\,ds;\\
\mathcal R_{\bot}
\triangleq & 
\frac 32i \fe^{iT^\varepsilon\langle\nabla\rangle} \int_0^{T^\varepsilon} \fe^{-is\langle\nabla\rangle} 
T(\nabla) \Bigg\{\Big[\fe^{is \langle\nabla\rangle}
\mathcal{S}_\varepsilon v_0\Big]^2\> \Big[ \Delta \Big(\fe^{-is(\langle\nabla\rangle+1)}-\fe^{is(\langle\nabla\rangle-1)}\Big)\mathcal{S}_\varepsilon v_0 \Big]\Bigg\}\,ds\\
&\quad+
 \frac 32i \fe^{iT^\varepsilon\langle\nabla\rangle} \int_0^{T^\varepsilon} \fe^{-is\langle\nabla\rangle} \frac{1}{i(\langle\nabla\rangle-3)} \langle \nabla\rangle^{-1}P_{\le 1}\notag\\
 &\qquad  \Bigg\{\Big[2\fe^{is \langle \nabla\rangle}
\mathcal{S}_\varepsilon v_0\cdot \mathcal N_{\bot}(w) +\big(\mathcal N_{\bot}(w)\big)^2\Big]
 \Big[ \Delta \Big(\fe^{-is\langle\nabla\rangle}-\fe^{is\langle\nabla\rangle}\Big)\mathcal{S}_\varepsilon v_0 +\widetilde{\mathcal N}_{\bot}(w)\Big]\Bigg\}\,ds,
 \end{align*}
 where $T(\nabla)=\frac{1}{i(\langle\nabla\rangle-3)} P_{\le 1}+\frac1{2i}$. 
 Note that for any $1<p<+\infty$, 
\begin{align*}
\big\|T(\nabla)f\big\|_{L^p_x}
\lesssim \|\Delta f\|_{L^p_x}.
\end{align*}
Then using this estimate, Lemma \ref{lem:strichartz},  \eqref{est:mathcal-N-h} and \eqref{est:mathcal-N-tilde-h}, we have that 
\begin{align*}
\big\|\mathcal R_{\bot}\big\|_{L^2_x}
\le  &
C_{a_0, b_0}\delta_0^3 \varepsilon^{2-s_c} 
\big(\delta_0^2+\varepsilon^2\big). 
\end{align*}
 
Now we consider $\mathcal P_{\bot}$. To do this,  we split  it as 
$$
\mathcal P_{\bot}=\mathcal P_{\bot,1}+\mathcal P_{\bot,2}, 
$$
where 
\begin{align*}
\mathcal P_{\bot,1}\triangleq &
- \frac 34\varepsilon^2 \fe^{iT^\varepsilon\langle\nabla\rangle} \int_0^{T^\varepsilon} \fe^{-is\langle\nabla\rangle} 
 \langle \nabla\rangle^{-1}P_{\le 1}\Big[\Big(\fe^{is \langle\nabla\rangle}
\mathcal{S}_\varepsilon v_0\Big)^2\> \fe^{-is\langle\nabla\rangle}\mathcal{S}_\varepsilon\big(\Delta v_0\big) \Big]\,ds;\\
\mathcal P_{\bot,2}\triangleq &
- \frac 34\varepsilon^2 \fe^{iT^\varepsilon\langle\nabla\rangle} \int_0^{T^\varepsilon} \fe^{-is(\langle\nabla\rangle-3)} 
 \langle \nabla\rangle^{-1}P_{\le 1}\Big[\Big(\fe^{is (\langle\nabla\rangle-1)}
\mathcal{S}_\varepsilon v_0\Big)^2\> \fe^{is(\langle\nabla\rangle-1)}\mathcal{S}_\varepsilon\big(\Delta v_0\big) \Big]\,ds.
\end{align*}

For $\mathcal P_{\bot,1}$, by Lemma \ref{lem:strichartz}, we have that 
\begin{align*}
\big\|\mathcal P_{\bot,1}\big\|_{L^2_x}
\lesssim &
\big\|\fe^{is \langle\nabla\rangle}
\mathcal{S}_\varepsilon v_0\big\|_{\mathcal X_{s_c}([0,T^\varepsilon])}^2 \big\|\fe^{is \langle\nabla\rangle}
\mathcal{S}_\varepsilon \big(\Delta v_0\big)\big\|_{L^\infty_tL^2_x([0,T^\varepsilon])}\\
\lesssim &
\delta_0^3\varepsilon^{2-s_c}\Big\|\fe^{it\langle \nabla\rangle}\mathcal{S}_\varepsilon\Big(\fe^{-i\frac{b_0 |x|^2}{2}}f\Big) \Big\|_{\mathcal X_{s_c}([0,T^\varepsilon])}^2 \big\|\Delta v_0\big\|_{L^2_x}.
\end{align*}
Therefore, by \eqref{est:delta-v0} and Lemma \ref{lem:conterexmp-ID}, we obtain that 
\begin{align*}
\big\|\mathcal P_{\bot,1}\big\|_{L^2_x}
\lesssim &
\delta_0^3\varepsilon^{2-s_c}\Big(\rho(b_0)
+
C_{a_0,b_0} \varepsilon^2 T_0^{1+\frac1q}\Big)\Big(b_0a_0^{s_c}+b_0^2a_0^{s_c+2}+a_0^{s_c-2}\Big)\\
\le& \delta_0^3\varepsilon^{2-s_c}\Big(\rho(b_0)a_0^{s_c-2}+C_{b_0}a_0^{s_c}+C_{a_0,b_0,T_0} \varepsilon^2\Big).
\end{align*}

For $\mathcal P_{\bot,1}$, integration by parts, we have that 
\begin{align*}
\mathcal P_{\bot,2}= &
\frac 34\varepsilon^2\fe^{iT^\varepsilon\langle\nabla\rangle}\Bigg[ 
\frac{\fe^{-is(\langle\nabla\rangle-3)} }{(\langle\nabla\rangle-3)\langle \nabla\rangle} P_{\le 1}\Big[\Big(\fe^{is (\langle\nabla\rangle-1)}
\mathcal{S}_\varepsilon v_0\Big)^2\> \fe^{is(\langle\nabla\rangle-1)}\mathcal{S}_\varepsilon\big(\Delta v_0\big) \Big|_0^{T^\varepsilon} \\
&\quad - \int_0^{T^\varepsilon} 
\frac{\fe^{-is(\langle\nabla\rangle-3)} }{(\langle\nabla\rangle-3)\langle \nabla\rangle}P_{\le 1} \partial_s\Big[\Big(\fe^{is (\langle\nabla\rangle-1)}
\mathcal{S}_\varepsilon v_0\Big)^2\> \fe^{is(\langle\nabla\rangle-1)}\mathcal{S}_\varepsilon\big(\Delta v_0\big) \Big]\,ds\Bigg].
\end{align*}
Since 
$$
\partial_s \fe^{is(\langle\nabla\rangle-1)}=i(\langle\nabla\rangle-1) \fe^{is(\langle\nabla\rangle-1)}
=\frac{i|\nabla|^2}{1+\langle\nabla\rangle}\fe^{is(\langle\nabla\rangle-1)},
$$
by Lemma \ref{lem:strichartz}, we have that 
\begin{align*}
\big\|\mathcal P_{\bot,2}\big\|_{L^2_x}
\lesssim &
\delta_0^3C_{a_0,b_0} \varepsilon^4.
\end{align*}
Combining with the estimates on $\mathcal P_{\bot,1}, \mathcal P_{\bot,2}$ and $\mathcal R_{\bot}$ above, we obtain that 
 \begin{align*}
\big\| \eqref{Duh-mod'-4}\big\|_{L^2_x}
\le  &
\delta_0^3\varepsilon^{2-s_c}\Big(\rho(b_0)a_0^{s_c-2}+C_{b_0}a_0^{s_c}+C_{a_0, b_0}\delta_0^2+C_{a_0,b_0,T_0} \varepsilon^2\Big).
\end{align*}
Arguing similarly, we also have that 
 \begin{align*}
\big\| \eqref{Duh-mod'-5}\big\|_{L^2_x}
\le  &
\delta_0^3\varepsilon^{2-s_c}\Big(\rho(b_0)a_0^{s_c-2}+C_{b_0}a_0^{s_c}+C_{a_0, b_0}\delta_0^2+C_{a_0,b_0,T_0} \varepsilon^2\Big).
\end{align*}

Combining the estimates on \eqref{est:Duh-mod'},   we obtain that
for any $T\in [T_0-1,T_0]$ and  any $\varepsilon\in (0,\varepsilon_0]$,
 \begin{align*} 
\big\|\phi(T^\varepsilon)\big\|_{L^2_x}
\ge &\delta_0^3\varepsilon^{2-s_c}\Big[ \tilde c_0a_0^{s_c-2}
-\Big(\rho(b_0)a_0^{s_c-2}+C_{b_0}a_0^{s_c}+T_0^{-d}+C_{a_0, b_0}\delta_0^2+C_{a_0,b_0,T_0}\delta_0^{-2} \varepsilon^2\Big)\Big].
\end{align*}
Choosing $b_0$ suitably large such that $\rho(b_0)\le \frac14\tilde c_0$ and $T_0\ge 2$, and by the relationship \eqref{relationships}, we obtain that 
 \begin{align}\label{phi-lbd}
\big\|\phi(T^\varepsilon)\big\|_{L^2_x}
\ge &\frac12 \tilde c_0a_0^{s_c-2}\delta_0^3\varepsilon^{2-s_c}.
\end{align}

%By \eqref{phi-varphi}, it infers that  either
%\begin{align*}
%\big\| R(T^\varepsilon)\big\|_{L^2_x}
%\ge &\frac14c_0\varepsilon^{2-s_c}, 
%\end{align*}
%or 
%\begin{align*}
%\big\|\partial_t R(T^\varepsilon)\big\|_{L^2_x}
%\ge & \frac14c_0\varepsilon^{2-s_c}.
%\end{align*}
Now we are in the position to prove Theorem \ref{thm:main1-optimal}. Argue for contradiction, and suppose that for any $T\in [T_0-1,T_0]$,  
\begin{align}\label{Assup-contrid}
\big\| R(T^\varepsilon)\big\|_{L^2_x}
< & \frac1{8} \tilde c_0a_0^{s_c-2}\delta_0^3\varepsilon^{2-s_c}.
\end{align}
%Otherwise, there exists  $T\in [T_0-1,T_0]$, such that 
%$$
%\big\| R(T)\big\|_{L^2_x}
%\ge   \frac1{8} \tilde c_0a_0^{s_c-2}\delta_0^3\varepsilon^{2-s_c}.
%$$
%Then it establishes the desired conclusion. 
Applying \eqref{phi-varphi}, it is equivalent to 
\begin{align*}
\big\|\im \>\phi(T^\varepsilon)\big\|_{L^2_x}
< & \frac1{8} \tilde c_0a_0^{s_c-2}\delta_0^3\varepsilon^{2-s_c},
\end{align*}
Then from \eqref{phi-lbd},  it leads that for any  $T\in [T_0-1,T_0]$, 
\begin{align*}
\big\|\re \>\phi(T^\varepsilon)\big\|_{L^2_x}
\ge & \frac{3}{8} \tilde c_0a_0^{s_c-2}\delta_0^3\varepsilon^{2-s_c}.
\end{align*}
Using \eqref{phi-varphi} again, it drives that 
\begin{align*}
\big\|\partial_t R(T^\varepsilon)\big\|_{L^2_x}
\ge & \frac{3}{8} \tilde c_0a_0^{s_c-2}\delta_0^3\varepsilon^{2-s_c}.
\end{align*}
By Mean Value Theorem, we have that 
$$
R(T^\varepsilon+1)= R(T^\varepsilon)+ \partial_tR(\tilde T^\varepsilon),\quad \mbox{for some } \tilde T^\varepsilon\in [T^\varepsilon,T^\varepsilon+1]. 
$$
Therefore, 
\begin{align*}
\big\|R(T^\varepsilon+1)\big\|_{L^2_x}
\ge &  \big\|\partial_tR(\tilde T^\varepsilon)\big\|_{L^2_x}-\big\|R(T^\varepsilon)\big\|_{L^2_x}\\
\ge &  \frac{3}{8} \tilde c_0a_0^{s_c-2}\delta_0^3\varepsilon^{2-s_c}- \frac1{8} \tilde c_0a_0^{s_c-2}\delta_0^3\varepsilon^{2-s_c}\\
=&\frac{1}{4} \tilde c_0a_0^{s_c-2}\delta_0^3\varepsilon^{2-s_c}\\
>& \frac1{8}\tilde c_0a_0^{s_c-2}\delta_0^3\varepsilon^{2-s_c}.
\end{align*}
This is a contradiction with \eqref{Assup-contrid} and thus there exists  $T\in [T_0-1,T_0]$, such that 
$$
\big\| R(T^\varepsilon)\big\|_{L^2_x}
\ge   \frac1{8} \tilde c_0a_0^{s_c-2}\delta_0^3\varepsilon^{2-s_c}.
$$
Scaling back, we obtain that
  $$
\big\|r(T)\big\|_{L^2_x}\ge \frac18\tilde c_0a_0^{s_c-2}\delta_0^3\varepsilon^{2}.
$$
This completes the proof of Theorem \ref{thm:main1-optimal}.

  \vskip 1.5cm

\section{Single modulated limit}\label{sec:singe-modulate}
 \vskip .5cm

 In this section, we shall give the proof of Theorem \ref{thm:main3}.

  \subsection{Formulation}\label{sec:Th3-formula}

 Denote that $r=\widetilde{v}^\varepsilon-\widetilde{v}$, then it is the solution of
 \EQn{
	\label{eq:kg-r-3}
	\left\{ \aligned
	&\varepsilon^2\partial_{tt} r +2i\partial_t r- \De r  +\varepsilon^2 \partial_{tt} \widetilde{v}+\big|\widetilde{v}^\varepsilon\big|^2\widetilde{v}^\varepsilon-|\widetilde{v}|^2\widetilde{v}
	=0, \\
	& r(0,x) = r_0(x),\quad  \partial_tr(0,x) = \varepsilon^{-2} r_1(x),
	\endaligned
	\right.
}
where 
$$
r_0\triangleq \widetilde{v}^\varepsilon_0-\widetilde{v}_0;\quad 
r_1\triangleq \widetilde{v}^\varepsilon_1(x)-\frac{\varepsilon^2}{2i}(\Delta \widetilde{v}_0-\mu |\widetilde{v}_0|^2\widetilde{v}_0).
$$
Moreover,  we denote
$$
R=\fe^{it} \mathcal{S}_\varepsilon  r,
$$
then
 \EQn{
	\label{eq:kg-R-3}
	\left\{ \aligned
	&\partial_{tt} R - \De R + R +\varepsilon^4 \fe^{it}  \mathcal{S}_\varepsilon( \partial_{tt} \widetilde{v})+\fe^{it} \Big(\big|\mathcal{S}_\varepsilon \widetilde{v}^\varepsilon\big|^2\mathcal{S}_\varepsilon \widetilde{v}^\varepsilon-|\mathcal{S}_\varepsilon \widetilde{v}|^2\mathcal{S}_\varepsilon \widetilde{v}\Big)
	=0, \\
	& R(0,x) = \mathcal{S}_\varepsilon r_0,\quad  \partial_t R(0,x) =i \mathcal{S}_\varepsilon r_0+ \mathcal{S}_\varepsilon r^\varepsilon_1.
	\endaligned
	\right.
}

Similar as Section \ref{sec:LR-wave}, we  denote
\EQ{
	\left\{ \aligned
	&\mathcal R_1=\langle \nabla \rangle^{-1}\big(\partial_t-i \langle \nabla\rangle\big)R, \\
	&\mathcal R_2=\langle \nabla \rangle^{-1}\big(\partial_t+i \langle \nabla\rangle\big)R.
	\endaligned
	\right.
}
Then
\EQn{
	\label{eq:R-R12}
	\left\{ \aligned
	&R=\frac i2\big(\mathcal R_1-\mathcal R_2\big), \\
	&\partial_t R=\frac 12\langle\nabla\rangle\big(\mathcal R_1+\mathcal R_2\big),
	\endaligned
	\right.
}
and the initial datum
\EQn{
	\label{eq:R-12-KG-intialdatum}
	\left\{ \aligned
&\mathcal R_1(0)=\mathcal R_{1,0}\triangleq \langle \nabla \rangle^{-1}\Big[i\big(1- \langle \nabla\rangle\big)\mathcal{S}_\varepsilon r_0+\mathcal{S}_\varepsilon r^\varepsilon_1\Big];\\
&\mathcal R_2(0)=\mathcal R_{2,0}\triangleq \langle \nabla \rangle^{-1}\Big[i\big(1+ \langle \nabla\rangle\big)\mathcal{S}_\varepsilon r_0+\mathcal{S}_\varepsilon r^\varepsilon_1\Big].
	\endaligned
	\right.
}
Moreover,  $\mathcal R_j$ obeys the following Duhamel's formula,
\begin{align}\label{Duh-Rj-th3}
\mathcal R_j(t)=&\fe^{\mp i(t-t_0)\langle \nabla\rangle} \mathcal R_j(t_0)\notag\\
& -\int_{t_0}^t \fe^{\mp i(t-s)\langle \nabla\rangle}\fe^{is} \langle\nabla \rangle^{-1} \Big[\varepsilon^4 \mathcal{S}_\varepsilon( \partial_{tt} \widetilde{v})+\big|\mathcal{S}_\varepsilon \widetilde{v}^\varepsilon\big|^2\mathcal{S}_\varepsilon \widetilde{v}^\varepsilon-|\mathcal{S}_\varepsilon \widetilde{v}|^2\mathcal{S}_\varepsilon \widetilde{v}\Big]\,ds.
\end{align}

%
%\begin{align}\label{Duh-Rj-th3}
%\mathcal R_j(t)=&\fe^{\mp i(t-t_0)\langle \nabla\rangle} \mathcal R_j(t_0)\notag\\
%& -\int_{t_0}^t \fe^{\mp i(t-s)\langle \nabla\rangle}\langle\nabla \rangle^{-1} \Big[\varepsilon^4 \mathcal{S}_\varepsilon( \partial_{tt} v)+3\big|\mathcal{S}_\varepsilon v^\varepsilon\big|^2\mathcal{S}_\varepsilon v^\varepsilon-3|\mathcal{S}_\varepsilon v|^2\mathcal{S}_\varepsilon v\Big]\,ds.
%\end{align}

\subsection{Proof of Theorem \ref{thm:main3}}

  \subsubsection{Regular cases}\label{sec:Thm3-case1-1}

  In this subsubsection, we consider the case when $\alpha=4$.  The main result is
  \begin{prop}\label{prop:Thm3-case1-1}
  Let $d=2,3$, $T>0$ and $\beta\in [1,2]$, then there exists $\delta_0>0$ such that the following property hold.  Assume that
  \begin{itemize}
  \item[(1)] $v_0\in H^{4+s_c}(\R^d),\widetilde{v}^\varepsilon_0\in H^\beta_x(\R^d), v_1^\varepsilon\in H^{s_c}(\R^d)$;
  \item[(2)] $(\widetilde{v}^\varepsilon_0-\widetilde{v}_0,\widetilde{v}_1^\varepsilon)\to (0,0) \mbox{ in } H^{s_c}(\R^d).$
  \end{itemize}
  When $d=3$,  additionally assume that
    \begin{itemize}
    \item[(3)] $\varepsilon^2 T\|\widetilde{v}\|_{L^\infty_t\dot H^{\frac92}_x}\le \delta_0$.
  \end{itemize}
Let $\widetilde{v}^\varepsilon, \widetilde{v}$ be the corresponding solutions to  \eqref{eq:nls-wave-N} and \eqref{eq:nls-1} respectively, then
for any $\beta\ge 1$,
  \begin{align*}
\sup\limits_{t\in [0,T]}\big\|\widetilde{v}^\varepsilon(t)-\widetilde{v}(t)\big\|_{L^2_x}
\le  &C\Big( \|(\widetilde{v}^\varepsilon_0-\widetilde{v}_0,\widetilde{v}_1^\varepsilon)\|_{L^2_x\times L^2_x}+\varepsilon^\beta\big\|\widetilde{v}^\varepsilon_0-\widetilde{v}_0\big\|_{H^\beta_x} \\
&\quad+\varepsilon^2\big\|\Delta \widetilde{v}_0\big\|_{L^2_x}+\varepsilon^2 T\|\Delta^2 \widetilde{v}\|_{L^\infty_t L^2_x}\Big),
\end{align*}
where the constant $C>0$ only depends on $\|\widetilde{v}_0\|_{H^1_x}$, $\big\|(\widetilde{v}_0^\varepsilon,\widetilde{v}_1^\varepsilon)\big\|_{H^1_x\times L^2_x}$.
  \end{prop}
\begin{proof}
First,  we denote
\begin{align*}
\big\|f\big\|_{S_{s_c}(I)}\triangleq & \big\||\nabla|^{s_c}f\big\|_{L^
\frac{2(d+2)}{d}_{tx}(I)};\\
\big\|f\big\|_{Y_{s_c}(I)}\triangleq & \big\|f\big\|_{S_{s_c}(I)}+\big\||\nabla|^{s_c}\langle\nabla\rangle^\frac12f\big\|_{L^\infty_tL^2_x(I)};
\end{align*}
and
\begin{align*}
\|f\|_{Z_{s_c}(I)}&\triangleq \|f\|_{L^\frac43_{tx}(I\times\R^2)}, \mbox{when } d=2; \\
\|f\|_{Z_{s_c}(I)}&\triangleq \big\||\nabla|^\frac12 f\big\|_{L^\frac{10}9_tL^{\frac{30}{17}}_x(I\times\R^3)}, \mbox{when } d=3.
\end{align*}
Then by \eqref{Duh-Rj-th3} and Lemma \ref{lem:strichartz}, we have that for any $\gamma\ge -s_c$ and any $I=[t_0,t_1]\subset \R$,
\begin{align}\label{est:Thm3-Rj-case1}
\big\||\nabla|^\gamma\mathcal R_j\big\|_{Y_{s_c}(I)}
\lesssim &
\big\||\nabla|^\gamma\langle\nabla\rangle^\frac12\mathcal R_j(t_0)\big\|_{\dot H^{s_c}_x}
+
\varepsilon^4\big\||\nabla|^{s_c+\gamma}\mathcal{S}_\varepsilon( \partial_{tt} \widetilde{v})\big\|_{L^1_tL^2_x(I)}\notag\\
&\qquad +\Big\||\nabla|^\gamma\Big(\big|\mathcal{S}_\varepsilon \widetilde{v}^\varepsilon\big|^2\mathcal{S}_\varepsilon \widetilde{v}^\varepsilon-|\mathcal{S}_\varepsilon \widetilde{v}|^2\mathcal{S}_\varepsilon \widetilde{v}\Big) \Big\|_{Z_{s_c}(I)}.
\end{align}

For second term, we have that
\begin{align}\label{est:Thm3-Rj-case1-pttv-1}
\big\||\nabla|^{s_c+\gamma} \mathcal{S}_\varepsilon( \partial_{tt} \widetilde{v})\big\|_{L^1_tL^2_x(I)}
\lesssim &
| I |\big\||\nabla|^{s_c+\gamma} \mathcal{S}_\varepsilon( \partial_{tt} \widetilde{v})\big\|_{L^\infty_t L^2_x(\R)}\notag\\
= &
\varepsilon^\gamma| I | \big\||\nabla|^{s_c+\gamma}  \partial_{tt} \widetilde{v}\big\|_{L^\infty_t L^2_x(\R)}.
\end{align}
Now we consider the estimates on $\partial_{tt}\widetilde{v}$. From the equation \eqref{eq:nls}, we have the rough form that
$$
\partial_{tt}\widetilde{v}= \Delta^2v+O\big(\widetilde{v}^2\Delta \widetilde{v}+\widetilde{v}^2(\nabla \widetilde{v})^2+\widetilde{v}^5\big).
$$
Therefore, by Gagliardo-Nirenberg's inequality, we have that
$$
\big\||\nabla|^{s_c+\gamma}  \partial_{tt} \widetilde{v}\big\|_{L^\infty_t L^2_x(\R)}
\le C(\|\widetilde{v}\|_{L^\infty_t\dot H^{s_c}_x(\R)})\big\|\widetilde{v}\big\|_{L^\infty_t \dot H^{4+s_c+\gamma}_x(\R)}.
$$
By Lemma \ref{lem:spacetime-norm-NLS}, we further get
\begin{align}\label{est:vptt-gamma}
\big\||\nabla|^{s_c+\gamma}  \partial_{tt} \widetilde{v}\big\|_{L^\infty_t L^2_x(\R)}
\le C(\|\widetilde{v}_0\|_{\dot H^{s_c}_x)}\big\|\widetilde{v}\big\|_{L^\infty_t\dot H^{4+s_c+\gamma}_x(\R)}.
\end{align}
Inserting this into \eqref{est:Thm3-Rj-case1-pttv-1}, we obtain that
\begin{align}\label{est:Thm3-Rj-case1-pttv}
\big\||\nabla|^{s_c+\gamma} \mathcal{S}_\varepsilon( \partial_{tt} \widetilde{v})\big\|_{L^1_tL^2_x(I)}
\lesssim
C(\|\widetilde{v}_0\|_{L^2_x})\>\varepsilon^\gamma| I | \big\|\widetilde{v}\big\|_{L^\infty_t \dot H^{4+s_c+\gamma}_x(\R)} .
\end{align}

For the third term, by Kato-Ponce's inequality,  we have that when $d=3$, for any $-\frac12\le \gamma<0$,
\begin{align}\label{est:Thm3-Rj-case1-Sv-better}
&\Big\||\nabla|^\gamma\Big(\big|\mathcal{S}_\varepsilon \widetilde{v}^\varepsilon\big|^2\mathcal{S}_\varepsilon \widetilde{v}^\varepsilon-|\mathcal{S}_\varepsilon \widetilde{v}|^2\mathcal{S}_\varepsilon \widetilde{v}\Big)  \Big\|_{Z_{s_c}(I)}\notag\\
\lesssim &
\big\||\nabla|^\gamma \mathcal{S}_\varepsilon r\big\|_{S_{s_c}(I)}
\Big(\big\|\mathcal{S}_\varepsilon \widetilde{v}^\varepsilon\big\|_{S_{s_c}(I)}^2
+\big\|\mathcal{S}_\varepsilon \widetilde{v}\big\|_{S_{s_c}(I)}^2\Big)
\notag\\
\lesssim &
\big\||\nabla|^\gamma R\big\|_{S_{s_c}(I)}
\Big(\big\|R\big\|_{S_{s_c}(I)}^2
+\big\|\mathcal{S}_\varepsilon \widetilde{v}\big\|_{S_{s_c}(I)}^2\Big).
\end{align}
For  $ \gamma\ge 0$, then by Kato-Ponce's inequality, there is an additional term,  and the corresponding estimate reads
\begin{align}\label{est:Thm3-Rj-case1-Sv}
&\Big\||\nabla|^\gamma\Big(\big|\mathcal{S}_\varepsilon \widetilde{v}^\varepsilon\big|^2\mathcal{S}_\varepsilon \widetilde{v}^\varepsilon-|\mathcal{S}_\varepsilon \widetilde{v}|^2\mathcal{S}_\varepsilon \widetilde{v}\Big)  \Big\|_{Z_{s_c}(I)}\notag\\
\lesssim &
\big\||\nabla|^\gamma \mathcal{S}_\varepsilon r\big\|_{S_{s_c}(I)}
\Big(\big\|\mathcal{S}_\varepsilon \widetilde{v}^\varepsilon\big\|_{S_{s_c}(I)}^2
+\big\|\mathcal{S}_\varepsilon \widetilde{v}\big\|_{S_{s_c}(I)}^2\Big)
\notag\\
&\quad +
\big\|\mathcal{S}_\varepsilon r\big\|_{S_{s_c}(I)}\big\||\nabla|^\gamma  \mathcal{S}_\varepsilon \widetilde{v}\big\|_{S_{s_c}(I)}
\Big(\big\|\mathcal{S}_\varepsilon \widetilde{v}^\varepsilon\big\|_{S_{s_c}(I)}
+\big\|\mathcal{S}_\varepsilon \widetilde{v}\big\|_{S_{s_c}(I)}\Big)\notag\\
\lesssim &
\big\||\nabla|^\gamma R\big\|_{S_{s_c}(I)}
\Big(\big\|\mathcal{S}_\varepsilon \widetilde{v}^\varepsilon\big\|_{S_{s_c}(I)}^2
+\big\|\mathcal{S}_\varepsilon \widetilde{v}\big\|_{S_{s_c}(I)}^2\Big)
\notag\\
&\quad +
C\varepsilon^\gamma\big\|R\big\|_{S_{s_c}(I)}
\Big(\big\|\mathcal{S}_\varepsilon \widetilde{v}^\varepsilon\big\|_{S_{s_c}(I)}
+\big\|\mathcal{S}_\varepsilon \widetilde{v}\big\|_{S_{s_c}(I)}\Big),
\end{align}
where in the last step we have used
$$
\big\||\nabla|^\gamma  \mathcal{S}_\varepsilon \widetilde{v}\big\|_{S_{s_c}(\R)}
= \varepsilon^\gamma
\big\||\nabla|^\gamma  \widetilde{v}\big\|_{S_{s_c}(\R)}
\le C\varepsilon^\gamma,
$$
and the positive constant $C$ only depend on $\|\widetilde{v}_0\|_{H^1_x}$ and $\|\widetilde{v}_0\|_{\dot H^{s_c+\gamma}}$.

$\bullet$ When $d=2$, by \eqref{est:Thm3-Rj-case1}, \eqref{est:Thm3-Rj-case1-pttv} and \eqref{est:Thm3-Rj-case1-Sv},
we have that
\begin{align}\label{est:Thm3-Rj-case1-d2}
\big\||\nabla|^\gamma \mathcal R_j\big\|_{Y_0(I)}
\lesssim &
\big\||\nabla|^\gamma \langle\nabla\rangle^\frac12\mathcal R_j(t_0)\big\|_{L^2_x}
+
\varepsilon^{4+\gamma}|I|  \big\|\widetilde{v}\big\|_{L^\infty_t \dot H^{4+\gamma}_x(\R)} \notag\\
&\qquad +\big\||\nabla|^\gamma R\big\|_{S_{0}(I)}
\Big(\big\|\mathcal{S}_\varepsilon \widetilde{v}^\varepsilon\big\|_{S_{0}(I)}^2
+\big\|\mathcal{S}_\varepsilon \widetilde{v}\big\|_{S_{0}(I)}^2\Big)
\notag\\
&\qquad +
\varepsilon^\gamma\big\|R\big\|_{S_{0}(I)}
\Big(\big\|\mathcal{S}_\varepsilon  \widetilde{v}^\varepsilon\big\|_{S_{0}(I)}
+\big\|\mathcal{S}_\varepsilon  \widetilde{v}\big\|_{S_{0}(I)}\Big).\notag\\
= &
\big\||\nabla|^\gamma \langle\nabla\rangle^\frac12\mathcal R_j(t_0)\big\|_{L^2_x}
+
\varepsilon^{4+\gamma}|I|  \big\|\widetilde{v}\big\|_{L^\infty_t \dot H^{4+\gamma}_x(\R)} \notag\\
&\qquad +\big\||\nabla|^\gamma R\big\|_{S_{0}(I)}
\Big(\big\|\widetilde{v}^\varepsilon\big\|_{S_{0}(\varepsilon^2I)}^2
+\big\|\widetilde{v}\big\|_{S_{0}(\varepsilon^2I)}^2\Big)
\notag\\
&\qquad +
\varepsilon^\gamma\big\|R\big\|_{S_{0}(I)}
\Big(\big\| \widetilde{v}^\varepsilon\big\|_{S_{0}(\varepsilon^2I)}
+\big\| \widetilde{v}\big\|_{S_{0}(\varepsilon^2I)}\Big).
\end{align}
Moreover, arguing similarly as in Section \ref{sec:RCE}, we have that
\begin{align}\label{est:Thm3-ve-L4tx}
\big\|  \widetilde{v}^\varepsilon\big\|_{S_0(\R)}
\le  C\big(\big\|(\widetilde{v}_0^\varepsilon,\widetilde{v}_1^\varepsilon)\big\|_{H^1_x\times L^2_x}\big).
\end{align}
Indeed, consider
$$
\phi=\fe^{it} \mathcal{S}_\varepsilon \widetilde{v}^\varepsilon,
$$
then it solves the following Cauchy problem:
	\EQ{
	\left\{ \aligned
	&\partial_{tt} \phi - \De \phi + \phi+|\phi|^2\phi=0, \\
	& \phi(0,x) = \mathcal{S}_\varepsilon \widetilde{v}^\varepsilon_0(x),\quad  \partial_t\phi(0,x) =i\mathcal{S}_\varepsilon \widetilde{v}^\varepsilon_0(x) +\mathcal{S}_\varepsilon \widetilde{v}^\varepsilon_1(x).
	\endaligned
	\right.
}
Note that
\begin{align*}
E(\phi(0,x),\partial_t\phi(0,x))
\lesssim & \big\|\widetilde{v}^\varepsilon_0\big\|_{L^2_x}^2+\big\|\widetilde{v}^\varepsilon_1\big\|_{L^2_x}^2
+\varepsilon^2 \big( \big\|\nabla \widetilde{v}^\varepsilon_0\big\|_{L^2_x}^2+ \big\|\widetilde{v}^\varepsilon_0\big\|_{L^4}^4\big)\\
\le &  C\big(\big\|(\widetilde{v}_0^\varepsilon,\widetilde{v}_1^\varepsilon)\big\|_{H^1_x\times L^2_x}\big).
\end{align*}
Then \eqref{est:Thm3-ve-L4tx} follows from Lemma \ref{lem:spacetime-norm-KG}.
Moreover, from Lemma  \ref{lem:spacetime-norm-NLS},
\begin{align}\label{est:Thm3-v-L4tx}
\big\|  \widetilde{v}\big\|_{S_0(\R)}\le  C\big(\big\|\widetilde{v}_0\big\|_{L^2_x}\big).
\end{align}
We also need the following estimates.  By \eqref{eq:R-12-KG-intialdatum}, Sobolev's and Bernstein's inequalities,
\begin{align}\label{est:Thm3-Rj0-gamma}
\big\||\nabla|^\gamma \langle\nabla\rangle^\frac12\mathcal R_{j,0}\big\|_{L^2_x}
\lesssim &
\big\||\nabla|^\gamma \langle\nabla\rangle^\frac12 \mathcal{S}_\varepsilon r_0\big\|_{L^2_x}+
\big\||\nabla|^\gamma \mathcal{S}_\varepsilon r_1^\varepsilon\big\|_{L^2_x}\notag\\
\lesssim &
\big\||\nabla|^\gamma (\mathcal{S}_\varepsilon r_0,\mathcal{S}_\varepsilon \widetilde{v}_1^\varepsilon)\|_{L^2_x}+\big\|P_{\ge1}|\nabla|^\gamma \langle\nabla\rangle^\frac12 \mathcal{S}_\varepsilon r_0\big\|_{L^2_x}
+\varepsilon^2\big\||\nabla|^\gamma \Delta \widetilde{v}_0\big\|_{L^2_x}\notag\\
\lesssim&
 \big\||\nabla|^\gamma (\mathcal{S}_\varepsilon r_0,\mathcal{S}_\varepsilon \widetilde{v}_1^\varepsilon)\|_{L^2_x}+ \| \mathcal{S}_\varepsilon r_0\|_{\dot H^{\beta+\gamma}}
 +\varepsilon^2\big\||\nabla|^\gamma \Delta \widetilde{v}_0\big\|_{L^2_x}\notag\\
\lesssim &
\varepsilon^\gamma\big\||\nabla|^\gamma (r_0,\widetilde{v}_1^\varepsilon)\|_{L^2_x}+\varepsilon^{\beta+\gamma} \|r_0\|_{H^{\beta+\gamma}}+\varepsilon^2\big\||\nabla|^\gamma \Delta \widetilde{v}_0\big\|_{L^2_x}.
\end{align}

First, we consider the case when $\gamma=0$. Then by \eqref{est:Thm3-Rj-case1-d2}, we have that
\begin{align}\label{est:Thm3-Rj-case1-d2-2}
\big\| \mathcal R_j\big\|_{Y_0(I)}
\le &
C_0\big\|\langle\nabla\rangle^\frac12\mathcal R_j(t_0)\big\|_{L^2_x}
+
C_1\varepsilon^{4}|I|   \big\|\Delta^2\widetilde{v}\big\|_{L^\infty_tL^2_x(\R)} \notag\\
&\qquad
+C_2 \big\|R\big\|_{S_{0}(I)}
\Big(\big\| \widetilde{v}^\varepsilon\big\|_{S_{0}(\varepsilon^2I)}
+\big\| \widetilde{v}\big\|_{S_{0}(\varepsilon^2I)}\Big).
\end{align}
From \eqref{est:Thm3-ve-L4tx} and \eqref{est:Thm3-v-L4tx}, there exists a constant $K=K(C_0,C_3)$ and a sequence of time intervals
$$
\bigcup\limits_{k=0}^K J_k=\R^+,
$$
(the negative time direction can be treated similarly) with
$$
J_0=[0, t_1], \quad J_k=(t_k, t_{k+1}] \mbox{ for } k=1,\cdots, K-1,  \quad J_{K}=[t_K, +\infty),
$$
such that
\begin{align}\label{subinterval-length-v-vep-d2}
C_2\Big(\big\| \widetilde{v}^\varepsilon\big\|_{S_{0}(J_k)}
+\big\| \widetilde{v}\big\|_{S_{0}(J_k)}\Big)
\le \frac12.
\end{align}
For any fixed $T>0$, denote
$$
I_k=\varepsilon^{-2}J_k\cap [0,T], \quad\mbox{and } \quad
T_k=\varepsilon^{-2} t_k.
$$
Then by \eqref{est:Thm3-Rj-case1-d2-2} and \eqref{subinterval-length-v-vep-d2}, it drives that for any $1\le k\le K$,
\begin{align*}
\big\| \mathcal R_j\big\|_{Y_0(I_k)}
\le &
2C_0\big\|\langle\nabla\rangle^\frac12\mathcal R_j(T_k)\big\|_{L^2_x}
+
2C_1\varepsilon^{4}T   \big\|\Delta^2\widetilde{v}\big\|_{L^\infty_tL^2_x(\R)}.
\end{align*}
By iteration, we obtain that
%\begin{align*}
%\big\| \mathcal R_j\big\|_{Y_0(I_k)}
%\le &
%\big(2C_0\big)^K\Big(\big\|\langle\nabla\rangle^\frac12\mathcal R_{j,0}\big\|_{L^2_x}
%+
%2C_1\varepsilon^{4}T   \big\|\Delta^2\widetilde{v}\big\|_{L^\infty_tL^2_x(\R)}\Big).
%\end{align*}
for any finite interval $I\subset \R$,
$$
\big\|\mathcal R_j\big\|_{Y_0(I)}
\le C\big(\|\widetilde{v}_0\|_{H^1_x}\big)\Big[
\big\|\langle\nabla\rangle^\frac12\mathcal R_{j,0}\big\|_{L^2_x}
+
\varepsilon^{4}|I|  \big\|\Delta^2\widetilde{v}\big\|_{L^\infty_tL^2_x(\R)} \Big].
$$
Therefore, by \eqref{est:Thm3-Rj0-gamma} we get that
\begin{align*}
\big\|\mathcal R_j\big\|_{Y_0(I)}
\le C\big(\|\widetilde{v}_0\|_{H^1_x}\big)\Big[ \big\|(r_0,\widetilde{v}_1^\varepsilon)\|_{L^2_x}+\varepsilon^
\beta \|r_0\|_{H^\beta}+\varepsilon^2\big\|\Delta \widetilde{v}_0\big\|_{L^2_x}+\varepsilon^4|I|\big\|\Delta^2\widetilde{v}\big\|_{L^\infty_tL^2_x(\R)}\Big].
\end{align*}
Inserting this estimate into \eqref{est:Thm3-Rj-case1-d2} and using the same argument, we also get that
\begin{align*}
\big\||\nabla|^\gamma \mathcal R_j\big\|_{Y_0(I)}
\le  C\Big[
\big\||\nabla|^\gamma\langle\nabla\rangle^\frac12\mathcal R_{j,0}\big\|_{L^2_x}
+\varepsilon^{\gamma}\big\|\mathcal R_j\big\|_{Y_0(I)}
+
\varepsilon^{4+\gamma}|I| \big\||\nabla|^\gamma \partial_{tt} \widetilde{v}\big\|_{L^\infty_t L^2_x(\R)}\Big],
\end{align*}
where  the positive constant $C$ only depend on $\norm{\widetilde{v}_0}_{H^1_x}$ and $\norm{\widetilde{v}_0}_{\dot H^{s_c+\gamma}}$.
By \eqref{est:Thm3-Rj0-gamma}, we obtain that
\begin{align}\label{est:R-12-upp-d2}
\big\||\nabla|^\gamma \mathcal R_j\big\|_{Y_0(I)}
\le C \varepsilon^{\gamma} \Big[ \big\|(r_0,\widetilde{v}_1^\varepsilon)\|_{H^\gamma}+\varepsilon^\beta \|r_0\|_{H^{\beta+\gamma}}+\varepsilon^2\big\||\nabla|^\gamma \Delta \widetilde{v}_0\big\|_{L^2_x}+\varepsilon^4|I|\big\|\Delta^2\widetilde{v}\big\|_{L^\infty_t H^\gamma_x(\R)}\Big].
\end{align}
In particular, by \eqref{eq:R-R12}, this gives that
\begin{align*}
\big\||\nabla|^\gamma \mathcal{S}_\varepsilon r\big\|_{L^\infty_tL^2_x(I)}
\le & C \varepsilon^{\gamma} \Big[ \big\|(r_0,\widetilde{v}_1^\varepsilon)\|_{H^\gamma}+\varepsilon^\beta \|r_0\|_{H^{\beta+\gamma}}\\
&\quad +\varepsilon^2\big\||\nabla|^\gamma \Delta \widetilde{v}_0\big\|_{L^2_x}+\varepsilon^4|I|\big\|\Delta^2\widetilde{v}\big\|_{L^\infty_t H^\gamma_x(\R)}\Big].
\end{align*}
Scaling bark, and choosing $\gamma=0$, we obtain that for any $T<+\infty$,
\begin{align*}
\big\| r\big\|_{L^\infty_tL^2_x([0,T])}
\le C\Big[ \big\|(r_0,\widetilde{v}_1^\varepsilon)\|_{L^2_x}+\varepsilon^\beta \|r_0\|_{H^\beta_x}+\varepsilon^2\big\|\Delta \widetilde{v}_0\big\|_{L^2_x}+\varepsilon^2T\big\|\Delta^2\widetilde{v}\big\|_{L^\infty_t L^2_x(\R)}\Big],
\end{align*}
where  the positive constant $C$ only depend on $\norm{\widetilde{v}_0}_{H^1_x}$ and $\norm{\widetilde{v}_0}_{L^2_x}$.
This proves the proposition when $d=2$.

$\bullet$ When $d=3$,  we need to assume that $I: \varepsilon^{4} |I| \|\widetilde{v}\|_{L^\infty_t\dot H^{\frac92}_x} \le \delta_0 $ ($\delta_0$ will be determined later).
First, by \eqref{est:Thm3-Rj-case1}, \eqref{est:Thm3-Rj-case1-pttv} and \eqref{est:Thm3-Rj-case1-Sv-better},
we have that for any $-\frac12\le \gamma<0$,
\begin{align}\label{est:Thm3-Rj-case1-d3-better}
\big\||\nabla|^\gamma\mathcal R_j\big\|_{Y_\frac12(I)}
\lesssim &
\big\||\nabla|^\gamma\langle\nabla\rangle^\frac12\mathcal R_j(t_0)\big\|_{\dot H^\frac12}
+
\varepsilon^{4+\gamma}|I| \big\||\nabla|^\gamma\partial_{tt} \widetilde{v}\big\|_{L^\infty_t \dot H^\frac12_x(\R)}\notag\\
&\qquad +\big\||\nabla|^\gamma R\big\|_{S_\frac12(I)}
\Big(\big\| R\big\|_{S_\frac12(I)}^2
+\big\|\mathcal{S}_\varepsilon \widetilde{v}\big\|_{S_\frac12(I)}^2\Big)\notag\\
=&
\big\||\nabla|^\gamma\langle\nabla\rangle^\frac12\mathcal R_j(t_0)\big\|_{\dot H^\frac12}
+
\varepsilon^{4+\gamma}|I| \big\||\nabla|^\gamma\partial_{tt} \widetilde{v}\big\|_{L^\infty_t \dot H^\frac12_x(\R)}\notag\\
&\qquad +\big\||\nabla|^\gamma R\big\|_{S_\frac12(I)}
\Big(\big\| R\big\|_{S_\frac12(I)}^2
+\big\|\widetilde{v}\big\|_{S_\frac12(\varepsilon^2I)}^2\Big).
\end{align}
For $\gamma\ge 0$, applying \eqref{est:Thm3-Rj-case1-Sv} instead, we have that
\begin{align}\label{est:Thm3-Rj-case1-d3}
\big\||\nabla|^\gamma\mathcal R_j\big\|_{Y_\frac12(I)}
\lesssim &
\big\||\nabla|^\gamma\langle\nabla\rangle^\frac12\mathcal R_j(t_0)\big\|_{\dot H^\frac12}
+
\varepsilon^{4+\gamma}|I| \big\||\nabla|^\gamma\partial_{tt} \widetilde{v}\big\|_{L^\infty_t \dot H^\frac12_x(\R)}\notag\\
&\qquad +\big\||\nabla|^\gamma R\big\|_{S_\frac12(I)}
\Big(\big\| R\big\|_{S_\frac12(I)}^2
+\big\|\widetilde{v}\big\|_{S_\frac12(\varepsilon^2I)}^2\Big)\notag\\
&\qquad +C \varepsilon^\gamma \big\|R\big\|_{S_\frac12(I)}
\Big(\big\| R\big\|_{S_\frac12(I)}
+\big\|\widetilde{v}\big\|_{S_\frac12(\varepsilon^2I)}\Big),
\end{align}
where the positive constant $C$ only depend on $\norm{\widetilde{v}_0}_{H^1_x}$ and $\norm{\widetilde{v}_0}_{\dot H^{\frac12+\gamma}}$.
%\begin{align}\label{est:Thm3-Rj-case1-d3-better}
%\big\||\nabla|^\gamma\mathcal R_j\big\|_{Y_\frac12(I)}
%\lesssim &
%\big\||\nabla|^\gamma\langle\nabla\rangle^\frac12\mathcal R_j(t_0)\big\|_{\dot H^\frac12}
%+
%\varepsilon^{4+\gamma}|I| \big\||\nabla|^\gamma\partial_{tt} v\big\|_{L^\infty_t \dot H^\frac12_x(\R)}\notag\\
%&\qquad +\big\||\nabla|^\gamma R\big\|_{S_\frac12(I)}
%\Big(\big\| R\big\|_{S_\frac12(I)}^2
%+\big\|v\big\|_{S_\frac12(\varepsilon^2I)}^2\Big).
%\end{align}
Similar as the two dimensional case, we have that
\begin{align}
\big\|  \widetilde{v}\big\|_{S_\frac12(\R)}\le &  C\big(\big\|\widetilde{v}_0\big\|_{H^1_x}\big);\label{est:Thm3-v-Lqtx-3d}\\
\big\||\nabla|^\gamma\langle\nabla\rangle^\frac12\mathcal R_{j,0}\big\|_{\dot H^\frac12}
\le &C\Big(\varepsilon^\gamma\big\|(r_0,\widetilde{v}_1^\varepsilon)\|_{\dot H^{\frac12+\gamma}}+\varepsilon^{\beta+\gamma} \|r_0\|_{H^{\frac12+\beta+\gamma}}+\varepsilon^2\big\|\Delta \widetilde{v}_0\big\|_{\dot H^{\frac12+\gamma}}\Big) \label{est:Thm3-v-Rj0-3d}
\end{align}
%and
%\begin{align}\label{est:Thm3-v-vptt-3d}
%\big\||\nabla|^\gamma\partial_{tt} v\big\|_{L^\infty_t \dot H^\frac12_x(\R)}
%\le C\Big(\|v_0\|_{H^1_x}\Big)\big\|\Delta^2\widetilde{v}\big\|_{L^\infty_t \dot H^{\frac12+\gamma}_x(\R)},
%\end{align}
(the difference from the two dimensional case is that,  we have no boundedness of $S_\frac12(\R)$-norm on $\widetilde{v}^\varepsilon$ the in three dimensional case).
In particular, from \eqref{est:Thm3-Rj-case1-d3}  we further get that
for any $\gamma\ge 0$, 
\begin{align}\label{est:Thm3-Rj-case1-d3-2}
\sum\limits_{j=1,2}\big\||\nabla|^\gamma\mathcal R_j\big\|_{Y_\frac12(I)}
\le  &
C_1\sum\limits_{j=1,2}\big\||\nabla|^\gamma\langle\nabla\rangle^\frac12\mathcal R_j(t_0)\big\|_{\dot H^\frac12}
+
C_2\varepsilon^{4+\gamma}|I| \big\|\Delta^2\widetilde{v}\big\|_{L^\infty_t \dot H^{\frac12+\gamma}_x(\R)}\notag\\
&\qquad +C_3\big\||\nabla|^\gamma R\big\|_{S_\frac12(I)}
\Big(\big\| R\big\|_{S_\frac12(I)}^2
+\big\|\widetilde{v}\big\|_{S_\frac12(\varepsilon^2I)}^2\Big)\notag\\
&\qquad +C_4 \varepsilon^\gamma \big\|R\big\|_{S_\frac12(I)}
\Big(\big\| R\big\|_{S_\frac12(I)}
+\big\|\widetilde{v}\big\|_{S_\frac12(\varepsilon^2I)}\Big).
\end{align}

First, we treat the case when $\gamma=0$.
%Similarly as the proof in Section \ref{subsec:Imp-Stri}, w
Repeating the similar process as the case of $d=2$, we split $\R$ by several subinterval (also denoted by $J_k$) as
$$
\bigcup\limits_{k=0}^K J_k=\R^+,
$$
with
$$
J_0=[0, t_1], \quad J_k=(t_k, t_{k+1}] \mbox{ for } k=1,\cdots, K-1,  \quad J_{K}=[t_K, +\infty),
$$
such that
\begin{align}\label{subinterval-length-2}
C_3\big\| \widetilde{v}\big\|_{S_\frac12(J_k)}^2+C_4\big\|\widetilde{v}\big\|_{S_\frac12(J_k)}
\le \frac14.
\end{align}
Denote
$$
I_k=\varepsilon^{-2}J_k\cap [0,T], \quad\mbox{and } \quad
T_k=\varepsilon^{-2} t_k.
$$
Then by \eqref{est:Thm3-Rj-case1-d3-2} and \eqref{subinterval-length-2}, we obtain that for any $k\le K$,
\begin{align} \label{est:Thm3-Rj-case1-d3-3}
\sum\limits_{j=1,2}\big\|\mathcal R_j\big\|_{Y_\frac12(I_k)}
\le  &
\frac43C_1\sum\limits_{j=1,2}\big\||\nabla|^\frac12\langle\nabla\rangle^\frac12\mathcal R_j(T_k)\big\|_{L^2_x}
+
\frac43C_2\varepsilon^4|I_k| \big\||\nabla|^\frac12\Delta^2\widetilde{v}\big\|_{L^\infty_tL^2_x(\R)}\notag\\
&
+\frac43\Big(C_3\big\|R\big\|_{S_\frac12(I_k)}^3+C_4\big\|R\big\|_{S_\frac12(I_k)}^2\Big)\notag\\
\le  &
\frac43C_1\sum\limits_{j=1,2}\big\||\nabla|^\frac12\langle\nabla\rangle^\frac12\mathcal R_j(T_k)\big\|_{L^2_x}
+
\frac43C_2\varepsilon^4T \big\||\nabla|^\frac12\Delta^2\widetilde{v}\big\|_{L^\infty_tL^2_x(\R)}\notag\\
&
+\frac43\Big(C_3\big\|R\big\|_{S_\frac12(I_k)}^3+C_4\big\|R\big\|_{S_\frac12(I_k)}^2\Big).
\end{align}

Denote
$$
c_0\triangleq \big\|\langle\nabla\rangle^\frac12\mathcal R_{j,0}\big\|_{\dot H^\frac12};\quad
C_0\triangleq C_2 \big\||\nabla|^\frac12\Delta^2\widetilde{v}\big\|_{L^\infty_tL^2_x(\R)}.
$$
Then by \eqref{est:Thm3-v-Rj0-3d},
\begin{align*}
c_0\le C\Big(\big\|(r_0,\widetilde{v}_1^\varepsilon)\|_{L^2_x}+\varepsilon^\beta \|r_0\|_{H^{\frac12+\beta}}\Big).
\end{align*}
Moreover, denote
$$
a_0\triangleq \big(2C_1\big)^K\Big[c_0+2C_0\varepsilon^4T \Big].
$$
Choosing $c_0, \varepsilon_0$ and $\delta_0$ suitably small  such that for any $\varepsilon\in [0, \varepsilon_0], I: \varepsilon^4 |I| \|\widetilde{v}\|_{L^\infty_t\dot H^{\frac92}_x}  \le \delta_0$, we have that
\begin{align}\label{small-a0}
\frac43\big(C_3a_0^2+C_4a_0\big)\le \frac14.
\end{align}

Suppose that for some $0\le k\le K-1$,
\begin{align}\label{iteration-k-local}
\sum\limits_{j=1,2}\big\||\nabla|^\frac12\langle\nabla\rangle^\frac12\mathcal R_j(T_k)\big\|_{L^2_x}
\le
c_0\big(2C_1\big)^k+2C_0\sum\limits_{j=0}^k \big(2C_1\big)^{j-1}.
\end{align}
then
$$
2C_1\sum\limits_{j=1,2}\big\||\nabla|^\frac12\langle\nabla\rangle^\frac12\mathcal R_j(T_k)\big\|_{L^2_x}
+
2 C_0\varepsilon^4T\le a_0.
$$
Therefore, by \eqref{est:Thm3-Rj-case1-d3-3} for $I_k$ and \eqref{small-a0}, and applying the bootstrap argument on $S_\frac12(I_k)$ norm first and obtain that
$$
\big\|R\big\|_{S_\frac12(I_k)}
\le
c_0\big(2C_1\big)^k+2C_0\sum\limits_{j=0}^k \big(2C_1\big)^{j-1}
\le a_0.
$$
Then inserting it back in \eqref{est:Thm3-Rj-case1-d3-3} to obtain the estimate on  $Y_\frac12(I_k)$,  we have that
(for which we choose $\beta=\frac12$)
\begin{align*}
\sum\limits_{j=1,2}\big\|\mathcal R_j\big\|_{Y_\frac12(I_{k+1})}
\le
c_0\big(2C_1\big)^{k+1}+2C_0\sum\limits_{j=0}^{k+1} \big(2C_1\big)^{j-1}.
\end{align*}
This proves \eqref{iteration-k-local} for $k+1$ and thus it holds for any $1\le l\le K$.
Therefore, we obtain that
\begin{align*}
\sum\limits_{j=1,2}\big\|\mathcal R_j\big\|_{Y_\frac12(I)}
\le & C\big(\|\widetilde{v}_0\|_{H^1_x}\big)\Big[ \big\|(r_0,\widetilde{v}_1^\varepsilon)\|_{\dot H^\frac12}+\varepsilon^\frac12 \|r_0\|_{H^1_x}\\
&\quad+\varepsilon^2\big\|\Delta \widetilde{v}_0\big\|_{\dot H^{\frac12}}+\varepsilon^4 T  \big\||\nabla|^\frac12\Delta^2\widetilde{v}\big\|_{L^\infty_tL^2_x(\R)}\Big].
\end{align*}
From assumptions (1)--(3), choosing $\varepsilon_0$ and $\delta_0$ suitably small, then the last estimate above gives that
\begin{align*}
C_3\big\|R\big\|_{S_\frac12(I)}^2\le \frac12.
\end{align*}
Inserting this estimate into \eqref{est:Thm3-Rj-case1-d3} and using the same argument, we also get that for any $\gamma\ge0$,
\begin{align*}
\sum\limits_{j=1,2}\big\||\nabla|^\gamma \mathcal R_j\big\|_{Y_\frac12(I)}
\le  C\Big[
\big\||\nabla|^\gamma\langle\nabla\rangle^\frac12\mathcal R_{j,0}\big\|_{\dot H^\frac12}
+\varepsilon^{\gamma}\big\|\mathcal R_j\big\|_{Y_\frac12(I)}
+
\varepsilon^{4+\gamma}|I| \big\||\nabla|^\gamma \Delta^2 \widetilde{v}\big\|_{L^\infty_t \dot H^\frac12_x(\R)}\Big],
\end{align*}
where  the positive constant $C$ only depend on $\norm{\widetilde{v}_0}_{H^1_x}$ and $\norm{\widetilde{v}_0}_{\dot H^{\frac12+\gamma}}$.
Inserting it into \eqref{est:Thm3-Rj-case1-d3-better} instead, we get that for any $-\frac12\le \gamma<0$,
\begin{align*}
\sum\limits_{j=1,2}\big\||\nabla|^\gamma \mathcal R_j\big\|_{Y_\frac12(I)}
\le  C\Big[
\big\||\nabla|^\gamma\langle\nabla\rangle^\frac12\mathcal R_{j,0}\big\|_{\dot H^\frac12}
+
\varepsilon^{4+\gamma}|I| \big\||\nabla|^\gamma \Delta^2 \widetilde{v}\big\|_{L^\infty_t \dot H^\frac12_x(\R)}\Big],
\end{align*}
where  the positive constant $C$ only depend on $\norm{\widetilde{v}_0}_{H^1_x}$.

Applying these two estimates and treating similarly as in two dimensional case, we can obtain a uniform bound of
$\big\||\nabla|^\gamma \mathcal R_j\big\|_{L^\infty_t \dot H^\frac12_x(I)}$ as
\begin{align*}
\sum\limits_{j=1,2}\big\||\nabla|^\gamma \mathcal R_j\big\|_{L^\infty_t \dot H^\frac12_x(I)}
\le & C\varepsilon^{\gamma} \Big[ \big\|(r_0,\widetilde{v}_1^\varepsilon)\|_{H^{\gamma+\frac12}}+\varepsilon^\beta \|r_0\|_{H^{\beta+\gamma+\frac12}}\\
&\quad +\varepsilon^2\big\||\nabla|^\gamma \Delta \widetilde{v}_0\big\|_{\dot H^{\frac12}}+\varepsilon^4|I|\big\|\Delta^2\widetilde{v}\big\|_{L^\infty_t \dot H^{\gamma+\frac12}_x(\R)}\Big].
\end{align*}
Choosing $\gamma=-\frac12$, it gives that
\begin{align}\label{est:R-12-upp-d3}
\sum\limits_{j=1,2}\big\|\mathcal R_j\big\|_{L^\infty_t L^2_x(I)}
\le & C\varepsilon^{s-\frac12} \Big[ \big\|(r_0,\widetilde{v}_1^\varepsilon)\|_{L^2_x}+\varepsilon^\beta \|r_0\|_{H^\beta_x}
\notag\\
&\quad +\varepsilon^2\big\|\Delta \widetilde{v}_0\big\|_{L^2_x}+\varepsilon^4|I|\big\|\Delta^2\widetilde{v}\big\|_{L^\infty_t L^2_x(\R)}\Big],
\end{align}
and thus 
\begin{align*}
\big\|R\big\|_{L^\infty_t L^2_x(I)}
\le \varepsilon^{s-\frac12} C\Big[ \big\|(r_0,\widetilde{v}_1^\varepsilon)\|_{L^2_x}+\varepsilon^\beta \|r_0\|_{H^\beta_x}+\varepsilon^2\big\|\Delta \widetilde{v}_0\big\|_{L^2_x}+\varepsilon^4|I|\big\|\Delta^2\widetilde{v}\big\|_{L^\infty_t L^2_x(\R)}\Big].
\end{align*}
Scaling bark,  it implies that for any
$$
T: \varepsilon^2 T\|\widetilde{v}\|_{L^\infty_t\dot H^{4+\frac12}_x}\le \delta_0,
$$
it holds that
\begin{align*}
\big\| r\big\|_{L^\infty_tL^2_x([0,T])}
\le  C\Big[ \big\|(r_0,\widetilde{v}_1^\varepsilon)\|_{L^2_x}+\varepsilon^\beta \|r_0\|_{H^\beta_x}+\varepsilon^2\big\|\Delta \widetilde{v}_0\big\|_{L^2_x}+\varepsilon^2T\big\|\Delta^2\widetilde{v}\big\|_{L^\infty_t L^2_x(\R)}\Big].
\end{align*}
This proves the proposition when $d=3$.
\end{proof}

  \subsubsection{Non-regular cases}\label{sec:Thm3-case1-2}

In this subsubsection, we continue to consider the case when $\alpha\le 4$.  The main result is
  \begin{prop}\label{prop-Th3-case2}
  Let $d=2,3$, $T>0$,  $0\le \alpha\le 4$ and $\beta\in [1,2]$ with $\beta\le \alpha$, then there exists $\delta_0>0$ such that the following property hold.  Assume that
  \begin{itemize}
  \item[(1)] $\widetilde{v}_0\in H^1\cap H^\alpha(\R^d),\widetilde{v}^\varepsilon_0\in H^1\cap  H^\beta_x(\R^d), \widetilde{v}_1^\varepsilon\in H^{s_c}(\R^d)$;
  \item[(2)] $(\widetilde{v}^\varepsilon_0-\widetilde{v}_0,\widetilde{v}_1^\varepsilon)\to (0,0) \mbox{ in } H^{s_c}(\R^d).$
  \end{itemize}
  When $d=3$,  additionally assume that
    \begin{itemize}
    \item[(3)] $\varepsilon^2 T \le \delta_0$.
  \end{itemize}
Let $\widetilde{v}^\varepsilon, \widetilde{v}$ be the corresponding solutions to  \eqref{eq:nls-wave-N} and \eqref{eq:nls-1} respectively, then
for any $\beta\ge 1$,
  \begin{align*}
\sup\limits_{t\in [0,T]}\big\|\widetilde{v}^\varepsilon(t)-\widetilde{v}(t)\big\|_{L^2_x}
\lesssim & \|(\widetilde{v}^\varepsilon_0-\widetilde{v}_0,\widetilde{v}_1^\varepsilon)\|_{L^2_x\times L^2_x}
+\varepsilon^\beta\big\|\widetilde{v}^\varepsilon_0-\widetilde{v}_0\big\|_{H^\beta_x}+\varepsilon^\alpha+\big(\varepsilon^2 T\big)^{\frac14\alpha}.
\end{align*}
  \end{prop}
\begin{proof}
Let $N=N(\varepsilon)>0$ be a parameter which satisfies that
$$
N(\varepsilon)\to \infty, \quad \mbox{ when }\quad \varepsilon\to 0,
$$
and will be determined later.
Denote $\widetilde{v}_N $ to be the solution to  the following equation
\EQ{
	\left\{ \aligned
	&2i\partial_{t} \widetilde{v} - \De \widetilde{v}= - |\widetilde{v}|^2 \widetilde{v}, \\
	& \widetilde{v}(0,x) = P_{\le N}\widetilde{v}_0(x),
	\endaligned
	\right.
}
and denote $\widetilde{v}^N =\widetilde{v}-\widetilde{v}_N $. Then $\widetilde{v}^N $ is the solution to the  following equation
\EQ{
	\left\{ \aligned
	&2i\partial_{t} \widetilde{v}^N  - \De \widetilde{v}^N  = - 3\big(|\widetilde{v}|^2 \widetilde{v}-|\widetilde{v}_N |^2 \widetilde{v}_N \big), \\
	& \widetilde{v}^N (0,x) = P_{> N}\widetilde{v}_0(x).
	\endaligned
	\right.
}
Then by Lemma \ref{lem:spacetime-norm-NLS}, there exists some constant $C>0$ which only depends on $\|\widetilde{v}_0\|_{H^1_x}$, such that 
\begin{align}\label{est:v-spacetime}
\big\|\widetilde{v}\big\|_{S_{s_c}(\R)}
+\big\|\widetilde{v}_N \big\|_{S_{s_c}(\R)}
\le C.
\end{align}
Moreover, by Lemma \ref{lem:strichartz}, we have that for any $I=[t_0,t_1]\subset \R$,
\begin{align*}
\big\|\widetilde{v}^N \big\|_{L^\infty_tL^2_x(I)}
\lesssim &
\big\|\widetilde{v}^N (t_0)\big\|_{L^2_x}
+\big\|\widetilde{v}^N \big\|_{L^\infty_tL^2_x(I)}
\Big(\big\|\widetilde{v}\big\|_{S_{s_c}(I)}^2
+\big\|\widetilde{v}_N \big\|_{S_{s_c}(I)}^2\Big).
\end{align*}
By \eqref{est:v-spacetime} and then arguing similarly as in Section \ref{sec:Thm3-case1-1}, we obtain that
\begin{align}\label{est:Thm3-w-case2-d2}
\big\|\widetilde{v}^N \big\|_{L^\infty_tL^2_x(\R)}
\le &
C\big\|P_{> N}\widetilde{v}_0\big\|_{L^2_x}
\le CN^{-\alpha}\big\|\widetilde{v}_0\big\|_{H^\alpha_x},
\end{align}
where the constant $C>0$ only depends on $\|\widetilde{v}_0\|_{H^1_x}$.

Now we write
$$
\widetilde{v}^\varepsilon(t)-\widetilde{v}(t)=\widetilde{v}^\varepsilon(t)-\widetilde{v}_N (t)-\widetilde{v}^N .
$$
Note that under the hypothesis in the proposition, $P_{\le N}\widetilde{v}_0, \widetilde{v}_0^\varepsilon \widetilde{v}_1^\varepsilon$ verify the assumptions in Proposition \ref{prop:Thm3-case1-1}. Then we have that
\begin{align*}
\sup\limits_{t\in [0,T]}\big\|\widetilde{v}^\varepsilon(t)-\widetilde{v}_N (t)\big\|_{L^2_x}
\le  &C\Big( \|(\widetilde{v}^\varepsilon_0-P_{\le N}\widetilde{v}_0,\widetilde{v}_1^\varepsilon)\|_{L^2_x\times L^2_x}+\varepsilon^\beta\big\|\widetilde{v}^\varepsilon_0-P_{\le N} \widetilde{v}_0\big\|_{H^\beta_x}\\
&\qquad  +\varepsilon^2\big\|\Delta P_{\le N}  \widetilde{v}_0\big\|_{L^2_x}+\varepsilon^2 T\|\Delta^2 \widetilde{v}_N \|_{L^\infty_t L^2_x}\Big),
\end{align*}
for any $T$ verifying
\begin{align}\label{Thm3-T-condition}
T<+\infty \quad \mbox{when }d=2;\quad \varepsilon^2 T N^\frac72\le C(\|\widetilde{v}_0\|_{H^1_x}) \delta_0 \quad \mbox{when }d=3.
\end{align}
By Lemma \ref{lem:spacetime-norm-NLS}, we have that
$$
\|\Delta^2 \widetilde{v}_N \|_{L^\infty_t L^2_x}\le C(\|\widetilde{v}_0\|_{H^1_x}) \big\|\Delta^2  P_{\le N}\widetilde{v}_0\big\|_{L^2_x}
\le N^{4-\alpha} C(\|\widetilde{v}_0\|_{H^1_x}, \|\widetilde{v}_0\|_{H^\alpha_x}).
$$
Moreover,
\begin{align*}
\|\widetilde{v}^\varepsilon_0-P_{\le N}\widetilde{v}_0\|_{L^2_x}
\lesssim & \|\widetilde{v}^\varepsilon_0-\widetilde{v}_0\|_{L^2_x}+\|P_{>N}\widetilde{v}_0\|_{L^2_x}\\
\lesssim &\|\widetilde{v}^\varepsilon_0-\widetilde{v}_0\|_{L^2_x}+N^{-\alpha}\|\widetilde{v}_0\|_{H^\alpha_x};
\end{align*}
\begin{align*}
\varepsilon^\beta\big\|\widetilde{v}^\varepsilon_0-P_{\le N} \widetilde{v}_0\big\|_{H^\beta_x}
\lesssim &
\varepsilon^\beta\big\|\widetilde{v}^\varepsilon_0- \widetilde{v}_0\big\|_{H^\beta_x}+\varepsilon^\beta\|P_{>N}\widetilde{v}_0\|_{H^\beta_x}\\
\lesssim &
\varepsilon^\beta\big\|\widetilde{v}^\varepsilon_0- \widetilde{v}_0\big\|_{H^\beta_x}+\varepsilon^\beta N^{\beta-\alpha}\|P_{>N}\widetilde{v}_0\|_{H^\alpha_x}\\
\lesssim &
\varepsilon^\beta\big\|\widetilde{v}^\varepsilon_0- \widetilde{v}_0\big\|_{H^\beta_x}+\big(\varepsilon^\alpha+ N^{-\alpha}\big)\|P_{>N}\widetilde{v}_0\|_{H^\alpha_x};
\end{align*}
and when $\alpha\le 2$, 
\begin{align*}
\varepsilon^2\big\|\Delta P_{\le N}  \widetilde{v}_0\big\|_{L^2_x}
\lesssim &
\varepsilon^2N^{2-\alpha}\big\| P_{\le N}\widetilde{v}_0\big\|_{H^\alpha_x}\\
\lesssim &
\big(\varepsilon^\alpha+ N^{-\alpha}\big)\|\widetilde{v}_0\|_{H^\alpha_x};
\end{align*}
when $\alpha> 2$, 
\begin{align*}
\varepsilon^2\big\|\Delta P_{\le N}  \widetilde{v}_0\big\|_{L^2_x}
\le  &
\varepsilon^2\big\|\widetilde{v}_0\big\|_{H^\alpha_x}.
\end{align*}
Hence, it implies that
\begin{align*}
\sup\limits_{t\in [0,T]}\big\|\widetilde{v}^\varepsilon(t)-\widetilde{v}_N (t)\big\|_{L^2_x}
\le  &C\Big( \|(\widetilde{v}^\varepsilon_0-\widetilde{v}_0,\widetilde{v}_1^\varepsilon)\|_{L^2_x\times L^2_x}+\varepsilon^\beta\big\|\widetilde{v}^\varepsilon_0-\widetilde{v}_0\big\|_{H^\beta_x}\\
&\qquad +\varepsilon^{\min\{\alpha,2\}} +\varepsilon^2 T N^{4-\alpha}+ N^{-\alpha}\Big),
\end{align*}
where $C=C(\|\widetilde{v}_0\|_{H^1_x}, \|\widetilde{v}_0\|_{H^\alpha_x})>0$.
This together with \eqref{est:Thm3-w-case2-d2} infers that
\begin{align*}
\sup\limits_{t\in [0,T]}\big\|\widetilde{v}^\varepsilon(t)-v(t)\big\|_{L^2_x}
\le  &C\Big( \|(\widetilde{v}^\varepsilon_0-\widetilde{v}_0,\widetilde{v}_1^\varepsilon)\|_{L^2_x\times L^2_x}+\varepsilon^\beta\big\|\widetilde{v}^\varepsilon_0-\widetilde{v}_0\big\|_{H^\beta_x}\\
&\qquad  \varepsilon^{\min\{\alpha,2\}} +\varepsilon^2 T N^{4-\alpha}+N^{-\alpha}\Big).
\end{align*}
Choosing
$$
N=(\varepsilon^2 T)^{-\frac14},
$$
it gives the desired estimate for $T$ satisfying \eqref{Thm3-T-condition}. In particular, adjusting the value of $\delta_0$ appropriately, then we have that assumption (3) implies \eqref{Thm3-T-condition} when $d=3$.
\end{proof}
\vskip .5cm

\begin{proof}[Proof of Theorem \ref{thm:main3}]
When $\varepsilon^2 T\ge \delta_0$, multiply the equation \eqref{eq:nls-wave-N} with $\overline{\widetilde{v}^\varepsilon}$, integrate in $x$, and take the imaginary part, then we obtain
\begin{align}\label{mass-NLS-wave}
\int\big(\varepsilon^2\mbox{Im} (\partial_t \widetilde{v}^\varepsilon \overline{\widetilde{v}^\varepsilon}) + |\widetilde{v}^\varepsilon|^2\big)dx
= \int\big(\mbox{Im} (\widetilde{v}_1^\varepsilon \overline{\widetilde{v}_0^\varepsilon}) + |\widetilde{v}_0^\varepsilon|^2\big)dx.
\end{align}
This implies that
\begin{align*}
\big\|\widetilde{v}^\varepsilon\big\|_{L^2_x}^2
\lesssim \varepsilon^4 \big\|\partial_t \widetilde{v}^\varepsilon\big\|_{L^2_x}^2+\big\|(\widetilde{v}_0^\varepsilon,\widetilde{v}_1^\varepsilon)\big\|_{L^2_x\times L^2_x}^2.
\end{align*}
Note that from \eqref{energy-NLS-wave}, we have that
\begin{align*}
\big\|\partial_t  \widetilde{v}^\varepsilon\big\|_{L^2_x}^2
\lesssim \varepsilon^{-4}\big\|(\widetilde{v}_0^\varepsilon,\widetilde{v}_1^\varepsilon)\big\|_{H^1_x\times L^2_x}^2.
\end{align*}
These two estimates give that
\begin{align}\label{est:L2-uniform-vep}
\big\|\widetilde{v}^\varepsilon\big\|_{L^\infty_tL^2_x(\R)}
\lesssim \big\|(\widetilde{v}_0^\varepsilon,\widetilde{v}_1^\varepsilon)\big\|_{H^1_x\times L^2_x}.
\end{align}
Moreover, by the mass conservation law of the solution to  \eqref{eq:nls}, we have that
\begin{align*}
\big\|\widetilde{v}\big\|_{L^\infty_tL^2_x(\R)}
= \big\|\widetilde{v}_0\big\|_{L^2_x}.
\end{align*}
Therefore, we have that
\begin{align*}
\big\|\widetilde{v}^\varepsilon(t)-\widetilde{v}(t)\big\|_{L^\infty_tL^2_x([0,T])}
\lesssim & \big\|\widetilde{v}^\varepsilon\big\|_{L^\infty_tL^2_x([0,T])}+\big\|\widetilde{v}\big\|_{L^\infty_tL^2_x([0,T])}\\
\lesssim &
\big\|(\widetilde{v}_0^\varepsilon,\widetilde{v}_1^\varepsilon)\big\|_{H^1_x\times L^2_x}+ \big\|\widetilde{v}_0\big\|_{L^2_x}.
\end{align*}
It yields \eqref{est:Thm3-main} when $\varepsilon^2 T\ge \delta_0$.
Hence, we only need to consider when $\varepsilon^2 T\le \delta_0$.
This completes the proof of Theorem  \ref{thm:main3}.
\end{proof}

\subsection{Optimality}\label{sec:Thm3-optimal-1}

In this subsection, we give the proof of Theorem  \ref{thm:main3-optimal}.

First, we set the initial data 
$$
\widetilde{v}_0=\delta_0f,\quad 
\widetilde{v}_0^\varepsilon-\widetilde{v}_0=\delta_0r_0,\quad 
\widetilde{v}_1^\varepsilon=0,
$$
where $\delta_0>0$ is a small  constant determined later, $f\in \mathcal S$ is real-valued and independent of $\delta_0$.

Consider
$$
\phi=\fe^{it} \mathcal{S}_\varepsilon \widetilde{v}^\varepsilon,
$$
then it solves the following Cauchy problem:
	\EQ{
	\left\{ \aligned
	&\partial_{tt} \phi - \De \phi + \phi=- |\phi|^2\phi, \\
	& \phi(0,x) =\delta_0 \mathcal{S}_\varepsilon (f+r_0)(x),\quad  \partial_t\phi(0,x) =i\delta_0 \mathcal{S}_\varepsilon (f+r_0)(x).
	\endaligned
	\right.
}
By  the scaling invariant of $\mathcal{S}_\varepsilon$ in $\dot H^{s_c}_x$, we have that 
$$
\big\|\langle\nabla\rangle^\frac12 (\phi(0),\langle\nabla\rangle^{-1}\partial_t\phi(0))\big\|_{\dot H^{s_c}_x\times \dot H^{s_c}_x}\lesssim \delta_0\big(\|f\|_{H^1_x}+\|r_0\|_{H^1_x}\big).
$$
Choosing $\delta_0$ suitably small, 
then by Lemma \ref{lem:spacetime-norm-KG-sd}, we obtain that for some positive constant $C=C(\|f\|_{H^1_x},\|r_0\|_{H^1_x})>0$, 
\begin{align}\label{est:Sv-ep-Xs}
\big\|\widetilde{v}^\varepsilon\big\|_{X_{s_c}(\R)}
=\big\|\mathcal{S}_\varepsilon \widetilde{v}^\varepsilon\big\|_{X_{s_c}(\R)}
=\|\phi\|_{X_{s_c}(\R)}
\le C \delta_0.
\end{align}
Moreover, by Lemma \ref{lem:spacetime-norm-NLS}, we have that  for some $C=C(\|f\|_{H^1_x})>0$, 
\begin{align}\label{est:v-Xs}
\big\|\widetilde{v}\big\|_{X_{s_c}(\R)}
=  \|\widetilde{v}\|_{X_{s_c}(\R)}
\le C \delta_0.
\end{align}
Based on this estimate and the Strichartz estimates, we write 
\begin{align}\label{v-linear-app-1}
\widetilde{v}=\delta_0\fe^{-\frac12 it\Delta}f+\mathcal N(\widetilde{v}),
\end{align}
where the term $\mathcal N$ satisfies that  for some $C=C(\|f\|_{H^{s_c+\gamma}_x})>0$, 
\begin{align}\label{re-linear-app-1}
\big\||\nabla|^\gamma\mathcal N(\widetilde{v})\big\|_{X_{s_c}(\R)}\le C\delta_0^3.
\end{align}
Moreover, we also have that 
\begin{align}\label{v-linear-app-2}
\partial_{tt}\widetilde{v}=-\frac14\delta_0\fe^{-\frac12 it\Delta}\Delta^2 f+\widetilde{\mathcal N}(\widetilde{v}),
\end{align}
where  there exists some $C=C(\|f\|_{H^1_x})>0$, 
\begin{align}\label{re-linear-app-1}
\big\|\widetilde{\mathcal N}(\widetilde{v})\big\|_{L^\infty_tL^2_x(\R)}\le C\delta_0^3\big\|\Delta^2 f\big\|_{L^2_x}.
\end{align}

Now we follow by the same formulation in Section \ref{sec:Th3-formula} and adopt the same notations. 
In particular, we recall that 
$$
r=\widetilde{v}^\varepsilon-\widetilde{v},\quad 
R=\fe^{it} \mathcal{S}_\varepsilon  r,\quad 
\mathcal R_j=\langle \nabla \rangle^{-1}\big(\partial_t\mp i \langle \nabla\rangle\big)R \mbox{ for } j=1,2. 
$$
Next, we will show that  the upper bound of the {\it Leftward wave} $\mathcal R_1$, and the lower and upper bounds of the {\it Rightward wave} $\mathcal R_2$.

Denote 
$$
I(f)\triangleq 
\{t\in \R^+:  t\varepsilon^4\|\Delta^2f\|_{\dot H^{s_c}_x}\le 1\}.
$$
\begin{lem}
Let $f\in \mathcal S(\R^d)$,
% and there exists $C_0>0$ such that  
%\begin{align}\label{assp:f-lowerbd}
%\min\{\big\|\Delta f\big\|_{L^2_x},\big\|\Delta^2 f\big\|_{L^2_x}\}\ge C_0.
%\end{align}
then 
% satisfying 
%\begin{align}\label{time-restriction}
%t\in \R^+:  t\varepsilon^4\|\Delta^2f\|_{\dot H^{s_c}_x}\le 1,
%\end{align}
 there exist constants $C(\|f\|_{L^2_x})>0$ such that for any $t\in I(f)$, there hold
 \begin{align}\label{R1-conter-ubd}
\big\|\mathcal R_1(t)\big\|_{L^2_x}
\le& 
C\delta_0\varepsilon^{-s_c}\Big[\big\|r_0\big\|_{L^2_x}+\delta_0^2 \varepsilon \|r_0\|_{H^1_x}+\varepsilon^2\big\|\Delta f\big\|_{L^2_x} 
+(1+\delta_0^2t)\varepsilon^4\big\|\Delta^2 f\big\|_{L^2_x}\Big],
\end{align}
and 
\begin{align}\label{R2-conter-lbd}
&\Big|\big\|\mathcal R_2(t)\big\|_{L^2_x}-\frac14\delta_0 t\varepsilon^{4-s_c}  \big\|\Delta^2 f\big\|_{L^2_x} -\frac 1{64}\delta_0t^2 \varepsilon^{8-s_c}  \big\|\Delta^4 f\big\|_{L^2_x}\Big|\notag\\
&\le  C\delta_0\varepsilon^{-s_c}\Big[\big(\big\|r_0\big\|_{L^2_x}+\delta_0^2 \varepsilon\|r_0\|_{H^1_x}\big)
 +\varepsilon^2\big\|\Delta f\big\|_{L^2_x}\notag\\
 &\quad +\delta_0^2 t\varepsilon^4  \big\|\Delta^2 f\big\|_{L^2_x} +t\varepsilon^6  \big\|\Delta^3f\big\|_{L^2_x} + C t^2\varepsilon^{10}  \big\|\Delta^5 f\big\|_{L^2_x}
 \Big] .
\end{align}
\end{lem}
\begin{proof}
Then  we have the following Duhamel's formula: For $j=1,2$, 
\begin{align*}
\mathcal R_j(t)=&\fe^{\mp it\langle \nabla\rangle} \mathcal R_{j,0}\notag\\
& -\int_0^t \fe^{\mp i(t-s)\langle \nabla\rangle}\fe^{is} \langle\nabla \rangle^{-1} \Big[\varepsilon^4 \mathcal{S}_\varepsilon( \partial_{tt} \widetilde{v})+\big|\mathcal{S}_\varepsilon \widetilde{v}^\varepsilon\big|^2\mathcal{S}_\varepsilon \widetilde{v}^\varepsilon-|\mathcal{S}_\varepsilon \widetilde{v}|^2\mathcal{S}_\varepsilon \widetilde{v}\Big]\,ds.
\end{align*}
Then by \eqref{v-linear-app-2}, we rewrite it as 
\begin{align}\label{Rj-123}
\mathcal R_j(t)=&\fe^{\mp it\langle \nabla\rangle} \mathcal R_{j,0}+\fe^{\mp it\langle \nabla\rangle} I_j(t)+ \mathcal R_{j}^1(t),
\end{align}
where we denote 
\begin{align*}
I_j(t)\triangleq & \frac14\delta_0 \int_0^t \fe^{\pm  is\langle \nabla\rangle}\fe^{is} \langle\nabla \rangle^{-1} \Big[\varepsilon^4 \mathcal{S}_\varepsilon(\fe^{-\frac12 is\Delta}\Delta^2 f)\Big]\,ds\\
\mathcal R_{j}^1(t) \triangleq & -\int_0^t \fe^{\mp i(t-s)\langle \nabla\rangle}\fe^{is} \langle\nabla \rangle^{-1} \Big[\varepsilon^4 \mathcal{S}_\varepsilon\widetilde{\mathcal N}(\widetilde{v})+\big|\mathcal{S}_\varepsilon \widetilde{v}^\varepsilon\big|^2\mathcal{S}_\varepsilon \widetilde{v}^\varepsilon-|\mathcal{S}_\varepsilon \widetilde{v}|^2\mathcal{S}_\varepsilon \widetilde{v}\Big]\,ds.
\end{align*}

 {\bf Upper bound of $\mathcal R_{j,0}$.}

Under the hypotheses,  we have that 
$$
\mathcal R_{j,0}=\langle \nabla \rangle^{-1}\Big[i\big(1\mp \langle \nabla\rangle\big)\mathcal{S}_\varepsilon r_0+\mathcal{S}_\varepsilon r_1\Big],\quad 
\mbox{ with }\quad r_1=-\frac{\varepsilon^2}{2i}(\Delta \widetilde{v}_0-\mu |\widetilde{v}_0|^2\widetilde{v}_0).
$$
This gives that for some $C(\|f\|_{L^2_x})>0$,
\begin{align}\label{est:conterex-intial}
\left\|\fe^{\mp it\langle \nabla\rangle} \mathcal R_{j,0}\right\|_{L^2_x}
=\big\| \mathcal R_{j,0}\big\|_{L^2_x}
\le C \delta_0\varepsilon^{-s_c}\big(\big\|r_0\big\|_{L^2_x}+\varepsilon^2\big\|\Delta f\big\|_{L^2_x}\big).
\end{align}

For $I_j(t)$, by the scaling invariance of the linear Schr\"odinger flow, we have that 
\begin{align*}
I_j(t)= & \frac14\delta_0\varepsilon^4 \int_0^t \fe^{is(\pm\langle \nabla\rangle+1-\frac12 \Delta)} \langle\nabla \rangle^{-1}\mathcal{S}_\varepsilon(\Delta^2 f)\,ds.
\end{align*}

 {\bf Upper bound of $I_1$.}
Taking the Fourier transformation, we have that 
\begin{align*}
\hat I_1(t,\xi)= & \frac14\delta_0\varepsilon^4 \int_0^t \fe^{is(\langle \xi\rangle+1+\frac12|\xi|^2)}  \langle\xi \rangle^{-1}\widehat{\mathcal{S}_\varepsilon(\Delta^2 f)}(\xi)\,ds\\
= & \frac14\delta_0\varepsilon^4 \frac{\fe^{is(\langle \xi\rangle+1+\frac12|\xi|^2)}-1}{\langle \xi\rangle+1+\frac12|\xi|^2} \langle\xi \rangle^{-1}\widehat{\mathcal{S}_\varepsilon(\Delta^2 f)}(\xi).
\end{align*}
This gives that
\begin{align}
\label{est:I1-upp-bd}
\big\|I(t)\big\|_{L^2_x}
\lesssim & \delta_0\varepsilon^4\big\|\mathcal{S}_\varepsilon(\Delta^2 f)\big\|_{L^2_x}\notag\\
\lesssim & \delta_0\varepsilon^{4-s_c}\big\|\Delta^2 f\big\|_{L^2_x}.
\end{align}

 {\bf Approximation of $I_2$.}
 Denote the operator
$$
\tilde T_t  =\frac{\fe^{it(-\langle \nabla\rangle+1-\frac12 \Delta)}-1-\frac18 it\Delta^2}{t\Delta^3}.
$$
Then by the Mihlin-H\"ormander Multiplier Theorem,  we have the $L^p\mapsto L^p$ boundedness which is uniform in time: for any $1<p<+\infty$,
\begin{align}\label{est:Tt-tilde}
\big\|\tilde T_t h\big\|_{L^p}\le C\|h\|_{L^p},
\end{align}
where the constant $C>0$ is independent of $t$ and $h$.  Now we write
\begin{align}\label{apprx-linearKG}
\fe^{it(-\langle \nabla\rangle+1-\frac12 \Delta)}=1+\frac18 it\Delta^2+t \Delta^3 \tilde T_t.
\end{align}

 Applying \eqref{apprx-linearKG}, we write that 
 \begin{align*}
 I_2(t)=&  \frac14\delta_0 t\varepsilon^4 \mathcal{S}_\varepsilon(\Delta^2 f)
 + \frac i{64}\delta_0t^2 \varepsilon^4 \Delta^2 \mathcal{S}_\varepsilon(\Delta^2 f)\\
& + \frac14\delta_0 \int_0^t \fe^{\pm  is\langle \nabla\rangle}\fe^{is} \big(\langle\nabla \rangle^{-1}-1\big) \Big[\varepsilon^4 \mathcal{S}_\varepsilon(\fe^{-\frac12 is\Delta}\Delta^2 f)\Big]\,ds\\
 &\quad + \frac14\delta_0\varepsilon^4 \int_0^t s\tilde{T_s} \Delta^3 \langle\nabla \rangle^{-1}\mathcal{S}_\varepsilon(\Delta^2 f)\,ds.
 \end{align*}
 Since $f\in \R$, this yields that 
  \begin{align*}
\Big|\big\| I_2(t)\big\|_{L^2_x}-& \frac14\delta_0t\varepsilon^4  \big\|\mathcal{S}_\varepsilon(\Delta^2 f)\big\|_{L^2_x}
-\frac 1{64}\delta_0t^2 \varepsilon^4  \big\|\Delta^2 \mathcal{S}_\varepsilon(\Delta^2 f)\big\|_{L^2_x}\Big|\\
\le &
 \frac14\delta_0t\varepsilon^4  \big\|\big(\langle\nabla \rangle^{-1}-1\big)\mathcal{S}_\varepsilon(\Delta^2 f)\big\|_{L^2_x}\\
 &\quad +\frac14\delta_0\varepsilon^4 \int_0^t s\big\|\tilde{T_s} \Delta^3 \langle\nabla \rangle^{-1}\mathcal{S}_\varepsilon(\Delta^2 f)\big\|_{L^2_x}\,ds.
 \end{align*}
 Using \eqref{est:Tt-tilde}, we further get that there exists some constant $C>0$ such that 
   \begin{align*}
\Big|\big\| I_2(t)\big\|_{L^2_x}& - \frac14\delta_0t\varepsilon^{4-s_c}  \big\|\Delta^2 f\big\|_{L^2_x}-\frac 1{64}\delta_0t^2 \varepsilon^{8-s_c}  \big\|\Delta^4 f\big\|_{L^2_x}\Big|\\
\le &
C\delta_0t\varepsilon^{6-s_c}  \big\|\Delta^3f\big\|_{L^2_x} + C\delta_0t^2\varepsilon^{10-s_c}  \big\|\Delta^5 f\big\|_{L^2_x}.
 \end{align*}

 {\bf Upper bound of $\mathcal R_j^1$.}
By the Strichartz and H\"older inequalities, we have that 
\begin{align*}
\big\|\mathcal R_j^1\big\|_{L^2_x}
\lesssim & t\varepsilon^4\big\|\mathcal{S}_\varepsilon\widetilde{\mathcal N}(\widetilde{v})\big\|_{L^\infty_tL^2_x(\R)}
+\big(\big\|\mathcal R_1\big\|_{L^2_x}+\big\|\mathcal R_2\big\|_{L^2_x}\big)\big(\big\|\mathcal{S}_\varepsilon v^\varepsilon\big\|_{X_{s_c}(\R)}^2+\big\|\mathcal{S}_\varepsilon v\big\|_{X_{s_c}(\R)}^2\big)\\
\lesssim &  t\varepsilon^{4-s_c}\big\|\widetilde{\mathcal N}(\widetilde{v})\big\|_{L^\infty_tL^2_x(\R)}
+\big(\big\|\mathcal R_1\big\|_{L^2_x}+\big\|\mathcal R_2\big\|_{L^2_x}\big)\big(\big\|v^\varepsilon\big\|_{X_{s_c}(\R)}^2+\big\|v\big\|_{X_{s_c}(\R)}^2\big).
\end{align*}
By  the conclusion in the proof of Proposition \ref{prop:Thm3-case1-1} (see \eqref{est:R-12-upp-d2},  \eqref{est:R-12-upp-d3} and \eqref{est:L2-uniform-vep}), we have that for any $t\in I(f)$,
it holds that 
\begin{align}\label{est:R-upper-thm3}
\big\|\mathcal R_j(t)\big\|_{L^2_x}
\lesssim \delta_0 \varepsilon^{-s_c}\|r_0\|_{L^2_x}+\delta_0 \varepsilon^{1-s_c}\|r_0\|_{H^1_x}+\delta_0 t \varepsilon^{4-s_c}\big\|\Delta^2 f\big\|_{L^2_x}.
\end{align}
%This combining with the uniform boundedness of $v^\varepsilon$ (see \eqref{est:L2-uniform-vep}) and $v$ in $L^2(\R^d)$ gives that \eqref{est:R-upper-thm3} holds for any $t\in \R^+$. 
By \eqref{est:Sv-ep-Xs}, \eqref{est:v-Xs} and \eqref{est:R-upper-thm3}, we further get that 
\begin{align}\label{est:Rj1-1}
\big\|\mathcal R_j^1\big\|_{L^2_x}
\lesssim & \delta_0^3 \varepsilon^{-s_c}\|r_0\|_{L^2_x}+\delta_0^3 \varepsilon^{1-s_c}\|r_0\|_{H^1_x}+\delta_0^3 t \varepsilon^{4-s_c}\big\|\Delta^2 f\big\|_{L^2_x}.
\end{align}
%On the other hand,  similar as above, we have that 
%\begin{align*}
%\big\|\mathcal R_j^1\big\|_{L^2_x}
%\lesssim & t\varepsilon^4\big\|\mathcal{S}_\varepsilon\widetilde{\mathcal N}(\widetilde{v})\big\|_{L^\infty_tL^2_x(\R)}
%+\big(\big\|\mathcal{S}_\varepsilon v^\varepsilon\big\|_{L^2_x}+\big\|\mathcal{S}_\varepsilon v\big\|_{L^2_x}\big)\big(\big\|\mathcal{S}_\varepsilon v^\varepsilon\big\|_{X_{s_c}(\R)}^2+\big\|\mathcal{S}_\varepsilon v\big\|_{X_{s_c}(\R)}^2\big)\\
%\lesssim &  t\varepsilon^{4-s_c}\big\|\widetilde{\mathcal N}(\widetilde{v})\big\|_{L^\infty_tL^2_x(\R)}
%+\varepsilon^{-s_c}\big(\big\|v^\varepsilon\big\|_{L^2_x}+\big\|v\big\|_{L^2_x}\big)\big(\big\|v^\varepsilon\big\|_{X_{s_c}(\R)}^2+\big\|v\big\|_{X_{s_c}(\R)}^2\big).
%\end{align*}
%By the uniform boundedness of $v^\varepsilon$ (see \eqref{est:L2-uniform-vep}) and $v$ in $L^2(\R^d)$,  \eqref{est:Sv-ep-Xs} and \eqref{est:v-Xs}, we have that 
%\begin{align*}
%\big\|\mathcal R_j^1\big\|_{L^2_x}
%\le C\delta_0^3\varepsilon^{-s_c}\Big( t\varepsilon^4\big\|\Delta^2 f\big\|_{L^2_x}+1\Big).
%\end{align*}
%Therefore,  if $ t\varepsilon^4\big\|\Delta^2 f\big\|_{L^2_x}\ge \delta_0$, then the estimate above reduces to 
%\begin{align}\label{est:Rj1-2}
%\big\|\mathcal R_j^1\big\|_{L^2_x}
%\le C\delta_0^2\varepsilon^{4-s_c}t\Big(\big\|\Delta^2 f\big\|_{L^2_x}+1\Big).
%\end{align}
%This combining with  \eqref{est:Rj1-1} implies that \eqref{est:Rj1-2} is valid for any $t\in I(f)$. 

Collecting  the estimates on three terms in \eqref{Rj-123} and choosing $\delta_0$ suitably small, we have that the desired estimates.
This finishes the proof of the lemma.
\end{proof}

\subsubsection{Regular case} \label{sec:Conter-Regular case}
In this case, we set $r_0=0$. 
Let  
$$
f(x)\triangleq A_0^d \fe^{-A_0|x|^2},
$$
where $A_0$ is a suitable large constant determined later. 
Then $f\in \mathcal S(\R^d)$. 
%Moreover, for any $\gamma\ge 0$,
%$$
%\big\||\nabla|^\gamma f\big\|_{L^2_x}\sim L_0^\gamma.
%$$
It implies that 
$$
I(f) =\{t\in \R^+: t\varepsilon^4\lesssim 1\}.
$$ 
Then  \eqref{R1-conter-ubd} implies that for any $t\in I(f)$, 
\begin{align*}
\big\|\mathcal R_1(t)\big\|_{L^2_x}
\le &C\delta_0A_0^2\varepsilon^{2-s_c} 
+C\delta_0(1+\delta_0^2t)A_0^4\varepsilon^{4-s_c} ,
\end{align*}
and \eqref{R2-conter-lbd} implies that there exist some constants $C_0>0,C>0$ such that  for any $t\in I(f)$,
\begin{align*} 
\big\|\mathcal R_2(t)\big\|_{L^2_x}
\ge & C_0\delta_0A_0^4t\varepsilon^{4-s_c}  
-C\delta_0A_0^6t\varepsilon^{6-s_c}   - C\delta_0A_0^8 t^2\varepsilon^{8-s_c}  
 -C\delta_0A_0^2\varepsilon^{2-s_c}.
\end{align*} 
We shrink the time interval and denote that  
$$
\tilde I(f) =\{t\in \R^+: t\varepsilon^4\le \delta_0, t\varepsilon^2\ge 1\}.
$$
Then choosing $\delta_0$ suitably small, we have that  there exists some $\varepsilon_0>0$ such that for any $\varepsilon \in (0,\varepsilon_0]$ and any $t\in \tilde I(f)$, 
\begin{align*}
\big\|\mathcal R_1(t)\big\|_{L^2_x}
\le &C\delta_0^3A_0^4t\varepsilon^{4-s_c}+C\delta_0A_0^2\varepsilon^{2-s_c},
\end{align*}
and 
\begin{align*} 
\big\|\mathcal R_2(t)\big\|_{L^2_x}
\ge & \frac12 C_0\delta_0A_0^4 t\varepsilon^{4-s_c}.
\end{align*} 
Applying \eqref{eq:R-R12} and choosing $A_0$ suitably large and then $\delta_0$ suitably small again, this infers that for any $t\in \tilde I(f)$, 
\begin{align*}
\big\| R(t)\big\|_{L^2_x}
\ge &  \frac14 C_0\delta_0A_0^4 t\varepsilon^{4-s_c}.
\end{align*}
Scaling back, this implies that for any 
$
t\in \R^+: 1\le t\le \delta_0\varepsilon^{-2}, 
$
\begin{align*}
\big\| r(t)\big\|_{L^2_x}
\ge &  \frac14 C_0\delta_0A_0^4 t\varepsilon^2.
\end{align*}
This  establishes the conclusion in regular case.

\subsubsection{Non-Regular case} 

In this case, we set $\widetilde{v}_0$ such that 
$$
\widehat{\widetilde{v}_0}(\xi)=\delta_0\langle \xi\rangle_2^{-\alpha-\frac d2}\big(\ln \langle \xi\rangle_2\big)^{-1}.
$$
%where $\beta>\frac12$ will be determined later. 
Here we denote $\langle x\rangle_2=\sqrt{|x|^2+2}$. Then $f\in H^\alpha(\R^d)$. Now we need the following elementary results:
\begin{equation}\label{f-high-low-value}
\begin{aligned}
\big\|f\big\|_{H^\alpha_x}  \sim & 1; \\
\big\|P_{\le N} f\big\|_{H^\gamma}\sim  & N^{-\alpha+\gamma} \big(\ln N\big)^{-1},\quad \mbox{for any } \gamma>\alpha; \\
\big\|P_{>  N} f\big\|_{H^\gamma}\le  & CN^{-\alpha+\gamma} \big(\ln N\big)^{-1},\quad \mbox{for any } \gamma<\alpha.
\end{aligned}
\end{equation}

Arguing similarly as in Section \ref{sec:Thm3-case1-2}, for fixing large $N$, we denote $\widetilde{v}_N $ to be the solution to  the following equation
\EQ{
	\left\{ \aligned
	&2i\partial_{t} \widetilde{v} - \De \widetilde{v}= - |\widetilde{v}|^2 \widetilde{v}, \\
	& \widetilde{v}(0,x) = P_{\le N}\widetilde{v}_0(x),
	\endaligned
	\right.
}
and denote $\widetilde{v}^N =\widetilde{v}-\widetilde{v}_N $. Then  similarly as \eqref{est:Thm3-w-case2-d2}, 
we have that 
\begin{align*}
\big\|\widetilde{v}^N \big\|_{L^\infty_tL^2_x(\R)}
\le C\big\|P_{>N}\widetilde{v}_0\big\|_{L^2_x},
\end{align*}
where the constant $C>0$ only depends on $\|\widetilde{v}_0\|_{H^1_x}$.

Write
$$
\widetilde{v}^\varepsilon(t)-\widetilde{v}(t)=\widetilde{v}^\varepsilon(t)-\widetilde{v}_N (t)-\widetilde{v}^N .
$$
To consider the first term, we denote 
$$
r=\widetilde{v}^\varepsilon(t)-\widetilde{v}_N (t),\quad 
R=\fe^{it} S_\varepsilon r_N,\quad
r_0=P_{>N}\widetilde v_0. 
%f=P_{\le N}\widetilde{v}_0.
$$
Note that $r$ obeys the equation \eqref{eq:kg-r-3} by replacing $v$ by $v_N$ and $r_1$ in this situation is defined by 
$$
r_1=-\frac{\varepsilon^2}{2i}(\Delta P_{\le N}\widetilde{v}_0-\mu |P_{\le N}v_0|^2P_{\le N} v_0).
$$
Again, denote   
$$
\mathcal R_j=\langle \nabla \rangle^{-1}\big(\partial_t\mp i \langle \nabla\rangle\big)R \quad \mbox{ for } \quad j=1,2.
$$
Then  from \eqref{R1-conter-ubd} and \eqref{f-high-low-value}, we have that 
for any $1\le \alpha\le 4$, 
\begin{align*}
\big\|\mathcal R_1(t)\big\|_{L^2_x}
\le& 
C\delta_0\varepsilon^{-s_c}\Big[\big\|P_{>N}\widetilde v_0\big\|_{L^2_x}+\delta_0^2 \varepsilon \|P_{>N}v_0\|_{H^1_x}\\
&\quad +\varepsilon^2\big\|\Delta P_{\le N}\widetilde{v}_0\big\|_{L^2_x} 
+(1+\delta_0^2t)\varepsilon^4\big\|\Delta^2 P_{\le N}\widetilde{v}_0\big\|_{L^2_x}\Big]\\
\le& 
C\delta_0\varepsilon^{-s_c}\Big[N^{-\alpha}\big(\ln N\big)^{-1}+\delta_0^2 \varepsilon \max\{1, N^{\alpha-1}\}\\
&\quad +\varepsilon^2\max\{1, N^{2-\alpha}\big(\ln N\big)^{-1}\}
+(1+\delta_0^2t)\varepsilon^4N^{4-\alpha}\big(\ln N\big)^{-1}\Big].
\end{align*}
Setting 
\begin{align}\label{def:N}
N\triangleq A_0 \big(t\varepsilon^4\big)^{-\frac14},\quad \mbox{ and } \quad t\varepsilon^2\ge 1,
\end{align}
where $A_0$ is a large constant determined later, then we note that choosing $\varepsilon_0=\varepsilon_0(A_0)$ suitably small, it holds $N\varepsilon^\frac12\le A_0$. This gives that 
for any $\varepsilon\in (0,\varepsilon_0]$, 
\begin{align*}
\big\|\mathcal R_1(t)\big\|_{L^2_x}
\le& 
C\delta_0\varepsilon^{-s_c}\Big[N^{-\alpha}+\delta_0^2t\varepsilon^4N^{4-\alpha}\Big]\big(\ln N\big)^{-1},
\end{align*}
and  it further infers that  
\begin{align}\label{R1-upper-nonreg}
\big\|\mathcal R_1(t)\big\|_{L^2_x}
\le& 
C\delta_0\varepsilon^{-s_c}\Big[A_0^{-\alpha} \big(t\varepsilon^4\big)^{\frac14\alpha}+\delta_0^2A_0^{4-\alpha}t \varepsilon^4\big(t\varepsilon^4\big)^{-\frac14(4-\alpha)}\Big]\big|\ln \big(t\varepsilon^4\big) \big|^{-1}\notag\\
\le& 
C\delta_0\varepsilon^{-s_c}\big(A_0^{-\alpha}+\delta_0^2A_0^{4-\alpha}\big)\big(t\varepsilon^4\big)^{\frac14\alpha}\big|\ln \big(t\varepsilon^4\big) \big|^{-1}.
\end{align}
Similarly, by \eqref{R2-conter-lbd}, we have that 
\begin{align*}
&\Big|\big\|\mathcal R_2(t)\big\|_{L^2_x}- \frac14\delta_0 \varepsilon^{-s_c}t\varepsilon^4  \big\|\Delta^2 P_{\le N}\widetilde{v}_0\big\|_{L^2_x} -\frac 1{64}\delta_0t^2 \varepsilon^{8-s_c}  \big\|\Delta^4 P_{\le N}\widetilde{v}_0\big\|_{L^2_x}\Big|\\
\le & C\delta_0 \varepsilon^{-s_c}\Big[t\varepsilon^6  \big\|\Delta^3 P_{\le N}\widetilde{v}_0\big\|_{L^2_x} 
+ t^2\varepsilon^{10}  \big\|\Delta^5 P_{\le N}\widetilde{v}_0\big\|_{L^2_x}
\notag\\
& \quad  +\big\|P_{>N}\widetilde v_0\big\|_{L^2_x}+\delta_0^2 \varepsilon \|P_{>N}v_0\|_{H^1_x}
 +\varepsilon^2\big\|\Delta P_{\le N}\widetilde{v}_0\big\|_{L^2_x} \Big]\\
 %\le & C\delta_0 \varepsilon^{-s_c}\Big[t\varepsilon^6  \big\|\Delta^3 P_{\le N}\widetilde{v}_0\big\|_{L^2_x} 
%+ t^2\varepsilon^{10}  \big\|\Delta^5 P_{\le N}\widetilde{v}_0\big\|_{L^2_x}
%+\big\|P_{>N}\widetilde v_0\big\|_{L^2_x}\Big]\\
 \le & C\delta_0 \varepsilon^{-s_c}\Big[N^{-\alpha}+t\varepsilon^6N^{6-\alpha}+ t^2\varepsilon^{10}N^{10-\alpha}\Big]\big(\ln N\big)^{-1}.
\end{align*}
By \eqref{f-high-low-value}, it further gives that there exist constants  $c_0>0, C_0>0$ such that 
\begin{align*}
 \big\|\mathcal R_2(t)\big\|_{L^2_x}
 \ge & c_0\delta_0 \varepsilon^{-s_c}\big(t\varepsilon^4N^4+t^2\varepsilon^8N^8\big)N^{-\alpha}\big(\ln N\big)^{-1}\\
 &\quad -C\delta_0 \varepsilon^{-s_c}\Big[N^{-\alpha}+t\varepsilon^6N^{6-\alpha}+ t^2\varepsilon^{10}N^{10-\alpha}\Big]\big(\ln N\big)^{-1}.
% \\
% \big\|\mathcal R_2(t)\big\|_{L^2_x}
%\le  & 
%C_0\delta_0 \varepsilon^{-s_c}\big(t\varepsilon^4N^4+t^2\varepsilon^8N^8\big)N^{-\alpha}\big(\ln N\big)^{-1}\\
% &\quad 
%+
%C\delta_0 \varepsilon^{-s_c}\Big[N^{-\alpha}+t\varepsilon^6N^{6-\alpha}+ t^2\varepsilon^{10}N^{10-\alpha}\Big]\big(\ln N\big)^{-1}.
\end{align*}
By \eqref{def:N}, it drives  that 
\begin{align*}
 \big\|\mathcal R_2(t)\big\|_{L^2_x}
 \ge & c_0A_0^{4-\alpha}\delta_0\varepsilon^{-s_c} \big(t\varepsilon^4\big)^{\frac14\alpha}\big|\ln \big(t\varepsilon^4\big) \big|^{-1}.
% \\
% \big\|\mathcal R_2(t)\big\|_{L^2_x}
%\le  & 
%C_0\delta_0 \varepsilon^{-s_c}t\varepsilon^4N^{4-\alpha}\big(\ln N\big)^{-1}\\
% &\quad 
%+
%C\delta_0 \varepsilon^{-s_c}\Big[N^{-\alpha}+t\varepsilon^4N^{6-\alpha}+ t^2\varepsilon^{10}N^{10-\alpha}\Big]\big(\ln N\big)^{-1}.
\end{align*}
Combining with the estimate \eqref{R1-upper-nonreg}, choosing $A_0$ suitably large  and arguing similarly as Section \ref{sec:Conter-Regular case},  we obtain that 
\begin{align*}
 \big\|R(t)\big\|_{L^2_x}
 \ge &\frac12 c_0A_0^{4-\alpha}\delta_0\varepsilon^{-s_c} \big(t\varepsilon^4\big)^{\frac14\alpha}\big|\ln \big(t\varepsilon^4\big) \big|^{-1}.
\end{align*}
Scaling back, this concludes that for any 
$
t\in \R^+: 1\le t\le \delta_0\varepsilon^{-2}, 
$
\begin{align*}
\big\| r(t)\big\|_{L^2_x}
\gtrsim &  \big(t\varepsilon^2\big)^{\frac14\alpha}\big|\ln \big(t\varepsilon^2\big) \big|^{-1}.
\end{align*}
This  establishes the conclusion in non-regular case.

   \vskip 1.5cm

\section{Schr\"odinger profile and Growth-in-time nonrelativistic limit}\label{sec:case1}
\vskip .5cm

 In this section, we shall give the proof of Theorem \ref{thm:main3-local-d2}.
%
%
%  \subsection{Proof of Theorem \ref{thm:main3-local-d2}}\label{sec:main3-local-d2}
%
Since the proof is much similar as the that of Theorem \ref{thm:main1-global}, we only give it briefly.

\begin{proof}[Proof of Theorem \ref{thm:main3-local-d2}]
We split it into the following two cases.

$\bullet$ Case 1: $d=2$. We write
\begin{align*}
r^\varepsilon(t)=&u^\varepsilon(t)-\fe^{\frac{it}{\varepsilon^2}}v(t)-\fe^{-\frac{it}{\varepsilon^2}}\bar v(t)\\
=&u^\varepsilon(t)-\fe^{\frac{it}{\varepsilon^2}}v^\varepsilon(t)-\fe^{-\frac{it}{\varepsilon^2}}\overline{v^\varepsilon}(t)
 +\fe^{\frac{it}{\varepsilon^2}}\big(v^\varepsilon(t)-v(t)\big)+\fe^{-\frac{it}{\varepsilon^2}}\big(\overline{v^\varepsilon}(t)-\bar v(t)\big),
\end{align*}
where $v^\varepsilon$ is the solution of \eqref{eq:nls-wave} with
$$
v_0^\varepsilon=v_0,\quad v_1^\varepsilon=v_1=0.
$$
Then by Theorem \ref{thm:main1-global}, we have that
\begin{align*}
\big\|u^\varepsilon(t)-\fe^{\frac{it}{\varepsilon^2}}v^{\varepsilon}(t)-\fe^{-\frac{it}{\varepsilon^2}}\bar v^{\varepsilon}(t) \big\|_{L^2_x}
\lesssim \varepsilon^2.
\end{align*}
Moreover, by Theorem \ref{thm:main3} and choosing $\alpha\ge2, \beta=2$, we have that
\begin{align*}
\big\|v^\varepsilon(t)-v(t)\big\|_{L^2_x}
\lesssim &\varepsilon^2+\big(\varepsilon^2 t\big)^{\frac14\alpha}.
\end{align*}
Combining with the estimates above,  we obtain \eqref{est:main3-local-d2}.

$\bullet$ Case 2: $d=3$.

Suppose that $v$ is the solution of \eqref{eq:nls}, then by \eqref{eq:r}-\eqref{eq:initial-data}, we obtain that
\EQn{
	\label{eq:kg-r-1}
	\left\{ \aligned
	&\varepsilon^2\partial_{tt} r^\varepsilon - \De r + \frac{1}{\varepsilon^2} r^\varepsilon +A_\varepsilon(v,r^\varepsilon)+B_\varepsilon(v)
	+\varepsilon^2\Big[\fe^{\frac{it}{\varepsilon^2}}\partial_{tt}v+\fe^{-\frac{it}{\varepsilon^2}}\partial_{tt}\bar v\Big]=0, \\
	& r^\varepsilon(0,x) = 0,\quad  \partial_tr^\varepsilon(0,x) = r_1(x),
	\endaligned
	\right.
}
where
\begin{align}\label{def:r1}
r_1\triangleq -2\re\big[\partial_tv(0)\big]=-2\im \big(\De v_0-|v
_0|^2v_0\big).
\end{align}

Now we use the scaling argument and denote
\begin{align}\label{scaling-phi-r}
R(t,x)=\mathcal{S}_\varepsilon r^\varepsilon(t, x),
\end{align}
then by \eqref{eq:kg-r-1}, $R$ obeys the following equation:
\begin{align}\label{eq:phi-scaling}
\partial_{tt} R - \De R + R +G\Big(\mathcal{S}_\varepsilon v,R\Big)+F_1(\mathcal{S}_\varepsilon v)+\varepsilon^4 F_2\big(\mathcal{S}_\varepsilon (v_{tt})\big)=0,
\end{align}
where  we have used the notations $G,F_1$ and $f^\varepsilon$  as in Section \ref{sec:global}, and additionally we denote
 \begin{align*}
F_2(f)=&\fe^{it}f+\fe^{-it}\bar f.
 \end{align*}
Moreover, the initial data of $R$ is that
\begin{align*}
R(0)=0,\quad \partial_{t} R(0)=\varepsilon^2 \mathcal{S}_\varepsilon r_1(x).
\end{align*}

Denote
\begin{align*}
\varphi\triangleq  \langle \nabla\rangle^{-1}\big(\partial_t +i\langle \nabla\rangle \big)R,
\end{align*}
then we have that
\begin{align*}
R= \im\> \varphi;\quad  \partial_t R= \re\big(\langle \nabla\rangle \varphi\big).
\end{align*}
Moreover, $\varphi$ obeys the following equation
\begin{align*}
\big(\partial_t -i\langle \nabla\rangle \big) \varphi +\langle \nabla\rangle^{-1}\Big[G\Big(\mathcal{S}_\varepsilon v ,R\Big)+F_1(v^\varepsilon)+\varepsilon^4 F_2\big((v_{tt})^\varepsilon)\Big]=0.
\end{align*}
with the initial data
$$
\varphi(0)=\varphi_0\triangleq  \varepsilon^2\langle \nabla\rangle^{-1}\big(\mathcal{S}_\varepsilon r_1\big).
$$
By Duhamel's formula, we have that
\begin{align*}
\varphi(t)= \fe^{it\langle\nabla\rangle}\varphi_0+\int_0^t \fe^{i(t-s)\langle\nabla\rangle} \langle \nabla\rangle^{-1}\Big[G\Big(\mathcal{S}_\varepsilon v ,R\Big)+F_1\big(\mathcal{S}_\varepsilon v\big)+\varepsilon^4 F_2\big((v_{tt})^\varepsilon)\Big]\,ds.
\end{align*}
Treating similarly as in \eqref{est:Duh-mod}, we obtain that for any $I=[t_0,t_1]\subset \R$,
\begin{align*}
\varphi(t)=& \fe^{it\langle\nabla\rangle}\varphi(t_0)
\notag\\
&\quad +\int_{t_0}^t \fe^{i(t-s)\langle\nabla\rangle} \langle \nabla\rangle^{-1}\Big[G\Big(\mathcal{S}_\varepsilon v ,R\Big)+P_{> 1}F_1\big(\mathcal{S}_\varepsilon v\big)+\varepsilon^4 F_2\big((v_{tt})^\varepsilon\big)\Big]\,ds\\
&\quad -   \frac{\fe^{3it} }{i(\langle\nabla\rangle-3)}\langle \nabla\rangle^{-1}P_{\le 1} \big(\mathcal{S}_\varepsilon v(t)\big)^3
+  \frac{\fe^{it\langle \nabla \rangle}}{i(\langle\nabla\rangle-3)}\langle \nabla\rangle^{-1}P_{\le 1} \big(\mathcal{S}_\varepsilon v(t_0)\big)^3\\
&\quad -   \frac{\fe^{3it} }{i(\langle\nabla\rangle+3)}\langle \nabla\rangle^{-1}P_{\le 1} \Big(\overline{\mathcal{S}_\varepsilon v(t)}\Big)^3
+  \frac{\fe^{it\langle \nabla \rangle}}{i(\langle\nabla\rangle+3)}\langle \nabla\rangle^{-1}P_{\le 1} \Big(\overline{\mathcal{S}_\varepsilon v(t_0)}\Big)^3\\
&\quad+3 \int_{t_0}^t \fe^{i(t-s)(\langle\nabla\rangle-3)} \frac{1}{i(\langle\nabla\rangle-3)} \langle \nabla\rangle^{-1}P_{\le 1} \Big[\big(\mathcal{S}_\varepsilon v\big)^2\> \partial_s\big(\mathcal{S}_\varepsilon v\big)\Big]\,ds\\
&\quad+3 \int_{t_0}^t \fe^{i(t-s)(\langle\nabla\rangle+3)} \frac{1}{i(\langle\nabla\rangle+3)} \langle \nabla\rangle^{-1}P_{\le 1} \Big[\big(\overline{\mathcal{S}_\varepsilon v})^2\> \partial_s\big(\overline{\mathcal{S}_\varepsilon v}\big)\Big]\,ds.
\end{align*}

First, we assume that $v_0$ is smooth, then treating similarly as in Sections \ref{sec:nonlinear} and \ref{sec:Thm3-case1-1}, and applying Lemma \ref{lem:spacetime-norm-NLS} instead, we have that for any $\gamma\in [-\frac12,0]$,
\begin{align}\label{est:varphi-local-d3}
\big\||\nabla|^\gamma \varphi\big\|_{Y_\frac12(I)}
\lesssim &
\big\||\nabla|^\gamma\langle \nabla\rangle^{\frac12}\varphi(t_0)\big\|_{H^\frac12}
+\big\|v\big\|_{S_\frac12(\varepsilon^2I)}^2\big\||\nabla|^\gamma\varphi\big\|_{S_\frac12(I)}
\notag\\
&
+\big\|\varphi\big\|_{S_\frac12(\varepsilon^2I)}\big\||\nabla|^\gamma\varphi\big\|_{S_\frac12(I)}
+\big\|\varphi\big\|_{S_\frac12(\varepsilon^2I)}^2\big\||\nabla|^\gamma\varphi\big\|_{S_\frac12(I)}
\notag\\
&
+\big\|\mathcal{S}_\varepsilon v\big\|_{S_\frac12(I)}^2\big\||\nabla|^\gamma\partial_s\big(\mathcal{S}_\varepsilon v\big)\big\|_{S_\frac12(I)}
+\varepsilon^{4+\gamma}|I| \big\|\Delta^2 v_0\big\|_{\dot H^{\gamma+\frac12}}
+\varepsilon^{2+\gamma}.
\end{align}
By Sobolev's inequality and Lemma \ref{lem:spacetime-norm-NLS}, we have that
\begin{align*}
\big\|\mathcal{S}_\varepsilon v\big\|_{S_\frac12(I)}
\le &\big\|v\big\|_{S_\frac12(\R)}
\le  C(\|v_0\|_{H^1_x});\\
\big\|\mathcal{S}_\varepsilon v\big\|_{S_\frac12(I)}
\lesssim & |I|^\frac{3}{10} \big\||\nabla|^\frac35 \mathcal{S}_\varepsilon v\big\|_{L^\infty_t\dot H^\frac12_x(I)}
\\
\le &  |I|^\frac{3}{10} \varepsilon^\frac35   C(\|v_0\|_{H^2});\\
\big\||\nabla|^\gamma\partial_s\big(\mathcal{S}_\varepsilon v\big)\big\|_{S_\frac12(I)}
\le &\varepsilon^{2+\gamma}  C(\|v_0\|_{H^1_x}) \>   \big\||\nabla|^{\gamma+\frac52}v_0\big\|_{L^2_x}.
\end{align*}
Inserting these estimates into \eqref{est:varphi-local-d3}, we obtain that
\begin{align*}
\big\||\nabla|^\gamma \varphi\big\|_{Y_\frac12(I)}
\lesssim &
\big\||\nabla|^\gamma\langle \nabla\rangle^{\frac12}\varphi(t_0)\big\|_{\dot H^\frac12}
+\big\||\nabla|^\gamma v\big\|_{S_\frac12(\varepsilon^2I)}^2\big\||\nabla|^\gamma\varphi\big\|_{S_\frac12(I)}
\\
&+\big\||\nabla|^\gamma\varphi\big\|_{S_\frac12(I)}^2
+\big\||\nabla|^\gamma\varphi\big\|_{S_\frac12(I)}^3
\\
&+\varepsilon^{2+\gamma}\Big(\min\big\{1, |I|^\frac{3}{10} \varepsilon^\frac35\big\}\big\||\nabla|^{\gamma+\frac52}v_0\big\|_{L^2_x}
+\varepsilon^2|I| \big\|\Delta^2 v_0\big\|_{\dot H^{\gamma+\frac12}}+ 1\Big),
\end{align*}
where the implicit constant is only dependent of $\|v_0\|_{H^2}$.
Repeating the same process at the end of Section  \ref{sec:Thm3-case1-1}, we have that if
\begin{align}\label{cond-T-3d-local}
\varepsilon^\frac35T^\frac{3}{10}  \big\||\nabla|^{\frac52}v_0\big\|_{L^2_x}\le \delta_0;\quad
\varepsilon^2T \big\|\Delta^2 v_0\big\|_{\dot H^{\frac12}}\le \delta_0,
\end{align}
it holds that for some constant  $C=C\big(\|v_0\|_{H^2}\big)>0$,
\begin{align*}
\big\|r^\varepsilon\big\|_{L^\infty_tL^2_x([0,T])}
\le C\Big( \varepsilon^2
+\varepsilon^2T \big\|\Delta^2 v_0\big\|_{L^2_x}\Big).
\end{align*}

If $v_0\in H^\alpha$ for some $\alpha\in [2,4]$, then arguing similarly as in Section \ref{sec:Thm3-case1-2}, we have that for any $T$ verifying
\begin{align}\label{cond-T-3d-local-2}
\varepsilon^{\frac35}T^\frac{3}{10} N^\frac12 \le   C\big(\big\|v_0\big\|_{H^2}\big) \delta_0;\quad
\varepsilon^2T N^\frac52 \le C\big(\big\|v_0\big\|_{H^2}\big) \delta_0
\end{align}
(which is from \eqref{cond-T-3d-local}), then
\begin{align*}
\big\|r^\varepsilon\big\|_{L^\infty_tL^2_x([0,T])}
\le C\Big( \varepsilon^2
+\varepsilon^2TN^{4-\alpha}+N^{-\alpha}\Big).
\end{align*}
Choosing
$$
N=\big(\varepsilon^2T\big)^{-\frac14},
$$
then \eqref{cond-T-3d-local-2} holds if $\varepsilon^2T\le \delta_0$. Thus we prove \eqref{est:main3-local-d2} when $\varepsilon^2T\le \delta_0$.

Note that we have the energy conservation law:
\begin{align*}
\int\big(\varepsilon^2|\partial_tu^\varepsilon|^2 & + |\nabla u^\varepsilon|^2  + \varepsilon^{-2}|u^\varepsilon|^2+ \frac{3}{2}|u^\varepsilon|^4\big)dx\\
= & \int\big(\varepsilon^{-2}\big\|(u_0,u_1)\big\|_{L^2\times L^2}^2+ |\nabla u_0|^2  + \frac{3}{2}|u_0|^4\big)dx.
\end{align*}
This implies that
\begin{align*}
\big\|u^\varepsilon\big\|_{L^\infty_tL^2_x(\R)}
\le C\big(\big\|(u_0,u_1)\big\|_{H^1_x\times L^2_x}\big).
\end{align*}
This yields \eqref{est:main3-local-d2} when $\varepsilon^2T\ge \delta_0$. This finishes the proof of Theorem \ref{thm:main3-local-d2}.
\end{proof}

%\section{The sharpness of the regularity: local nonrelativistic limit}\label{sec:sharp-local}

 \vskip 1.5cm

\bigskip
\section*{Acknowledgment}

We would like to thank Professor Weizhu Bao for bringing the problem to us and for useful discussions. The authors were in part supported by NSFC (Grant No. 11725102, 12171356), Sino-German
Center Mobility Programme (Project ID/GZ M-0548) and Shanghai Science and Technology
Program (Project No. 21JC1400600 and No. 19JC1420101).

\bibliographystyle{model1-num-names}

\end{document}